\documentclass{amsart}
\usepackage{graphicx}
\usepackage{amssymb}
\usepackage{pict2e}
\usepackage{array}
\usepackage{color}

\makeatletter
 \def\@textbottom{\vskip \z@ \@plus 12pt}
 \let\@texttop\relax
\makeatother

\newtheorem{thm}{Theorem}
\newtheorem{lem}[thm]{Lemma}
\newtheorem{ex}[thm]{Example}

\newtheorem{defi}[thm]{Definition}
\newtheorem{algo}[thm]{Algorithm}
\newtheorem{prop}[thm]{Proposition}
\newtheorem{rk}[thm]{Remark}
\newtheorem{nota}[thm]{Notation}
\newtheorem{sett}[thm]{Setting}
\newtheorem{ass}[thm]{Assumption}

\newcommand{\rr}{{\mathbb{R}}}
\newcommand{\nn}{{\mathbb{N}}}
\newcommand{\TT}{{\mathbb{T}}}
\newcommand{\CC}{{\mathbb{C}}}
\newcommand{\HH}{{\mathbb{H}}}
\newcommand{\KK}{{\mathbb{K}}}
\newcommand{\DD}{{\mathbb{D}}}
\newcommand{\E}{\mathbb{E}}

\newcommand{\indiq}{{{\mathbf 1}}}
\newcommand{\bx}{{{\mathbf x}}}

\newcommand{\by}{{{\mathbf y}}}
\newcommand{\btx}{{\tilde {\mathbf x}}}
\newcommand{\tx}{{\tilde x}}
\newcommand{\cT}{{\mathcal T}}
\newcommand{\cC}{{\mathcal C}}
\newcommand{\cH}{{\mathcal H}}

\newcommand{\cm}{{\mathfrak m}}
\newcommand{\cs}{{\mathfrak s}}
\newcommand{\cU}{{\mathcal U}}

\newcommand{\cF}{{\mathcal F}}
\newcommand{\cG}{{\mathcal G}}

\newcommand{\cL}{{\mathcal L}}
\newcommand{\cS}{{\mathcal S}}
\newcommand{\cD}{{\mathcal D}}
\newcommand{\cK}{{\mathcal K}}
\newcommand{\cP}{{\mathcal P}}
\newcommand{\tN}{{\tilde N}}
\newcommand{\tDelta}{{\tilde \Delta}}
\newcommand{\tZ}{{\tilde Z}}

\newcommand{\bN}{{\bar N}}
\newcommand{\tR}{{\tilde R}}
\newcommand{\ts}{{\tilde s}}

\newcommand{\ctF}{\tilde {\mathcal F}}

\newcolumntype{C}[1]{>{\centering\let\newline\\\arraybackslash\hspace{0pt}}m{#1}}
\newcommand{\btp}{\begin{center}\begin{tabular}{|C{5.5cm}|C{5.5cm}|}}
\newcommand{\etp}{\end{tabular}\end{center}}

\begin{document}

\title[Monte-Carlo tree search]{On Monte-Carlo tree search for deterministic
games with alternate moves and complete information}

\author{Sylvain Delattre}

\author{Nicolas Fournier}

\address{Sylvain Delattre, Laboratoire de Probabilit\'es et Mod\`eles Al\'eatoires, UMR 7599, 
Universit\'e Paris Diderot, Case courrier 7012, 75205 Paris Cedex 13, France.}

\email{sylvain.delattre@univ-paris-diderot.fr}

\address{Nicolas Fournier, Laboratoire de Probabilit\'es et Mod\`eles Al\'eatoires, UMR 7599,
Universit\'e Pierre-et-Marie Curie, Case 188, 4 place Jussieu, F-75252 Paris Cedex 5, France.}

\email{nicolas.fournier@upmc.fr}

\begin{abstract}
We consider a deterministic game with alternate moves and complete information, of which the
issue is always the victory of one of the two opponents.
We assume that this game is the realization of a random model enjoying some independence properties.
We consider algorithms in the spirit 
of Monte-Carlo Tree Search, 
to estimate at best the minimax value of a given position: it consists in simulating,
successively, $n$ well-chosen matches, starting from this position.
We build an algorithm, which is optimal, 
step by step, in some sense: once the $n$ first matches are simulated, the algorithm decides
from the statistics furnished by the $n$ first matches (and the {\it a priori} we have on the game)
how to simulate the $(n+1)$-th match
in such a way that the increase of information concerning the minimax value of the position under study
is maximal. This algorithm is remarkably quick.
We prove that our step by step optimal algorithm is not globally optimal
and that it always converges in a finite number of steps, 
even if the {\it a priori} we have on the game is completely irrelevant.
We finally test our algorithm, against MCTS, on Pearl's game \cite{p} and, 
with a very simple and universal {\it a priori}, 
on the game Connect Four and some variants. 
The numerical results are rather disappointing. We however
exhibit some situations in which our algorithm seems efficient.
\end{abstract}



\thanks{We warmly thank Bruno Scherrer for his invaluable help. 
We also thank the anonymous referee for his numerous fruitful comments.}

\maketitle

\section{Introduction}\label{intro}

\subsection{Monte-Carlo Tree Search algorithms} 
Monte-Carlo Tree Search (MCTS) are popular algorithms for heuristic search in two-player games.
Let us mention the  
book of Munos \cite{mlivre} and the survey paper of Browne {\it et al.} \cite{mctsurvey},
which we tried to briefly summarize here and to which we refer for a much more
complete introduction on the topic.

We consider a deterministic game with complete information and alternate moves involving two players,
that we call $J_1$ and $J_0$. We think of Go, Hex, Connect Four, etc. 
Such a game can always be represented as a discrete tree, of which the nodes
are the positions of the game. Indeed, even if a single position can be thought as the child of two different
positions, we can always reduce to this case by including the whole history of the game
in the position. Also, we assume that the only possible outcomes of the game, which are
represented by values on the leaves of the tree, are either
$1$ (if $J_1$ wins) or $0$ (if $J_0$ wins). If draw is a possible outcome, we e.g. identify it
to a victory of $J_0$.

Let $r$ be a configuration in which $J_1$ has to choose between several moves.
The problem we deal with is: how to select at best one of these moves with a computer and a given amount of time.

If having a huge amount of time, such a question can classically be completely solved by computing recursively,
starting from the leaves, the minimax 
values $(R(x))_{x \in\cT}$, see Remark \ref{mm}. Here $\cT$ is the tree (with root $r$ and set of leaves $\cL$) 
representing the game when starting from $r$, and for each $x\in\cT$, $R(x)$ is the value of the game starting 
from $x$.
In other words, $R(x)=1$ if $J_1$
has a winning strategy when starting from $x$ and $0$ else. So we compute $R(x)$ for all the children
$x$ of $r$ and choose a move leading to some $x$ such that $R(x)=1$, if such a child exists.

In practice, this is not feasible, except if the game (starting from $r$) is very small.
One possibility is to cut the tree at some reasonable depth $K$, to assign some estimated values
to all positions of depth $K$, and to compute the resulting (approximate) minimax values 
on the subtree above depth $K$.
For example if playing Connect Four, one can assign to a position the value 
\emph{remaining number of possible alignments for $J_1$ minus remaining number of possible alignments for $J_0$.}
Of course, the choice of such a value is highly debatable, 
and heavily depends on the game.

A more universal possibility, introduced by Abramson \cite{a},
is to use some Monte-Carlo simulations:
from each position with depth $K$, we handle a certain number $N$ of uniformly random matches
(or matches with a simple \emph{default policy}),
and we evaluate this position by the number of these matches that led to a victory of $J_1$ divided by $N$.
Such a procedure is now called {\it Flat MCTS}, see Coquelin and Munos \cite{cm}, Browne {\it et al.} 
\cite{mctsurvey},
see also Ginsberg \cite{g} and Sheppard \cite{s}.

Coulom \cite{c} introduced the class of MCTS algorithms.  Here are the main ideas: we have a {\it default policy}
and a {\it selection procedure}. We make $J_1$ play against $J_0$ a certain number of times
and make grow a subtree of the game. Initially, the subtree $\cT_0$ consists of the root and its children. 
After $n$ steps, we have the subtree $\cT_n$ and some statistics $(C(x),W(x))_{x \in\cT_n}$
provided by the previous matches: $C(x)$ is the number of times the node $x$ has been crossed and $W(x)$ 
is the number of times this has led to a victory of $J_1$. Then we select a leave $y$ of $\cT_n$ 
using the selection procedure (which relies on the statistics $(C(x),W(x))_{x \in\cT_n}$) and we end
the match (from $y$) by using the default policy. We then build $\cT_{n+1}$ by adding to $\cT_n$
the children of $y$, and we increment the values of $(C(x),W(x))_{x \in\cT_{n+1}}$ according to the
issue of the match (actually, it suffices
to compute these values for $x$ in the branch from $r$ to $y$ and for the children of $y$).
Once the given amount of time is elapsed, we choose the move leading to the 
child $x$ of $r$ with the highest $W(x)/C(x)$.

Actually, this procedure throws away a lot of data: $C(x)$ is not exactly the number of times $x$
has been crossed, it is rather the number of times it has been crossed since $x \in \cT_n$,
and a similar fact holds for $W(x)$. In practice, this limits the memory used by the algorithm.

The most simple and universal default policy is to choose each move at uniform random and this is the case we will
study in the present paper. It is of course more efficient to use a simplified
strategy, depending on the game under study, but this is another topic.

Another important problem is to decide how to select the leave $y$ of $\cT_n$. 
Kocsis and Szepesv\'ari \cite{ks} proposed to use some bandit ideas, developed
(and shown to be optimal, in a very weak sense, for bandit problems) by Auer {\it et al.} \cite{acbf},
see Bubeck and Cesa-Bianchi \cite{bcb} for a survey paper. 
They introduced a version of MCTS
called UCT (for UCB for trees, UCB meaning Upper Confidence Bounds), in which the selection procedure 
is as follows.
We start from the root $r$ and we go down in $\cT_n$: when in position $x$ where $J_1$ (resp. $J_0$) 
has to play, we choose the child $z$ of $x$ maximizing 
$W(z)/C(z)+\sqrt{c (\log n) / C(z)}$ (resp. $(C(z)-W(z))/C(z)+\sqrt{c (\log n) / C(z)}$).
At some time we arrive at some leave $y$ of $\cT_n$, this is the selected leave.
Here $c>0$ is a constant to be chosen empirically.
Kocsis and Szepesv\'ari \cite{ks} proved the convergence of UCT.

Chaslot {\it et al.} \cite{cwhetc} have broaden the framework of MCTS. Also, they proposed different
ways to select the best child (after all the computations): either the one with the highest $W/C$,
the one with the highest $C$, or something intermediate.

Gelly {\it et al.} \cite{mogo} experimented MCTS (UCT) on Go. They built the program MoGo, 
which also uses some pruning procedures, and obtained impressive results on a $9\times 9$ board.

The early paper of Coquelin and Munos \cite{cm} contains many results.
They showed that UCT can be inefficient on some particular trees and proposed a modification
taking into account some possible smoothness of the tree and outcomes (in some sense).
They also studied Flat MCTS.

Lee {\it et al.} \cite{teytetal} studied the problem of fitting precisely the parameters of the selection
process. Of course, $W/C$ means nothing if $C=0$, so they empirically  investigated what happens when
using $(W+a)/(C+b)+\sqrt{c (\log n) / (C+1)}$, for some constants $a,b,c>0$. 
They conclude that $c=0$ is often the best choice.
This is not so surprising, since the logarithmic term is here
to prevent from large deviation events, which do asymptotically not exist for deterministic games.
MoGo uses $c=0$ and \emph{ad hoc} constants $a$ and $b$.
This version of MCTS is the one presented in Section \ref{smcts} and the one we used to test our algorithm.

Let us mention the more recent theoretical work by Bu\c soniu, Munos and P\'all \cite{bmp},
as well as the paper
of Garivier, Kaufmann and Koolen \cite{gkk} who study in details a bandit model for a two-round
two-player random game.

The survey paper of Browne {\it et al.} \cite{mctsurvey} discusses many tricks to improve the numerical
results and, of course, all this has been adapted with very special and accurate procedures to
particular games. As everybody knows, AlphaGo \cite{alphago}
became the first Go program to beat a human professional Go player on a full-sized board.
Of course, AlphaGo is far from using only MCTS, it also relies on deep-learning and many other things.

\subsection{Our goal} We would like to study MCTS when using a probabilistic model for the game.
To simplify the problem as much as possible,
we only consider the case where the default policy is \emph{play at uniform random}, 
and we consider the modified version of MCTS described in Section \ref{smcts}, where we keep all the information. 
This may cause memory problems in practice, but we do not discuss such difficulties. 
So the modified version is as follows, for some constants $a,b>0$ to be fitted empirically. 
We make $J_1$ play against $J_0$ a certain number of times
and make a subtree of the game grow. Initially, the subtree $\cT_0$ consists in the root $r$ and its children. 
After $n$ steps, we have the subtree $\cT_n$ and some statistics $(C(x),W(x))_{x \in\cT_n}$ 
provided by the previous matches: $C(x)$ is the number of times the node $x$ has been crossed and $W(x)$ 
the number of times this has led to a victory of $J_1$. The $(n+1)$-th step is as follows: start from $r$
and go down in $\cT_n$ by following the highest values of 
$(W+a)/(C+b)$ (resp. $(C-W+a)/(C+b)$) if it is $J_1$'s turn to play (resp. $J_0$'s turn to play),
until we arrive at some leave $z$ of $\cT_n$. From there, complete the match at uniform random until
we reach a leave $y$ of $\cT$. We then build $\cT_{n+1}$ by adding to $\cT_n$ the whole branch from $z$ to $y$
together with all the brothers of the elements of this branch, and we compute the values of 
$(C(x),W(x))_{x \in\cT_{n+1}}$ (actually, it suffices
to compute these values for $x$ in the branch from $r$ to $y$).
Once the given amount of time is elapsed, we choose the move leading to the 
child $x$ of $r$ with the highest $C(x)/W(x)$.

As shown by Coquelin and Munos \cite{cm}, one can build games for which MCTS is not very efficient.
So it would be interesting to know for which class of games it is. This seems to be a very difficult problem.
One possibility is to study if MCTS works well {\it in mean}, i.e. when the game is chosen at random.
In other words, we assume that the tree and the outcomes are the realization of a random model.
We use a simple toy model enjoying some independance properties,
which is far from convincing if modeling true games but for which we can handle a complete 
theoretical study.

We consider a class of algorithms resembling the above mentioned 
modified version of MCTS. After $n$ simulated matches, we have some information $\cF_n$ on the game:
we have visited $n$ leaves, we know the outcomes of the game at these $n$ leaves, and we have the explored
tree $\cT_n=B_n\cup D_n$, where $B_n$ is the set of all crossed positions, and $D_n$ is the boundary
of the explored tree (roughly, $D_n$ consists of uncrossed positions of which the father belongs to $B_n$).

So we can approximate $R(r)$, which is the quantity of interest, by $\E[R(r)|\cF_n]$ (if the latter
can be computed).
Using only this information $\cF_n$ (and possibly some {\it a priori} on the game furnished by the model), 
how to select $z \in D_n$ so that, simulating a uniformly random match starting from $z$
and updating the information, $\E[R(r)|\cF_{n+1}]$ is as close as possible (in $L^2$) to $R(r)$?

We need a few assumptions. In words, (a) the tree and outcomes
enjoy some independence properties, (b) we can compute, at least numerically, for 
$x \in D_n$, $m(x)=$ {\it mean value of $R(x)$} and
$s(x)=$ {\it mean quantity of information that a uniformly random match starting from $x$ will provide.}
See Subsection \ref{trq} for precise definitions.

Under such conditions, we show that $\E[R(r)|\cF_n]$ can indeed be computed (numerically),
and $z$ can be selected as desired. The procedure resembles in spirit MCTS, but is of course more complicated
and requires more computations. 

This is extremely surprising: the computational cost to find the best $z \in D_n$
does not increase with $n$, because this choice requires to compute some values of which the update
does not concern the whole tree $\cT_n$, but only the last visited branch, as MCTS. (Actually, we also
need to update all the brothers of the last visited branch, but this remains rather reasonable).
It seems miraculous that this {\it theoretical} algorithm behaves so well.
Any modification, such as taking draws into account or changing the notion
of optimality, seems to lead to drastically more expensive algorithms, that require to update some values
on the whole visited tree $\cT_n$.
We believe that this is the most interesting fact of the paper.

The resulting algorithm is explained in details in the next section.
We will prove that this algorithm is convergent 
(in a finite number of steps) even if
the model is completely irrelevant. This is not very surprising, since all the leaves
of the game are visited after a finite number of steps. 
We will also prove on an example that our algorithm is {\it myopic}: it is not {\it globally} optimal.
There is a theory showing that for a class of problems, a step by step optimal
algorithm is {\it almost} globally optimal, i.e. up to some reasonable factor, see
Golovin and Kraus \cite{gk}. However, it is unclear whether this class of problems includes ours.

\subsection{Choice of the parameters}
We will show, on different classes of models, how to compute the functions $m$ and $s$ required 
to implement our algorithm. We studied essentially two possibilities.

In the first one, we assume that the tree $\cT$ is the realization of an inhomogeneous Galton-Watson tree,
with known reproduction laws, and that the outcomes of the game are the realizations of 
i.i.d. Bernoulli random variables of which the parameters depend only on the depths of the involved 
(terminal) positions. Then $m(x)$ and $s(x)$ depend only on the depth of the node $x$.
We can compute them numerically once for all, using rough statistics we have on the {\it true game} we want
to play (e.g. Connect Four), by handling a high number of uniformly random matches.

The second possibility is much more universal and adaptive and works better in practice. 
At the beginning, we prescribe that $m(r)=a$,
for some fixed $a \in (0,1)$.
Then each time a new litter $\{y_1,\dots,y_d\}$ (with father $x$) is created by the algorithm, 
we set $m(y_1)=\dots=m(y_d)=m(x)^{1/d}$ if $x$ is a position where it is $J_0$'s turn to play, and
$m(y_1)=\dots=m(y_d)=1-(1-m(x))^{1/d}$ else. Observe here that $d$ is discovered by the algorithm 
at the same time as the new litter.
Concerning $s$, we have different possibilities
(see Subsection \ref{pchoice} and Section \ref{comput}),
more or less justified from a theoretical point of view, among which 
$s(x)=1$ for all $x$ does not seem to be too bad.

The first possibility seems more realistic but works less well than the second one in practice
and requires some preliminary fitting.
The second possibility relies on a symmetry modeling consideration: assuming that all the individuals of the
new litter behave similarly necessarily leads to such a function $m$.  
This is much more universal in that once the value of $a=m(r)$ is fixed
(actually, $a=0.5$ does not seem to be worse than another value), everything can
be computed in a way not depending on the true game. Of course, the choice of $a=m(r)$
is debatable, but does not seem to be very important in practice.

\subsection{Comments and a few more references}

Our class of models generalize a lot the Pearl model \cite{p}, which simply consists of a regular
tree with degree $d$, depth $K$, with i.i.d. Bernoulli$(p)$ outcomes on the leaves.
However, we still assume a lot of independence. This is not 
fully realistic and is probably the main reason why our numerical experiments are rather disappointing.

The model proposed by Devroye-Kamoun \cite{dk} seems much more relevant,
as they introduce correlations between the outcomes as follows. 
They consider that each edge of the tree has a value
(e.g. Gaussian). The value of the leave is then $1$ if the sum of
the values along the corresponding branch is positive, and $0$ else. This more or less models that 
a player builds, little by little, his game. But from a theoretical point of view,
it is very difficult to study, because we only observe the values of the leaves, not of the edges.
So this unfortunately falls completely out of our scope. 
See also Coquelin and Munos \cite{cm} for a notion of smoothness of the tree, which seems rather
relevant from a modeling point of view.

Finally, our approach is often called {\it Bayesian}, because we have an {\it a priori} law for
the game. This has already been studied in the artificial intelligence literature.
See Baum and Smith \cite{bs} and Tesauro, Rajan and Segal \cite{trs}. Both 
introduce some conditional expectations of the minimax values
and formulas resembling \eqref{rnrec} already appear.

\subsection{Pruning}

Our algorithm automatically proceeds to some pruning, as AlphaBeta, which is a clever version to 
compute exactly the minimax values of a (small) game. 
This can be understood if reading the proof of Proposition \ref{toujconv}.
The basic idea is as follows: if we know from the information we have that
$R(x)=0$ for some internal node $x \in \cT$ and if the father $v$ of $x$ is a position where $J_0$ plays,
then there is no need to study the brothers of $x$, because $R(v)=0$.
And indeed, our algorithm will never visit these brothers. 

Some versions of MCTS with some additional pruning procedures have already been
introduced empirically. See  Gelly {\it et al.} \cite{mogo} for MoGo, as well
as many other references in \cite{mctsurvey}. It seems rather nice that our algorithm
automatically prunes and this holds even if the {\it a priori} we have on the game is completely irrelevant.
Of course, if playing a large game, this pruning will occur only near the leaves.

\subsection{Numerical experiments}

We have tested our {\it general} algorithm against some {\it general} versions
of MCTS. We would not have the least chance if using some versions of MCTS modified in
such a way that it takes into account some symmetries of a particular game, with a more
clever default policy, etc. But we hope our algorithm might also be adapted to
particular games.

Next, let us mention that our algorithm is subjected to some numerical problems, in that we have
to compute many products, that often lead numerically to $0$ or $1$. We overcome such problems
by using some logarithms, which complicates and slows down the program.

Let us now briefly summarize the results of our experiments, see Section \ref{numer}.

We empirically observed on various games 
that, very roughly, each iteration of our algorithm requires between $2$ and $4$ times more computational time 
than MCTS.

When playing Pearl's games, our algorithm seems rather competitive against MCTS 
(with a given amount of time per move), 
which is not very surprising, since our algorithm is precisely designed for such games.

We also tried to play various versions of Connect Four. Globally, we are clearly
beaten by MCTS. However, there are two situations where we win.

The first one is when 
the game is so large, or the amount of time so small, that very few iterations can be performed
by the challengers. This is quite natural, because (a) our algorithm is only optimal 
{\it step by step}, (b) it relies on some independence properties that are less and less
true when performing more and more iterations.

The second one is when the game is so small that we can hope to find the winning strategy
at the first move, and where our algorithm finds it before MCTS. We believe this is due 
to the automatic pruning.

\subsection{Comparison with AlphaBeta}
Assume that $\cT$ is a finite balanced tree, i.e. that all the nodes with the 
same depth have the same degree, and that we have some i.i.d. outcomes on the leaves. This slightly 
generalizes Pearl's game.
Then Tarsi \cite{tarsi} showed that AlphaBeta is optimal in the sense of the 
expected number of leaves necessary to perfectly find $R(r)$.
For such a game, Bruno Scherrer told us that our algorithm visits the leaves precisely 
in the same order as AlphaBeta, up to some random permutation (see also Subsection \ref{rat} 
for a rather convincing numerical indication in this direction).
The advantage is that we provide an estimated value during the whole process, while AlphaBeta produces 
nothing before it really finds $R(r)$.

This strict similarity with AlphaBeta does not hold generally, because
our algorithm takes into account the degrees of the nodes. On the one hand, a player
is happy to find a node with more possible moves than expected.
On the other hand, the value of such a node may be difficult to determine. 
So the way our algorithm takes degrees into account is complicated and not very transparent.

\subsection{Organization of the paper}

In the next section, we precisely state our main results and describe our algorithm.
Section \ref{ggcc} is devoted to the convergence proof.
In Sections \ref{prelim} and \ref{pprroo}, we establish our main result. 
Section \ref{comput} is devoted to the computation of the functions $m$ and $s$ for particular models.
In Section \ref{gof}, we show on an example that global optimality fails.
We present numerical experiments in Section \ref{numer}. In Section \ref{smcts}, we precisely describe the 
versions of MCTS and its variant we used to test our algorithms.

\section{Notation and main results}\label{ggggg}

\subsection{Notation}\label{ssn}
We first introduce once for all the whole notation we need concerning trees.

Let $\TT$ be the complete discrete ordered tree with root $r$ and infinite degree.
An element of $\TT$ is a finite word composed of letters in $\nn_*$.
The root $r$ is the empty word. If e.g. $x=n_1n_2n_3$, this means that $x$ is the
$n_3$-th child of the $n_2$-th child of the $n_1$-th child of the root.
We consider \emph{ordered trees} to simplify the presentation, but the order
will not play any role.

The depth (or generation) $|x|$ of $x\in\TT$ is the number of letters of $x$. In particular, $|r|=0$.

We say that $y \in \TT$ is the father of $x\in \TT\setminus\{r\}$ (or that $x$ is a child of $y$)
if there is $n\in\nn_*$ such that $x=yn$. We denote by $f(x)$ the father of $x$.

For $x \in \TT$, we call $\CC_x=\{y\in\TT: f(y)=x\}$ the set of all the children of $x$ and 
$\TT_x$ the whole progeny of $x$: $\TT_x$ is the subtree of $\TT$
composed of $x$, its children, its grandchildren, etc.

We say that $x,y\in\TT$ are brothers if they are different and have the same father.
For $x\in \TT\setminus\{r\}$, we denote by $\HH_x=\CC_{f(x)}\setminus\{x\}$ the set of all the brothers of $x$.
Of course, $\HH_r=\emptyset$.

For $x\in \TT$ and $y \in \TT_x$, we denote by $B_{xy}$ is the branch from $x$ to $y$.
In other words, $z\in B_{xy}$ if and only if $z\in \TT_x$ and $y \in \TT_z$.

For $x\in\TT$, we introduce $\KK_x=\{r\}\cup \bigcup_{y\in B_{rx}\setminus \{x\}} \CC_y=
\cup_{y\in B_{rx}} (\{y\}\cup \HH_y)$, which consists of $x$, its brothers, its father and uncles, its grandfather 
and granduncles, etc.

For $\bx \subset \TT$, we introduce $B_\bx=\cup_{x\in\bx} B_{rx}$, the finite subtree
of $\TT$ with root $r$ and set of leaves $\bx$. 

For $\bx \subset \TT$, we also set $\DD_\bx=(\cup_{x\in B_\bx} \HH_x)\setminus B_\bx$, 
the set of all the brothers of the elements of $B_\bx$
that do not belong to $B_\bx$. Observe that $B_\bx \cup \DD_\bx=\cup_{x\in\bx}\KK_x$.

Let $\cS_f$ be the set of all finite subtrees of $\TT$ with root $r$. For $T$ a finite 
subset of $\TT$,
it holds that $T\in\cS_f$ if and only if for all $x\in T$, $B_{rx}\subset T$. 

For $T\in\cS_f$ and  $x\in T$, we introduce $C_x^T=T\cap\CC_x$ the set of the children of $x$ in $T$,
$H^T_x=T\cap \HH_x$ the set of the brothers of $x$ in $T$,
$T_x=T\cap\TT_x$ the whole progeny of $x$ in $T$, and $K^T_x=T\cap \KK_x$ which contains $x$, its brothers (in $T$),
its father and uncles (in $T$), its grandfather and granduncles (in $T$), etc. See Figure 1.

We denote by $L_T=\{x\in T :T_x=\{x\} \}$ the set of the leaves of $T\in\cS_f$. We have $B_{L_T}=T$.

Finally, for $T\in\cS_f$ and  $\bx\subset T$, we introduce $D_\bx^T=T\cap \DD_\bx$, the set of all the brothers
(in $T$) of the elements of $B_\bx$ not belonging to $B_\bx$ (observe that $B_\bx\subset T$),
see Figure 2.
It holds that $B_\bx \cup D^T_\bx= \cup_{x \in \bx} K^T_x$.

\begin{figure}[b]
\setlength{\unitlength}{0.7mm}
\begin{center}\framebox{
\begin{picture}(100,80)
\put(49.5,67){$r$}
\put(50,65){\line(-32,-15){32}}\put(50,65){\line(0,-1){15}}\put(50,65){\line(32,-15){32}}
\put(18,50){\line(-10,-15){10}}\put(18,50){\line(10,-15){10}}   
\put(50,50){\line(-10,-15){10}}\put(50,50){\line(0,-15){15}}\put(50,50){\line(10,-15){10}}   
\put(82,50){\line(-10,-15){10}}\put(82,50){\line(10,-15){10}}
\put(8,35){\line(-4,-15){4}}\put(8,35){\line(4,-15){4}}
\put(28,35){\line(-4,-15){4}}\put(28,35){\line(4,-15){4}}
\put(40,35){\line(-4,-15){4}}\put(40,35){\line(0,-15){15}}\put(40,35){\line(4,-15){4}}
\put(60,35){\line(-4,-15){4}}\put(60,35){\line(-1,-15){1}}\put(60,35){\line(1,-15){1}}\put(60,35){\line(4,-15){4}}
\put(92,35){\line(-4,-15){4}}\put(92,35){\line(0,-15){15}}\put(92,35){\line(4,-15){4}}
\put(32,20){\line(-2,-15){2}}\put(32,20){\line(2,-15){2}}
\put(44,20){\line(-2,-15){2}}\put(44,20){\line(2,-15){2}}
\put(92,20){\line(-2,-15){2}}\put(92,20){\line(2,-15){2}}
\put(46,19){$x$}
\linethickness{2pt}
\put(50,65){\line(-32,-15){32}}\put(50,65){\line(0,-1){15}}\put(50,65){\line(32,-15){32}}
\put(50,50){\line(-10,-15){10}}\put(50,50){\line(0,-15){15}}\put(50,50){\line(10,-15){10}}  
\put(40,35){\line(-4,-15){4}}\put(40,35){\line(0,-15){15}}\put(40,35){\line(4,-15){4}}
\put(0,75){\small {\bf Figure 1.} The subtree $K_x^T$ of $T$ is thick.}
\end{picture}}
\framebox{
\begin{picture}(100,80)
\put(49.5,67){$r$}
\put(50,65){\line(-32,-15){32}}\put(50,65){\line(0,-1){15}}\put(50,65){\line(32,-15){32}}
\put(18,50){\line(-10,-15){10}}\put(18,50){\line(10,-15){10}}   
\put(50,50){\line(-10,-15){10}}\put(50,50){\line(0,-15){15}}\put(50,50){\line(10,-15){10}}   
\put(82,50){\line(-10,-15){10}}\put(82,50){\line(10,-15){10}}
\put(8,35){\line(-4,-15){4}}\put(8,35){\line(4,-15){4}}
\put(28,35){\line(-4,-15){4}}\put(28,35){\line(4,-15){4}}
\put(40,35){\line(-4,-15){4}}\put(40,35){\line(0,-15){15}}\put(40,35){\line(4,-15){4}}
\put(60,35){\line(-4,-15){4}}\put(60,35){\line(-1,-15){1}}\put(60,35){\line(1,-15){1}}\put(60,35){\line(4,-15){4}}
\put(92,35){\line(-4,-15){4}}\put(92,35){\line(0,-15){15}}\put(92,35){\line(4,-15){4}}
\put(32,20){\line(-2,-15){2}}\put(32,20){\line(2,-15){2}}
\put(44,20){\line(-2,-15){2}}\put(44,20){\line(2,-15){2}}
\put(92,20){\line(-2,-15){2}}\put(92,20){\line(2,-15){2}}
\put(46,19){$x_1$}\put(95,4){$x_2$}
\put(36,20){\circle*{2}}\put(40,20){\circle*{2}}\put(50,35){\circle*{2}}\put(60,35){\circle*{2}}
\put(72,35){\circle*{2}}\put(88,20){\circle*{2}}\put(96,20){\circle*{2}}
\put(90,5){\circle*{2}}\put(18,50){\circle*{2}}
\put(0,75){\small {\bf Figure 2.} $D_{\{x_1,x_2\}}^T$ consists of the bullets.}
\end{picture}}\end{center}
\end{figure}

\subsection{The general model}
We have two players $J_0$ and $J_1$. The game is modeled by a finite tree
$\cT \in \cS_f$ with root $r$ and leaves $\cL=L_\cT$. An element $x\in \cT$ represents a 
configuration of the game.
On each node $x\in \cT\setminus \cL$, we set $t(x)=1$ if it is $J_1$'s turn to play when in 
the configuration
$x$ and $t(x)=0$ else.
The move is alternate and the player $J_1$ starts, so that
we have $t(x)=\indiq_{\{|x| \hbox{ \tiny is even}\}}$, where $|x|$ 
is the depth of $x$.
We have some outcomes $(R(x))_{x\in \cL}$ in $\{0,1\}$.
We say that $x\in \cL$ is a winning outcome for $J_1$ (resp. $J_0$) if $R(x)=1$ (resp. $R(x)=0$).

So $J_1$ starts, he chooses a node $x_1$ among the children of $r$, then $J_0$ 
chooses a node $x_2$
among the children of $x_1$, and so on, until we reach a leave $y\in \cL$, and $J_1$ is the winner
if $R(y)=1$, while $J_0$ is the winner if $R(y)=0$.

\subsection{Notation} For $x\in\cT$, we set $\cC_x=C_x^{\cT}$, $\cH_x=H_x^\cT$ and $\cK_x=K_x^\cT$
and, for $\bx\subset \cT$, $\cD_\bx=D_\bx^\cT$.

\subsection{Minimax values}

Given the whole tree $\cT$ and the outcomes $(R(x))_{x\in \cL}$, we can theoretically completely 
solve the game.
We classically define the minimax values $(R(x))_{x\in \cT}$ as follows. For any $x\in\cT$,
$R(x)=1$ if $J_1$ has a winning strategy when starting from $x$ and $R(x)=0$ else (in which case
$J_0$ necessarily has a winning strategy when starting from $x$).

\begin{rk}\label{mm}
For $x\in \cL$, the value of $R(x)$ is prescribed. For $x\in \cT\setminus\cL$,
\begin{equation}\label{rr}
R(x)= \indiq_{\{t(x)=0\}} \min\{R(y):y \in \cC_x\}+
\indiq_{\{t(x)=1\}} \max\{R(y):y  \in \cC_x\}.
\end{equation}
It is thus possible to compute $R(x)$ for all $x\in\cT$ by backward induction, starting from the leaves.
\end{rk}

This is easily checked: for
$x\in \cT$ with $t(x)=0$, we have $R(x)=0$ if $x$ has at least one child $y\in\cT$ such that $R(y)=0$
(because $J_0$ can choose $y$ from which $J_1$ has no winning strategy)
and $R(x)=1$ else (because any choice of $J_0$ leads to a position $y$ from which $J_1$
has a winning strategy). This rewrites $R(x)= \min\{R(y):y \in \cC_x\}$.

If now $t(x)=1$, then $R(x)=1$ if $x$ has at least one child $y\in\cT$ such that $R(y)=1$ (because $J_1$
can choose $y$, from where it has a winning strategy) and  $R(x)=0$ else (because any choice
of $J_1$ leads to a position from which he has no winning strategy). This can be rewritten as 
$R(x)=\max\{R(y):y \in\cC_x\}$.

\subsection{The goal}

Our goal is to estimate at best $R(r)$
with a computer and a given amount of time. 

In practice, we (say, $J_1$) are playing at some true game 
such as \emph{Connect Four} or any deterministic game with alternate moves, 
against a true opponent (say, $J_0$). As already mentioned in the introduction, 
we may always consider that such a 
game is represented by a tree, and we may remove draws by identifying them to victories of $J_0$ (or of
$J_1$).

We are in some given configuration (after a certain number of 
true moves of both players). We call this configuration $r$.
We have to decide between several possibilities.
We thus want to estimate at best from which of these possibilities there is a winning 
strategy for $J_1$.
In other words, given a position $r$ (which will be the root of our tree), we want 
to know at best $R(r)=\max\{R(y):y\in \cC_r\}$: if our estimate suggests that $R(r)=0$,
then any move is similarly desperate. If our estimate suggests that $R(r)=1$, this necessarily relies
on the fact that we think that some (identified)
child $y_0$ of $r$ satisfies $R(y_0)=1$. Then in the true game, we will play $y_0$.

Of course, except for very \emph{small} games, it is not possible in practice to compute 
$R(r)$ as in Remark \ref{mm},
because the tree is too large. 

The computer knows nothing about the true game
except the rules: when it sees a position (node) $x \in\cT$,
it is able to decide if $x$ is terminal position (i.e. $x\in\cL$) or not;
if $x$ is a terminal position, it knows the outcome (i.e. $R(x)$);
if $x$ is not a terminal position, it knows who's turn it is to play (i.e. $t(x)$) 
and the possible moves (i.e. $\cC_x$).

The true game is deterministic and our study does not apply \emph{at all} to games of chance such as 
\emph{Backgammon} (because games of chance are more difficult to represent as trees, their minimax values 
or not clearly well-defined, etc.).
However, it is a very large game and, in some sense, unknown, 
so one might hope it resembles
the realization of a random model. We will thus assume that $\cT$, as well as the outcomes
$(R(x))_{x\in\cL}$, are given by the realization of some random model satisfying some independence properties.
It is not clear whether such an assumption is reasonable.
In any case, our theoretical results completely break down without such a condition.

\subsection{A class of algorithms}

We consider a large class of algorithms resembling the Monte Carlo Tree Search algorithm,
of which a version is recalled in details in the Appendix.
The idea is to make $J_1$ play against $J_0$ a large number of times: the first match is completely random,
but then we use the statistics of the preceding matches. MCTS makes $J_1$
and $J_0$ rather play some moves that often led them to victories. At the end, this provides
some ratings for the children of $r$. In the true game, against the true opponent, we then play
the move leading to the child of $r$ with the highest rating.

\begin{defi} 
We call a \emph{uniformly random match} starting from $x\in \cT$, with $y\in\cL$ 
as a \emph{resulting leave}, the following procedure.
Put $y_0=x$. If $y_0\in\cL$, set $y=y_0$.
Else, choose $y_1$ uniformly among $\cC_{y_0}$. If $y_1 \in \cL$, set $y=y_1$.
Else, choose $y_2$ uniformly among $\cC_{y_1}$. If $y_2 \in \cL$, set $y=y_2$. 
Etc.
Since $\cT$ is finite, this always ends up.
\end{defi}

The class of algorithms we consider is the following. 

\begin{defi}\label{proc}
An \emph{admissible algorithm} is a procedure of the following form.

{\bf Step 1.} 
Simulate a uniformly random match from $r$, call $x_1$ the resulting leave
and set $\bx_1=\{x_1\}$. Keep track of $R(x_1)$, of $B_{\bx_1}=B_{rx_1}$ and of 
$\cD_{\bx_1}=\cup_{y\in B_{rx_1}}\cH_y$.

{\bf Step n+1.} Using only the knowledge of $B_{\bx_n}$,
$\cD_{\bx_n}$ and $(R(x))_{x\in \bx_n}$,
choose some (possibly randomized) $z_n\in \cD_{\bx_n}\cup \bx_n$.

If $z_n\in\bx_n$, set $x_{n+1}=z_n$, $\bx_{n+1}=\bx_n$, $B_{\bx_{n+1}}=B_{\bx_n}$ and $\cD_{\bx_{n+1}}=\cD_{\bx_n}$.

If $z_n \in \cD_{\bx_n}$, simulate a uniformly 
random match starting from $z_n$ and call $x_{n+1}$ the resulting leave. Set $\bx_{n+1}=\bx_n\cup\{x_{n+1}\}$
and keep track of $R(x_{n+1})$, of 
$B_{\bx_{n+1}}=B_{\bx_n}\cup B_{rx_{n+1}}$ and of $\cD_{\bx_{n+1}}=(\cD_{\bx_n}\setminus \{z_n\}) \cup \bigcup_{y\in B_{z_nx_{n+1}}
\setminus \{z_n\}}\cH_y$.

{\bf Conclusion.} Stop after a given number of iterations $n_0$ 
(or after a given amount of time). Choose some (possibly randomized) 
\emph{best child} $x_*$ of $r$, using 
only the knowledge of
$B_{\bx_{n_0}}$, $\cD_{\bx_{n_0}}$ and $(R(x))_{x\in \bx_{n_0}}$.
\end{defi}

After $n$ matches, $B_{\bx_n}$ represents the explored part of $\cT$, while 
$\cD_{\bx_n}$ represents its boundary.

Note that we assume we that know $\cD_{\bx_n}$ (the set of all the brothers, in $\cT$, of the elements of $B_{\bx_n}$
that are not in $B_{\bx_n}$) after the simulation some matches 
leading to the set of leaves
$\bx_n$. This is motivated by the following reason.
Any $y\in\cD_{\bx_n}$ has its father in $B_{\bx_n}$. Thus at 
some point of the simulation,
we visited $f(y)$ for the first time and we had to decide (at random)
between all its children: we are aware of the fact that $y \in\cT$.

Also note that we assume that the best thing to do, when visiting a position for the first time
(i.e. when arriving at some element of $\cD_{\bx_n}$), is to simulate from there 
a uniformly random match.
This models that fact that we know nothing of the game, except the rules.

The randomization will allow us, in practice, 
to make some \emph{uniform} choice in case of equality.

Finally, it seems stupid to allow $x_{n+1}$ to belong to $\bx_n$,
because this means we will simulate \emph{precisely} 
a match we have already simulated: this will not give us some new information.
But this avoids many useless discussions. Anyway, a \emph{good} algorithm will always, 
or almost always,
exclude such a possibility.

\begin{rk}\label{tasoeur}
(i) In \emph{Step n+1}, by \emph{``using only the knowledge of $B_{\bx_n}$, 
$\cD_{\bx_n}$ and $(R(x))_{x\in \bx_n}$
choose some (possibly randomized) $z_n\in \cD_{\bx_n}\cup \bx_n$''}, we mean that 
$z_n=F(B_{\bx_n},\cD_{\bx_n}, (R(x))_{x\in \bx_n}, X_{n})$,
where $X_{n}\sim \cU([0,1])$ is independent of everything else and where $F$ is a
deterministic measurable application from $A$ to $\TT$, 
where $A$ is the set of all 
$w=(B,D,(\rho(x))_{x\in \bx},u)$, with $B\in \cS_f$, with $\bx=L_B$, with $D\subset \DD_{\bx}$
finite, with $(\rho(x))_{x\in \bx}\in \{0,1\}^{\bx}$, and with $u\in [0,1]$, such that $F(w) \in \bx\cup D$.

(ii) In \emph{Conclusion}, by \emph{``choose some (possibly randomized) 
\emph{best child} $x_*$ of $r$, using
only the knowledge of
$B_{\bx_{n_0}}, \cD_{\bx_{n_0}}, (R(x))_{x\in \bx_{n_0}}$''}, we  mean that 
$x_*=G(B_{\bx_{n_0}},\cD_{\bx_{n_0}}, (R(x))_{x\in \bx_{n_0}},X_{n_0})$,
where $X_{n_0}\sim \cU([0,1])$ is independent of everything else and where $G$ is a
deterministic application from $A$ to $\TT$ such that, for $w\in A$ as above, 
$G(w) \in C_r^{B\cup D}$.

(iii) The two applications $F,G$ completely characterize an admissible
algorithm.
\end{rk}

\subsection{Assumption}

Except for the convergence of our class of algorithms, the proof of which being purely 
deterministic, we will suppose at least the following condition.

\begin{ass}\label{as}
The tree $\cT$ is a random element of $\cS_f$. We denote by $\cL=L_\cT$ and, as already mentioned,
we set $\cC_x=C_x^{\cT}$, $\cH_x=H_x^\cT$ and $\cK_x=K_x^\cT$ for $x\in\cT$ and
$\cD_\bx=D_\bx^\cT$ for $\bx\subset \cT$.
Conditionally on $\cT$, we have some random outcomes $(R(x))_{x\in\cL}$ in $\{0,1\}$. 
We assume that for any 
$T\in\cS_f$ with leaves $L_T$, the family
$$
( (\cT_x,(R(y))_{y \in \cL \cap \cT_x}    ) , x\in L_T)
$$
is independent conditionally on $A_T=\{T\subset \cT$ and $\cD_{L_T} =\emptyset\}$ as soon as $\Pr(A_T)>0$.
\end{ass}

Observe that $A_T=\{T\subset \cT$ and $x\in T$ implies $\cH_x\subset T\}= 
\{T\subset \cT$ and $\forall\,x\in T$, $\cK_x=K^T_x\}$.

This condition is a branching property: knowing $A_T$, i.e. knowing that $T\subset \cT$ and that all the brothers
(in $\cT$) of $x\in T$ belong to $T$, we can write $\cT=T\cup \bigcup_{x\in L_T} \cT_x$, 
and the family $( (\cT_x,(R(y))_{y \in \cL \cap \cT_x}    ) , x\in L_T)$ is independent.
A first consequence is as follows.

\begin{rk}\label{ts}
Suppose Assumption \ref{as}. For $T\in\cS_f$ such that $\Pr(A_T)>0$ and $z \in L_T$,
we denote by $G_{T,z}$ the law of $(\cT_z,(R(y))_{y \in \cL \cap \cT_z})$ conditionally on $A_T$.
We have $G_{T,z}=G_{K^T_z,z}$.
\end{rk}

Indeed, put $S=K^T_z \subset T$ and observe that $A_T=A_S \cap \bigcap_{x\in L_S\setminus\{z\}} A'_x$,
where we have set $A'_x=\{T_x \subset \cT_x,\cD_{L_T}\cap \TT_x=\emptyset\}$.
But for each $x\in  L_S\setminus\{z\}$, $A'_x \in \sigma(\cT_x)$.
It thus follows from Assumption \ref{as} that $(\cT_z,(R(y))_{y \in \cL \cap \cT_z})$
is independent of $\bigcap_{x\in L_S\setminus\{z\}} A'_x$ knowing $A_S$. 
Hence its law knowing $A_T$ is the same as knowing $A_S$.

Assumption \ref{as} is of course satisfied if $\cT$ is deterministic and if the family $(R(x))_{x\in \cL}$ 
is independent.
It also holds true if $\cT$ is an inhomogeneous Galton-Watson tree and if, conditionally 
on $\cT$, the family 
$(R(x))_{x\in \cL}$ is independent and (for example) 
$R(x)$ is Bernoulli with some parameter depending only on the depth $|x|$.
But there are many other possibilities, see Section \ref{comput} for precise examples of models 
satisfying Assumption \ref{as}.

\subsection{Two relevant quantities}\label{trq}

Here we introduce two mean quantities necessary to our study.

\begin{defi}\label{dfms}
Suppose Assumption \ref{as}. Let $T\in\cS_f$ such that $\Pr(A_T)>0$, and $z \in L_T$.
Observe that on $A_T$, $z\in \cT$.

(i) We set $m(T,z)=\Pr(R(z)=1|A_T)$. By Remark \ref{ts},  $m(T,z)=m(K^T_z,z)$, because
$R(z)$ is of course a deterministic function of $(\cT_z,(R(y))_{y \in \cL \cap \cT_z})$,
see Remark \ref{mm}.

(ii) Simulate a uniformly random match starting from $z$, 
denote by $y$ the resulting leave. We put
$\cK_{zy}=\cK_y\cap\TT_z$ and introduce $\cG=\sigma(y,\cK_{zy},R(y))$.
We set
$$
s(T,z)=\E\Big[\Big(\Pr(R(z)=1\vert\cG \lor \sigma(A_T))
-m(T,z)\Big)^2 \Big\vert A_T \Big].
$$
By Remark \ref{ts},  $s(T,z)=s(K^T_z,z)$.
\end{defi}

Since $\E[\Pr(R(z)=1\vert\cG \lor \sigma(A_T))|A_T]=m(T,z)$, $s(T,z)$ is a conditional variance.

We will see in Section \ref{comput} that for some particular classes of 
models, $m$ and $s$ can be computed.

Recall that our goal is to produce some admissible algorithm.
Assume we have explored $n$ leaves $x_1,\dots,x_n$ and set $\bx_n=\{x_1,\dots,x_n\}$.
Recall that $B_{\bx_n}$ is the explored tree and that $\cD_{\bx_n}$ is, in some sense, its boundary.
For $z\in \bx_n$, we perfectly know $R(z)$. But for $z \in \cD_{\bx_n}$, we only know that
$z \in \cT$ and we precisely know $\cK_z$, since $\cK_z =K^\cT_z=K^{B_{\bx_n}\cup\cD_{\bx_n}}_z$.
Thus the best thing we can say is that
$R(z)$ equals $1$ with (conditional) probability 
$m(\cK_z,z)$. Also, $s(\cK_z,z)$ 
quantifies some mean amount of information we will get if
handling a uniformly random match starting from $z$.

\subsection{The conditional minimax values}

From now on, we work with the following setting.

\begin{sett}\label{n1}
Fix $n\geq 1$. Using an \emph{admissible} algorithm,
we have \emph{simulated} $n$ matches, leading to the leaves
$\bx_n=\{x_1,\dots,x_n\}\subset \cL$. Hence the $\sigma$-field
$\cF_n=\sigma(\bx_n, \cD_{\bx_n}, (R(x))_{x\in\bx_n})$ represents our knowledge
of the game. Observe that $B_{\bx_n}$ is of course $\cF_n$-measurable. Also, for any $x \in B_{\bx_n}\cup\cD_{\bx_n}$, 
$\cK_x$ and $\cH_x$ are $\cF_n$-measurable, as well as $\cC_x$ if $x\in B_{\bx_n}\setminus\{\bx_n\}$.
\end{sett}

This last assertion easily follows from the fact that for any $\bx\subset \cT$,
any element of $B_{\bx}\cup\cD_{\bx}$ has all its brothers (in $\cT$) in $B_{\bx}\cup\cD_{\bx}$.

We first want to compute $R_n(r)=\E[R(r)\vert \cF_n]=\Pr(R(r)=1 \vert \cF_n)$, 
which is, in some obvious sense,
the best approximation of $R(r)$ knowing $\cF_n$. Of course, we will have to compute
$R_n(x)$ on the whole explored subtree of $\cT$. We will check the following
result in Section \ref{pprroo}.

\begin{prop}\label{rn}
Grant Assumption \ref{as} and Setting \ref{n1}.
For all $x\in B_{\bx_n}\cup\cD_{\bx_n}$, define $R_n(x)=\Pr(R(x)=1 \vert \cF_n)$. 
They can be computed by backward induction,
starting from ${\bx_n}\cup\cD_{\bx_n}$, as follows.
For all $x\in {\bx_n}$, $R_n(x)=R(x)$. For all $x\in \cD_{\bx_n}$, 
$R_n(x)=m(\cK_x,x)$. For all
$x \in B_{\bx_n}\setminus{\bx_n}$, 
\begin{equation}\label{rnrec}
R_n(x)=\indiq_{\{t(x)=0\}} \prod_{y\in \cC_x} R_n(y) 
+ \indiq_{\{t(x)=1\}} \Big(1- \prod_{y\in \cC_x} (1-R_n(y))\Big).
\end{equation}
\end{prop}

\subsection{Main result}

We still work under Setting \ref{n1}.
We want to simulate a $(n+1)$-th match. We thus want to choose
some $z \in\cD_{\bx_n}$
and then simulate a uniformly random match starting from $z$,
in such a way that the increase of information
concerning $R(r)$ is as large as possible.
We unfortunately need a few more notation.

\begin{nota}\label{n2}
Adopt Setting \ref{n1}. 

(i) For $x\in (B_{\bx_n}\cup\cD_{\bx_n})\setminus\{r\}$, we set 
\begin{equation}\label{dfu}
U_n(x)=\indiq_{\{t(f(x))=0\}}\prod_{y\in \cH_x} R_n(y) 
+ \indiq_{\{t(f(x))=1\}} \prod_{y\in \cH_x} (1-R_n(y)).
\end{equation}

(ii) Define $Z_n(x)$, for all $x\in B_{\bx_n}\cup\cD_{\bx_n}$, by backward induction, 
starting from $\bx_n\cup \cD_{\bx_n}$,
as follows. If $x\in {\bx_n}$, set $Z_n(x)=0$. If $x\in \cD_{\bx_n}$, set 
$Z_n(x)=s(\cK_x,x)$. If $x \in B_{\bx_n}\setminus\bx_n$, set
\begin{equation}\label{dfz}
Z_n(x)=\max \{(U_n(y))^2Z_n(y) : y\in\cC_x \}.
\end{equation}

(iii) Fix $z\in \cD_{\bx_n}$, 
handle a uniformly random match starting from $z$, with resulting leave $y_z$,
set $\bx_{n+1}^z= {\bx_n}\cup\{y_z\}$
and denote by $\cF_{n+1}^z=\sigma(\bx_{n+1}^z,\cD_{\bx_{n+1}^z},(R(x))_{\bx_{n+1}^z})$
the resulting knowledge. Set $R_{n+1}^z(x)=\Pr(R(x)=1 \vert \cF_{n+1}^z)$
for all $x\in B_{\bx_{n+1}^z}\cup\cD_{\bx_{n+1}^z}$.
\end{nota}

Our main result reads as follows.

\begin{thm}\label{mr} Suppose Assumption \ref{as} and adopt Setting \ref{n1} 
and Notation \ref{n2}.
Define $z_* \in \cD_{\bx_n}\cup{\bx_n}$ as follows. Put $y_0=r$ and set
$y_1=\arg\!\max\{(U_n(y))^2Z_n(y):y \in \cC_{y_0}\}$. 
If $y_1 \in \cD_{\bx_n}\cup{\bx_n}$, set $z_*=y_1$.
Else, put $y_2=\arg\!\max\{(U_n(y))^2Z_n(y):y\in\cC_{y_1}\}$. 
If $y_2 \in \cD_{\bx_n}\cup{\bx_n}$, 
set $z_*=y_2$.
Else, put $y_3=\arg\!\max\{(U_n(y))^2Z_n(y):y \in \cC_{y_2}\}$, etc. This procedure
necessarily stops since $\cT$ is finite. Each time we use $\arg\!\max$, 
we choose e.g. at uniform
random in case of equality.

On the event $\{R_n(r)\notin\{0,1\}\}$, we have $z_*\in\cD_{\bx_n}$ and
\begin{equation}\label{super}
z_*=\arg\!\min \Big\{ \E\Big[(R_{n+1}^z(r)-R(r))^2\Big\vert \cF_n\Big] : z\in \cD_{\bx_n} \Big\}.
\end{equation}
\end{thm}

Observe that if $R_n(r)\in\{0,1\}$, then $R(r)=R_n(r)$,
because conditionally on $\cF_n$, $R(x)$ is Bernoulli with parameter $R_n(r)=\Pr(R(r)=1|\cF_n)$.
Hence on the event $\{R_n(r)\in\{0,1\}\}$, we perfectly know $R(r)$ from $\cF_n$ and 
thus the $(n+1)$-th match is useless.

When $R_n(r)\notin\{0,1\}$, we have the knowledge $\cF_n$, and the theorem tells us
how to choose $z_*\in\cD_{\bx_n}$ such that, after a uniformly random match starting from $z_*$,
we will estimate at best, in some $L^2$-sense, $R(r)$. 
In words, $z_*$ can be found by starting from the root, getting down in the tree
$B_{\bx_n}\cup\cD_{\bx_n}$ by following the maximum values of
$U_n^2Z_n$, until we arrive in $\cD_{\bx_n}$.

As noted by Bruno Scherrer, $z_*$ is also optimal if using a $L^1$-criterion.

\begin{rk}
With the assumptions and notation of Theorem \ref{mr}, it also holds that
\begin{equation}\label{super2}
z_*=\arg\!\min \Big\{ \E\Big[|R_{n+1}^z(r)-R(r)|\Big\vert \cF_n\Big] : z\in \cD_{\bx_n} \Big\}.
\end{equation}
\end{rk}

This is easily deduced from \eqref{super}, noting that conditionally on $\cF_{n+1}^z$ (which
contains $\cF_n$), $R(r)$ is Bernoulli$(R_{n+1}^z(r))$-distributed, and that for 
$X\sim$Bernoulli$(p)$, $\E[|X-p|]=2\E[(X-p)^2]$.

\subsection{The algorithm}

The resulting algorithm is as follows.

\begin{algo}\label{ouralgo} Each time we use $\arg\!\max$, we e.g. choose at uniform
random in case of equality.

{\bf Step 1.} 
Simulate a uniformly random match from $r$, call $x_1$ the resulting leave
and set $\bx_1=\{x_1\}$. 

During this random match, keep track of $R(x_1)$, of $B_{\bx_1}=B_{rx_1}$ and of 
$\cD_{\bx_1}=\cup_{y\in B_{rx_1}}\cH_y$ and set $R_1(x)=m(\cK_x,x)$ 
and $Z_1(x)=s(\cK_x,x)$ for all $x\in\cD_{\bx_1}$.

Set $x=x_1$, $R_1(x)=R(x_1)$ and $Z_1(x)=0$.

Do $\{x\!=\!f(x)$, compute $R_1(x)$ using \eqref{rnrec}, $(U_1(y))_{y\in\cC_x}$ using \eqref{dfu}
and $Z_1(x)$ using \eqref{dfz}$\}$ until $x=r$.

{\bf Step n+1.} Put $z\!=\!r$. Do 
$z=\arg\!\max\{(U_n(y))^2Z_n(y):y \in \cC_z\}$ until $z\in \cD_{\bx_n}$.
Set $z_n=z$.

Simulate a uniformly random match from $z_n$, call $x_{n+1}$ the 
resulting leave, set $\bx_{n+1}=\bx_n\cup\{x_{n+1}\}$.

During this random match, keep track of $R(x_{n+1})$, of 
$B_{\bx_{n+1}}=B_{\bx_n}\cup B_{rx_{n+1}}$ and of $\cD_{\bx_{n+1}}=(\cD_{\bx_n}\setminus \{z_n\}) \cup \bigcup_{y\in B_{z_nx_{n+1}}
\setminus \{z_n\}}\cH_y$ and set $R_{n+1}(x)=m(\cK_x,x)$ and $Z_{n+1}(x)=s(\cK_x,x)$ for all 
$x \in \bigcup_{y\in B_{z_nx_{n+1}}\setminus \{z_n\}}\cH_y$.

For all $x\in (B_{\bx_{n+1}}\cup\cD_{\bx_{n+1}})\setminus B_{rx_{n+1}}$,
put $R_{n+1}(x)=R_n(x)$ and $Z_{n+1}(x)=Z_n(x)$. 

For all $x\in (B_{\bx_{n+1}}\cup\cD_{\bx_{n+1}})\setminus \cK_{x_{n+1}}$,
put $U_{n+1}(x)=U_n(x)$.

Set $x=x_{n+1}$, put $R_{n+1}(x)=R(x_{n+1})$ and $Z_{n+1}(x)=0$.

Do $\{x=f(x)$, compute $R_{n+1}(x)$ using \eqref{rnrec}, compute $(U_{n+1}(y))_{y \in \cC_x}$ using \eqref{dfu}
and $Z_{n+1}(x)$ using \eqref{dfz}$\}$
until $x=r$.

If $R_{n+1}(r)\in\{0,1\}$, go directly to the conclusion.

{\bf Conclusion.} Stop after a given number of iterations $n_0$ 
(or after a given amount of time). As \emph{best child} of $r$, choose
$x_*=\arg\!\max\{R_{n_0}(x): x \in \cC_r\}$.
\end{algo}

\begin{figure}[b]
\setlength{\unitlength}{0.7mm}
\begin{center}\framebox{
\hbox{\hspace{-12mm}
\begin{picture}(110,100)
\put(49.5,67){$r$}
\put(50,65){\line(-32,-15){32}}\put(50,65){\line(0,-1){15}}\put(50,65){\line(32,-15){32}}
\put(50,50){\line(-10,-15){10}}\put(50,50){\line(0,-15){15}}\put(50,50){\line(10,-15){10}}   
\put(82,50){\line(-10,-15){10}}\put(82,50){\line(10,-15){10}}
\put(40,35){\line(-4,-15){4}}\put(40,35){\line(0,-15){15}}\put(40,35){\line(4,-15){4}}
\put(60,35){\line(-4,-15){4}}\put(60,35){\line(0,-15){15}}\put(60,35){\line(4,-15){4}}
\put(92,35){\line(-4,-15){4}}\put(92,35){\line(0,-15){15}}\put(92,35){\line(4,-15){4}}
\put(44,20){\line(-2,-15){2}}\put(44,20){\line(2,-15){2}}
\put(92,20){\line(-2,-15){2}}\put(92,20){\line(2,-15){2}}
\put(44,2){$x_1$}\put(63,17){$x_2$}\put(88,2){$x_3$}\put(66,35){$z_4$}
\linethickness{2pt}
\put(50,65){\line(0,-1){15}}\put(50,65){\line(32,-15){32}}
\put(50,50){\line(-10,-15){10}}\put(82,50){\line(10,-15){10}}\put(50,50){\line(10,-15){10}}  
\put(40,35){\line(4,-15){4}}\put(60,35){\line(4,-15){4}}\put(92,35){\line(0,-15){15}}
\put(44,20){\line(2,-15){2}}\put(92,20){\line(-2,-15){2}}
\put(36,20){\circle*{2}}\put(40,20){\circle*{2}}\put(50,35){\circle*{2}}
\put(72,35){\circle*{2}}\put(88,20){\circle*{2}}\put(96,20){\circle*{2}}
\put(56,20){\circle*{2}}\put(60,20){\circle*{2}}
\put(42,5){\circle*{2}}\put(94,5){\circle*{2}}\put(18,50){\circle*{2}}
\put(15,95){\small {\bf Figure 3.} After Step 3, we have the (thick)}
\put(15,90){\small explored tree $B_{\bx_3}$, its boundary (bullets)}
\put(15,85){\small  $\cD_{\bx_3}$, and the values of $R_3,Z_3$ and $U_3$ on the}
\put(15,80){\small whole picture.}
\put(15,73){\small 1. Select $z_4$ using the $(U_3,Z_3)$'s.}
\end{picture}}
\hbox{\hspace{-12mm}
\begin{picture}(118,100)
\put(49.5,67){$r$}
\put(11,-2){\line(0,1){103}}
\put(50,65){\line(-32,-15){32}}\put(50,65){\line(0,-1){15}}\put(50,65){\line(32,-15){32}}
\put(50,50){\line(-10,-15){10}}\put(50,50){\line(0,-15){15}}\put(50,50){\line(10,-15){10}}   
\put(82,50){\line(-10,-15){10}}\put(82,50){\line(10,-15){10}}
\put(40,35){\line(-4,-15){4}}\put(40,35){\line(0,-15){15}}\put(40,35){\line(4,-15){4}}
\put(60,35){\line(-4,-15){4}}\put(60,35){\line(0,-15){15}}\put(60,35){\line(4,-15){4}}
\put(92,35){\line(-4,-15){4}}\put(92,35){\line(0,-15){15}}\put(92,35){\line(4,-15){4}}
\put(44,20){\line(-2,-15){2}}\put(44,20){\line(2,-15){2}}
\put(92,20){\line(-2,-15){2}}\put(92,20){\line(2,-15){2}}
\put(72,35){\line(-2,-15){2}}\put(72,35){\line(2,-15){2}}
\put(74,20){\line(-2,-15){2}}\put(74,20){\line(2,-15){2}}
\put(44,2){$x_1$}\put(63,17){$x_2$}\put(88,2){$x_3$}\put(75,2){$x_4$}
\linethickness{2pt}
\put(50,65){\line(0,-1){15}}\put(50,65){\line(32,-15){32}}
\put(50,50){\line(-10,-15){10}}\put(82,50){\line(10,-15){10}}\put(50,50){\line(10,-15){10}}  
\put(82,50){\line(-10,-15){10}}  
\put(40,35){\line(4,-15){4}}\put(60,35){\line(4,-15){4}}\put(92,35){\line(0,-15){15}}
\put(44,20){\line(2,-15){2}}\put(92,20){\line(-2,-15){2}}
\put(72,35){\line(2,-15){2}}
\put(74,20){\line(2,-15){2}}
\put(36,20){\circle*{2}}\put(40,20){\circle*{2}}\put(50,35){\circle*{2}}
\put(88,20){\circle*{2}}\put(96,20){\circle*{2}}\put(70,20){\circle*{2}}
\put(56,20){\circle*{2}}\put(60,20){\circle*{2}}\put(72,5){\circle*{2}}
\put(42,5){\circle*{2}}\put(94,5){\circle*{2}}\put(18,50){\circle*{2}}

\put(15,95){\small 2. Simulate a uniformly random match from $z_3$,}
\put(15,90){\small leading to a leave $x_4$. This builds a new thick}
\put(15,85){\small branch and new bullets.}
\put(15,78){\small 3. Observe $R(x_4)$ and compute $R_4,Z_4,U_4$ on $\cK_{x_4}$.} 
\put(15,73){\small Everywhere else, set $(R_4,Z_4,U_4)=(R_3,Z_3,U_3)$.}
\end{picture}}}
\end{center}
\end{figure}

\subsection{The update is rather quick}\label{qu}

For e.g. a (deterministic) regular tree $\cT$ with degree $d$ and depth $K$,
the cost to achieve $n$ steps of the above algorithm is of order $nKd$,
because at each step, we have to update the values of $R_n(x)$ and $Z_n(x)$
for $x \in B_{rx_{n}}$ (which concerns $K$ nodes) and the values of $U_{n}(x)$ for $x \in \cK_{x_{n+1}}$
(which concerns $Kd$ nodes).

Observe that MCTS algorithms (see Sections \ref{intro} and \ref{mcts})
enjoy a cost of order
$Kn$, since the updates are done only on the branch $B_{rx_n}$ (or even less than that,
but in any case we have at least to simulate a random match, 
of which the cost is proportional to $K$, at each step).

The cost in $Kdn$ for Algorithm \ref{ouralgo} 
seems miraculous. We have not found any deep reason, but calculus, explaining why
this \emph{theoretically optimal} (in a loose sense) behaves so well.
It would have been more natural, see Remark \ref{update}, 
that the update would concern the whole explored tree
$B_{\bx_n}\cup\cD_{\bx_n}$, which contains much more than $Kd$ elements. 
Very roughly,
its cardinal is of order $Kdn$, which would lead to a cost of order
$K d n^2$ to achieve $n$ steps.

We did not write it down in the present paper, which is technical enough, 
but we also studied two variations of Theorem \ref{mr}.

$\bullet$ First, we considered the very same model, but we tried minimize
$\Pr(\indiq_{\{R_{n+1}^z(r)\}>1/2}\neq R(r) \vert \cF_n)$
instead of \eqref{super}. This is more natural, since in practice, one would rather estimate
$R(r)$ by $\indiq_{\{R_{n}(r)>1/2\}}$ than by $R_n(r)$ (because $R(r)$ takes values in $\{0,1\}$). 
It is possible to extend our theory, but this leads to an algorithm with a cost of order $K d n^2$
(at least, we found no way to reduce this cost).

$\bullet$ We also studied what happens in the case where the game may lead to draw. 
Then the outcomes $(A(x))_{x\in\cT}$ can take three values,
$0$ (if $J_0$ wins), $1$ (if $J_1$ wins) and $1/2$ (if the issue is a draw). 
For any $x\in\cT$, we can define the minimax rating $A(x)$ as 
$0$ (if $J_0$ has a winning strategy), $1$ (if $J_1$ has a winning strategy) and $1/2$
(else). The family $(A(x))_{x\in\cT}$ satisfies the backward induction relation \eqref{rr}.
A possibility is to identify a draw to a victory of $J_0$ (or of $J_1$). Then, under 
Assumption \ref{as}
with $R(x)=\indiq_{\{A(x)=1\}}$, we can
apply directly Theorem \ref{mr}. However, this leads to an algorithm that tries to find a winning 
move, but gives up if it thinks it cannot win: the algorithm does not make any difference between
a loss and a draw. It is possible to adapt our theory to overcome this default
by trying to estimate both $R(x)=\indiq_{\{A(x)=1\}}$ and $S(x)=\indiq_{\{A(x)=0\}}$, i.e. to
minimize something like $\E[a(R_{n+1}^z(r)-R(r))^2+b(S_{n+1}^z(r)-S(r))^2 \vert \cF_n]$
for some $a,b \geq 0$.
However, this leads, again, to an algorithm of which the cost is of order $K d n^2$, 
unless $b=0$ (or $a=0$), which means that we identify a draw to a victory of $J_0$ (or of $J_1$).
Technically, this comes from the fact that in such a framework, nothing like Observation \eqref{ttaacc} does occur,
see also Remark \ref{update}.

In practice, one can produce an algorithm taking draws into account as follows: at each step, we compute
$(R_n(x),Z_n(x))_{x \in B_{\bx_n}\cup\cD_{\bx_n}}$ identifying draws to victories of $J_0$ and, with obvious notation,
$(R_n'(x),Z_n'(x))_{x \in B_{\bx_n}\cup\cD_{\bx_n}}$ identifying draws to victories of $J_1$. 
We use our algorithm with $(R_n(x),Z_n(x))_{x \in B_{\bx_n}\cup\cD_{\bx_n}}$ while $R_n(r)$ is not too small, 
and we then switch to $(R_n'(x),Z_n'(x))_{x \in B_{\bx_n}\cup\cD_{\bx_n}}$. We have no clear idea of how to choose the 
threshold.

Finally, the situation is even worse for games with a large number (possibly infinite)
of game values (representing the gain of $J_1$). This could for example be modeled by independent Beta priors
on the leaves. As first crippling difficulty, Beta laws are not stable by maximum and minimum.

\subsection{Convergence}

It is not difficult to check that, even with a completely wrong model, Algorithm 
\ref{ouralgo} always converges
in a finite (but likely to be very large) number of steps.

\begin{rk}\label{nvi}
(i) Forgetting everything about theory,
we can use Algorithm \ref{ouralgo} with any pair of functions $m$ and $s$,
both defined on $\{(S,x): S \in \cS_f,x\in L_S\}$, $m$ being valued in $[0,1]$
and $s$ being valued in $[0,\infty)$.

(ii) For any constant $\lambda>0$, the algorithm using the functions $m$ and 
$\lambda s$ is precisely the same than
the one using $m$ and $s$.
\end{rk}

Note that we allow $s$ to be larger than $1$, which is never 
the case from a theoretical point of view. But in view of (ii), it is very natural.
We will prove the following result in Section \ref{ggcc}.

\begin{prop}\label{toujconv}
Consider any fixed tree $\cT\in \cS_f$ with $\cL$ its set of leaves 
and any fixed outcomes $(R(x))_{x\in \cL}$. Denote by $(R(x))_{x \in \cT}$ the corresponding minimax
values. Apply Algorithm \ref{ouralgo} with any given pair of functions $m$ and $s$ on
$\{(S,x): S \in \cS_f, S\subset \cT,x\in L_S\}$ with values in $[0,1]$ 
and $[0,\infty)$ respectively and satisfying the following condition:
for any $S\in\cS_f$ with $S\subset \cT$ and $\cD_{L_S}=\emptyset$, any $x\in L_S$,
$s(S,x)=0$ if and only if $m(S,x)\in\{0,1\}$, and in such a case, $m(S,x)=R(x)$.

(i) For every $n\geq 1$, $R_n(r)\notin \{0,1\}$ implies that $x_{n+1}\notin \bx_n$.

(ii) There is $n_0\geq 1$ finite such that 
$R_{n_0}(r)\in\{0,1\}$ and we then have $R_{n_0}(r)=R(r)$.
\end{prop}

Let us emphasize this proposition assumes nothing but the fact that $\cT$ is finite.
Assumption \ref{as} is absolutely not necessary here.
The condition on $m$ and $s$ is very general and obviously satisfied if e.g. $m$ is $(0,1)$-valued and if $s$ is 
$(0,\infty)$-valued.

Once a sufficiently large part of the tree is explored (actually, almost all the tree up to some pruning),
the algorithm knows perfectly the minimax value of $r$,
even if $m$ and $s$ are meaningless. Thus, the structure of the algorithm looks rather nice.
In practice, for a large game, the algorithm will never be able to explore 
such a large part of the tree, so that the choice of the functions $m$ and $s$ 
is very important. However,
Proposition \ref{toujconv} is reassuring: we hope that even if the modeling is approximate, 
so that the choices of $m$ and $s$ are not completely convincing,
the algorithm might still behave well.

\begin{lem}\label{msok}
Under Assumption \ref{as}, the functions $m$ and $s$ introduced in Definition \ref{dfms}
satisfy the condition of Proposition \ref{toujconv}: for any $S\in\cS_f$ 
such that $\Pr(A_S)>0$, any $x\in L_S$, 
$s(S,x)=0$ if and only if $m(S,x)\in\{0,1\}$, and in such a case, $R(x)=m(S,x)$ a.s. on $A_S$.
\end{lem}

Consequently, we can apply Proposition \ref{toujconv} under Assumption \ref{as} with the functions
$m$ and $s$ introduced in Definition \ref{dfms}: for any fixed realization of $\cT$ and $(R(x))_{x \in\cL}$,
if $S\in\cS_f$ with $S\subset \cT$ and $\cD_{L_S}=\emptyset$, then $A_S$ is realized, so that indeed,
for any $x\in L_S$, $s(S,x)=0$ if and only if $m(S,x)\in\{0,1\}$, and in such a case, $m(S,x)=R(x)$.

\subsection{Practical choice of the functions $m$ and $s$}\label{pchoice}

In Section \ref{comput}, we will describe a few models satisfying our conditions and 
for which we can compute the functions $m$ ans $s$. As seen in Definition \ref{dfms} 
(see also Algorithm \ref{ouralgo}),
it suffices to be able to compute $m(\cK_x,x)$ and $s(\cK_x,x)$ for all $x\in \cT$ (actually, for all
$x$ in the boundary $\cD_{\bx_n}$ of the explored tree).
Let us summarize the two main examples.

(i) First, assume that $\cT$ is an inhomogeneous Galton-Watson tree, with maximum depth $K$ and known reproduction
laws, and that conditionally on $\cT$, the outcomes $(R(x))_{x\in \cL}$ are independent Bernoulli random 
variables with parameters depending only on the depths of the involved nodes.
Then we will show that the functions $m(\cK_x,x)$ and $s(\cK_x,x)$ depend only on the depth of $x$ 
and can be computed
numerically, once for all, from the parameters of the model. See Subsection \ref{stm1} for precise statements.
Let us mention that the parameters of the model can be fitted to some real game such as Connect Four 
(even if it is not clear at all that this model is reasonable) by handling a high number of uniformly 
random matches, see Subsection \ref{thealgo} for a few more explanations.

(ii) Second, assume that $\cT$ is some given random tree to be specified later. 
Fix some values $a\in (0,1)$ and $b\in\rr$.
Define $m(\{r\},r)=a$, $s(\{r\},r)=1$ (or any other positive constant, see Remark \ref{nvi}) and, 
recursively, for all $x\in\cT$, define
$m(\cK_x,x)$ and $s(\cK_x,x)$ from $m(\cK_{f(x)},f(x))$, $s(\cK_{f(x)},f(x))$ and  $|\cC_{f(x)}|$
by the formulas
\begin{align*}
m(\cK_x,x)=&\indiq_{\{t(f(x))=0\}}[m(\cK_{f(x)},f(x))]^{1/|\cC_{f(x)}|}+\indiq_{\{t(f(x))=1\}}
(1-[1-m(\cK_{f(x)},f(x))]^{1/|\cC_{f(x)}|}),\\
s(\cK_x,x)=&\Big(\indiq_{\{t(f(x))=0\}}m(\cK_{f(x)},f(x))+\indiq_{\{t(f(x))=1\}}[1-m(\cK_{f(x)},f(x))])\Big)
^{b(|\cC_{f(x)}|-1)}\!\!\!\!s(\cK_{f(x)},f(x)).
\end{align*}
The formula for $m$ is a modeling symmetry assumption rather well-justified and we can treat
the following cases, see in Subsection \ref{stm3} for more details.

(ii)-(a) If we consider Pearl's model \cite{p}, i.e. $\cT$ is the deterministic $d$-regular tree with depth 
$K$ and the outcomes are i.i.d.
Bernoulli random variables with parameter $p$ (explicit as a function of $a,d,K$), then the above formula
for $s$ with $b=2$ is theoretically justified, see Remark \ref{pearl2}.

(ii)-(b) Assume next that $\cT$ is a finite homogeneous Galton-Watson tree with reproduction law 
$(1-p)\delta_0+p\delta_d$ (with $pd\leq 1$)
and that conditionally on $\cT$, the outcomes $(R(x))_{x\in \cL}$ are independent Bernoulli random variables
with parameters $(m(\cK_x,x))_{x\in \cL}$ (that do not need to be computed).
Then if $a=a_0$ is well-chosen (as a function of $p$ and $d$), the above formula for $s$ with $b=0$
(i.e. $s\equiv 1$) is theoretically justified, see Remark \ref{rtm3}.

We also experimented, without theoretical justification, other values of $b$.
But we obtained so few success in this direction that we will not present the corresponding numerical results.

Let us mention that while (i) requires to fit precisely the functions $m$ and $s$ using rough statistics
on the true game, (ii) is rather universal. In particular, it seems that the choice $a=0.5$ and $b=0$
works quite well in practice, and this is very satisfying. Also, the implementation is very simple,
since each time a new node $x$ is visited by the algorithm, we can compute
$m(\cK_x,x)$ from $m(\cK_{f(x)},f(x))$ and the number of children $|\cC_{f(x)}|$ of $f(x)$.

Finally observe that for any tree $\cT$ and any choices of $a\in (0,1)$ and $b\in \rr$, 
$m$ is $(0,1)$-valued and $s$ is $(0,\infty)$-valued, so that Proposition \ref{toujconv} applies:
the algorithm always converges in a finite number of steps.

\subsection{Global optimality fails}

By Theorem \ref{mr}, Algorithm \ref{ouralgo} is optimal, in a loose sense, \emph{step by step}.
That is, if knowing, for some $n\geq 1$, $\cD_{\bx_n}$ and the values of $R(x)$ for $x\in \bx_n \subset \cL$,
it tells us how to choose the next leave $x_{n+1}\in \cL$ so that
$\E[(R_{n+1}(r)-R(r))^2]$ is as small as possible.
However, it is not \emph{globally} optimal.

\begin{rk}\label{not}
Let $\cT$ be the (deterministic) binary tree with depth $3$ and assume that the 
outcomes $(R(x))_{x\in\cL}$
are i.i.d. and {\rm Bernoulli}$(1/2)$-distributed. Then Assumption \ref{as} is satisfied and 
we can compute the functions $m$ and $s$ 
introduced in Definition \ref{dfms}. We thus may 
apply Algorithm \ref{ouralgo}, producing some random leaves 
$x_1,x_2,\dots$. We set $\cF_n=\sigma(\{B_{\bx_n},\cD_{\bx_n},(R(x))_{x\in\bx_n}\})$ and $R_n(r)=\E[R(r)|\cF_n]$.

There is another \emph{admissible} algorithm, producing some random leaves 
$\tx_1,\tx_2,\dots$, such that, for $\ctF_n=\sigma(\{B_{\btx_n},\cD_{\btx_n},(R(x))_{x\in\btx_n}\})$ and 
$\tR_n(r)=\E[R(r)|\ctF_n]$, we have
$$
\E[(\tR_2(r)-R(r))^2] > \E[(R_2(r)-R(r))^2]\quad \hbox{but} \quad
\E[(\tR_3(r)-R(r))^2] < \E[(R_3(r)-R(r))^2].
$$
\end{rk}

It looks very delicate to determine the globally optimal algorithm.
Moreover, it is likely that such an algorithm will be very intricate and will not enjoy the 
\emph{quick update} property discussed in Subsection \ref{qu}.

\section{General convergence}\label{ggcc}

We first show that Algorithm \ref{ouralgo} is convergent with \emph{any} reasonable
parameters $m$ and $s$.

\begin{proof}[Proof of Proposition \ref{toujconv}]
We consider some fixed tree $\cT\in \cS_f$ with $\cL$ its set of leaves,
some fixed outcomes $(R(x))_{x\in \cL}$ and we denote by $(R(x))_{x \in \cT}$ 
the corresponding minimax
values. We also consider any pair of functions $m$ and $s$ 
on $\{(S,x): S \in \cS_f,x$ leave of $S\}$ with values in $[0,1]$
and $[0,\infty)$ respectively and we apply Algorithm \ref{ouralgo}.
We assume that for any $S\in\cS_f$ of which $x$ is a leave, 
$s(S,x)=0$ if and only if $m(S,x)\in\{0,1\}$, and that in such a case, $m(S,x)=R(x)$.

\emph{Step 1.} After the $n$-th step of the algorithm, we have 
some values $R_n(x)\in [0,1]$, $U_n(x)\in[0,1]$ and $Z_n(x)\geq 0$ for all $x\in B_{\bx_n}\cup\cD_{\bx_n}$,
for some $\bx_n=\{x_1,\dots,x_n\}\subset \cL$. These quantities can 
generally not be interpreted in terms 
of conditional expectations, but they always obey, by construction, the following rules.

(a) If $x\in\bx_n$, then $R_n(x)=R(x)$ and $Z_n(x)=0$.

(b) If $x\in\cD_{\bx_n}$, then $R_n(x)=m(\cK_x,x)$ and $Z_n(x)=s(\cK_x,x)$.

(c) If $x\in B_{\bx_n} \!\setminus \bx_n$, $R_n(x)=\indiq_{\{t(x)=0\}} 
\prod_{y\in \cC_x} R_n(y) + \indiq_{\{t(x)=1\}} [1- \prod_{y\in \cC_x} (1-R_n(y))]$.

(d) If $x\in B_{\bx_n}\!\setminus \bx_n$, $Z_n(x)=\max\{U_n^2(y)Z_n(y):y\in\cC_x\}$.

(e) If $x\in (B_{\bx_n}\cup\cD_{\bx_n})\setminus \!\{r\}$, 
$U_n(x)=\indiq_{\{t(v(x))=0\}}\prod_{y\in \cH_x} R_n(y) 
+ \indiq_{\{t(v(x))=1\}} \prod_{y\in\cH_x} (1-R_n(y)).$

We finally recall that, by Proposition \ref{rr},

(f) if $x \in \cT\setminus\cL$, $R(x)=\indiq_{\{t(x)=0\}} \min\{R(y):y\in\cC_x\}+
\indiq_{\{t(x)=1\}} \max\{R(y):y \in \cC_x\}$.

\emph{Step 2.} Here we prove that for all $x\in B_{\bx_n}\cup\cD_{\bx_n}$,
\begin{equation}\label{prop}
\hbox{$R_n(x)\in\{0,1\}$ and $Z_n(x)=0$ are equivalent and imply that $R_n(x)=R(x)$.}
\end{equation}

In the whole step, the notions of \emph{child} (of $x\in B_{\bx_n}\setminus \bx_n$) and 
\emph{brother} (of $x\in B_{\bx_n}\cup \cD_{\bx_n}$) refer to the tree $\cT$ or, equivalently,
to the tree $B_{\bx_n}\cup \cD_{\bx_n}$.

First, \eqref{prop} is obvious if $x\in\bx_n$ by point (a) 
(then $R_n(x)=R(x)\in\{0,1\}$ and $Z_n(x)=0$)
and if $x\in\cD_{\bx_n}$ by point (b) and our assumption on $m$ and $s$.
We next work by backward induction:
we consider  $x \in B_{\bx_n} \setminus \bx_n$, we assume that all its children
satisfy \eqref{prop},
and we prove that $x$ also satisfies \eqref{prop}. We assume for example that $t(x)=0$,
the case where $t(x)=1$ being treated similarly.

If $R_n(x)=0$, then by (c), $x$ has (at least) one child $y$ such that $R_n(y)=0$ 
whence, by induction 
assumption, $R(y)=0$ and thus $R(x)=0$ by (f). Furthermore, $R_n(y)=0$ implies 
that $U_n(z)=0$, whence
$U_n^2(z)Z_n(z)=0$, for all $z$ brother of $y$ by (e).
And by induction assumption, we have $Z_n(y)=0$, whence $U_n^2(y)Z_n(y)=0$.
We conclude, by (d), that $Z_n(x)=0$, and we have seen that $R_n(x)=0=R(x)$.

If $R_n(x)=1$, then by (c), all the children $y$ of $x$ satisfy $R_n(y)=1$, whence, by induction
assumption, $R(y)=1$ and thus $R(x)=1$ by (f). Still by induction assumption,
$Z_n(y)=0$ for all the children $y$ of $x$, whence $Z_n(x)=0$ by (d),
and we have seen that $R_n(x)=1=R(x)$.

Assume now that $Z_n(x)=0$, whence $U_n^2(y)Z_n(y)=0$ for all the children $y$ of $x$ by (d). 
If there is (at least)
one child $y$ of $x$ for which $U_n(y)=0$, this means that there is another child $z$ 
of $x$ for which
$R_n(z)=0$ by (e), whence $R_n(x)=0$ by (c). Else, we have $Z_n(y)=0$ for all 
the children $y$ of $x$, so that
$R_n(y)\in\{0,1\}$ by induction assumption, and thus $R_n(x)\in\{0,1\}$ by (c).

We have shown that $R_n(x)\in\{0,1\}$ implies $Z_n(x)=0$ and $R_n(x)=R(x)$, and we have verified
that  $Z_n(x)=0$ implies $R_n(x)\in\{0,1\}$. Hence $x$ satisfies \eqref{prop}, which completes
the step.

\emph{Step 3.} We now prove that if $x_{n+1}\in \bx_n$, then $R_n(r)\in\{0,1\}$ and this will
prove point (i).
Looking at Algorithm \ref{ouralgo}, we see that $x_{n+1}\in \bx_n$ means that the procedure
$$
\hbox{put $z=r$ and do 
$z=\arg\!\max\{U_n^2(y)Z_n(y):y \in \cC_z\}$ until $z\in \bx_n\cup\cD_{\bx_n}$}
$$
returns some $z \in \bx_n$.
But then, $Z_n(z)=0$ by (a). From (d) and the way $z$ has been built,
one easily gets convinced that this implies that $Z_n(r)=0$, whence $R_n(r)\in\{0,1\}$ by Step 1.

\emph{Step 4.} By Step 3 and since $\cT$ has a finite number of leaves, 
$n_0=\inf\{n\geq 1 : R_n(r)\in\{0,1\} \}$ is well-defined and finite.
Finally, we know from Step 2 that $R_{n_0}(r)=R(r)$.
\end{proof}

\section{Preliminaries}\label{prelim}

We first establish some general formulas concerning the functions $m$ and $s$.
They are not really necessary to understand the proof of our main result, but we need them
to show Lemma \ref{msok}. Also, we will use them to derive more tractable expressions
in some particular cases in Section \ref{comput}.

\begin{lem}\label{hor}
Suppose Assumption \ref{as} and recall Definition \ref{dfms}.
For any $S\in\cS_f$ such that $\Pr(A_S)>0$ and any $x\in L_S$,
\begin{align*}
m(S,x)=&\Pr(x\in\cL,R(x)=1 | A_S) + \!\!\!\sum_{\by \subset \CC_x,|\by|\geq 1}\!\! \Pr(\cC_x=\by |A_S)
\Theta(S,x,\by),\\
s(S,x)=&\sum_{k\in\{0,1\}}\Pr(x\in\cL,R(x)=k | A_S) [k-m(S,x)]^2
+\!\!\! \sum_{\by \subset \CC_x,|\by|\geq 1} \!\! \Pr(\cC_x=\by |A_S) \sum_{y\in\by} 
\frac {\Gamma(S,x,\by,y)}{|\by|},
\end{align*}
where
\begin{align*}
\Theta(S,&\,x,\by)=\indiq_{\{t(x)=0\}} \!\!\prod_{y\in\by} m(S\cup\by,y)+ 
\indiq_{\{t(x)=1\}} \Big(1\!\!-\!\!\prod_{y\in\by} (1-m(S\cup\by,y))\Big),\\
\Gamma(S,&\,x,\by,y)=\indiq_{\{t(x)=0\}} \Big(s(S\cup\by,y)\!\!\prod_{u\in\by\setminus\{y\}}[m(S\cup\by,u)]^2
+ \Big[\!\!\prod_{u\in\by} m(S\cup\by,u) - m(S,x)\Big]^2\Big)\\
&+ \indiq_{\{t(x)=1\}}  \Big(s(S\cup\by,y)\!\!\prod_{u\in\by\setminus\{y\}}[1-m(S\cup\by,u)]^2
+ \Big[\!\!\prod_{u\in\by} (1-m(S\cup\by,u)) - (1-m(S,x))\Big]^2\Big).
\end{align*}
\end{lem}

\begin{proof} We fix $S\in\cS_f$ such that $\Pr(A_S)>0$ and $x\in L_S$. We first observe that
for any $\by \subset \CC_x$ with $|\by|\geq 1$, $A_S\cap \{\cC_x=\by\}=A_{S\cup\by}$
(recall that $A_S$ is the event on which $S\subset \cT$ and all the brothers (in $\cT$)
of all the elements of $S$ also belong to $S$).

We now study $m(S,x)=\Pr(R(x)=1|A_S)$, starting from 
$$
m(S,x)=\Pr(x\in\cL,R(x)=1|A_S)+\sum_{\by \subset \CC_x,|\by|\geq 1}\Pr(\cC_x=\by,R(x)=1 |A_S).
$$
Hence the only difficulty is to verify that $\Pr(\cC_x=\by,R(x)=1 |A_S)= 
\Pr(\cC_x=\by|A_S)\Theta(S,x,\by)$ or, equivalently, that  
\begin{equation}\label{obj1}
\Pr(R(x)=1 |A_S\cap\{\cC_x=\by\})=
\Theta(S,x,\by).
\end{equation}

Since $A_S\cap \{\cC_x=\by\}=A_{S\cup\by}$ and since $\by \subset L_{S\cup\by}$, 
we know from Assumption \ref{as} 
that the family $(\cT_y,(R(u))_{u\in \cL\cap\cT_y})_{y\in\by}$ 
is independent conditionally on $A_S\cap \{\cC_x=\by\}$. Consequently, the family $(R(y))_{y\in\by}$ 
is independent conditionally on $A_S\cap \{\cC_x=\by\}$ (because $R(y)$ depends only on 
$\cT_y$ and $(R(u))_{u\in \cL\cap\cT_y}$, recall Remark \ref{mm}).
We assume e.g. $t(x)=1$. Since 
$R(x)=\max\{R(y):y\in\cC_x\}$, we may write
\begin{align*}
\Pr(R(x)=1 |A_S\cap\{\cC_x=\by\})=1- \prod_{y\in\by} \Pr(R(y)=0|A_S\cap \{\cC_x=\by\}).
\end{align*}
But for $y\in\by$, $\Pr(R(y)=0|A_S\cap \{\cC_x=\by\})=\Pr(R(y)=0|A_{S\cup\by}) =1-m(S\cup\by,y)$,
whence \eqref{obj1}, because $t(x)=1$.

We next study $s$.
Knowing $A_S$, we handle a uniformly random match starting from $x$, with resulting leave $v$
and we set $\cG=\sigma(v,R(v),\cK_{xv})$, where $\cK_{xv}=\cK_v \cap \TT_x$. 
We recall that $s(S,x)=\E[(R_1(x)-m(S,x))^2 |A_S]$,
where $R_1(x)=\Pr(R(x)=1|\cG\lor\sigma(A_S))$. If $x\in\cL$, then $v=x$, whence $R(x)$ is $\cG$-measurable 
and thus $R_1(x)=R(x)$. Consequently,
$$
s(S,x)=\E[(R(x)-m(S,x))^2\indiq_{\{x\in\cL\}}| A_S]
+\!\!\! \sum_{\by \subset \CC_x,|\by|\geq 1} \!\! \E[(R_1(x)-m(S,x))^2\indiq_{\{\cC_x=\by\}} |A_S].
$$
We have $\E[(R(x)-m(S,x))^2\indiq_{\{x\in\cL\}}| A_S]=\sum_{k\in\{0,1\}}\Pr(x\in\cL,R(x)=k | A_S) [k-m(S,x)]^2$.
We thus only have to check that
$\E[(R_1(x)-m(S,x))^2 \indiq_{\{\cC_x=\by\}}|A_S]=\Pr(\cC_x=\by|A_S)\sum_{y\in\by}|y|^{-1}\Gamma(S,x,\by,y)$
or, equivalently, that 
$\E[(R_1(x)-m(S,x))^2|A_S\cap \{\cC_x=\by\}]=\sum_{y\in\by}|y|^{-1}\Gamma(S,x,\by,y)$.

On $A_S\cap \{\cC_x=\by\}$, let $w$ be the child of $x$ belonging to $B_{xv}$.
Since $v$ is obtained by handling a uniformly random match starting from $x$,
$\Pr(w=y|A_S\cap \{\cC_x=\by\})=|\by|^{-1}$ for all $y\in\by$. We thus only have to verify that
\begin{equation}\label{obj2}
\E[(R_1(x)-m(S,x))^2|A_S\cap \{\cC_x=\by\}\cap\{w=y\}]=\Gamma(S,x,\by,y).
\end{equation}

But $A_S\cap \{\cC_x=\by\}=A_{S\cup\by}$, so that, by Assumption \ref{as}
(and since the random match is independent of everything else), the family
$(\cT_u,(R(z))_{z\in \cL\cap\cT_u})_{u\in\by}$ is independent
conditionally on $A_S\cap \{\cC_x=\by\}\cap\{w=y\}$.
Hence the family $(R(u))_{u \in \by\setminus\{y\}}$ is independent and independent of $(\cT_y,(R(z))_{z\in\cL\cap\cT_y})$
conditionally on $A_S\cap \{\cC_x=\by\}\cap\{w=y\}$. In particular, $(R(u))_{u \in \by\setminus\{y\}}$ 
is independent
of $\cG$ conditionally on $A_S\cap \{\cC_x=\by\}\cap\{w=y\}$. 

From now on, we assume e.g. that $t(x)=0$.

We have $R(x)=\min\{R(u):u\in\by\}=\prod_{u\in\by}R(u)$ on $\{\cC_x=\by\}$ and we conclude from the above
independence property that,
conditionally on $A_S\cap \{\cC_x=\by\}\cap\{w=y\}=A_{S\cup\by}\cap\{w=y\}$,
$$
R_1(x)= \Pr\Big(\prod_{u\in\by}R(u)=1\Big|\cG\lor\sigma(A_{S\cup\by}) \Big)=\Pr(R(y)=1|\cG\lor\sigma(A_{S\cup\by})) 
\prod_{u\in\by\setminus \{y\}} \Pr(R(u)=1|A_{S\cup\by}).
$$
But $\Pr(R(u)=1|A_{S\cup\by})=m(S\cup\by,u)$. Adopting the notation 
$R_1(y)= \Pr(R(y)=1|\cG\lor\sigma(A_{S\cup\by}))$, we deduce that
$R_1(x)=R_1(y)\prod_{u\in\by\setminus \{y\}}m(S\cup\by,u)$ on $A_S\cap \{\cC_x=\by\}\cap\{w=y\}$, whence
$$
R_1(x)-m(S,x)=[R_1(y)-m(S\cup \by,y)]\prod_{u\in\by\setminus \{y\}} m(S\cup\by,u) + 
\Big[\prod_{u\in\by} m(S\cup\by,u)-m(S,x)\Big].
$$
Recall that $t(x)=0$.
To conclude that \eqref{obj2} holds true, it only remains to verify that

(a) $\E[(R_1(y)-m(S\cup\by,y))^2|A_S\cap \{\cC_x=\by\}\cap\{w=y\}]=s(S\cup\by,y)$,

(b) $\E[R_1(y) | A_S\cap \{\cC_x=\by\}\cap\{w=y\}]=m(S\cup \by,y)$.

By definition, we have $s(S\cup\by,y)=\E[(R_1(y)-m(S\cup\by,y))^2|A_{S\cup\by}]$
conditionally on $\{w=y\}$, because on $\{w=y\}$, $R_1(y)= \Pr(R(y)=1|\cG\lor\sigma(A_{S\cup\by}))$
is indeed the conditional probability that $R(y)=1$ knowing the information
provided by a uniformly random match starting from $w$ (with resulting leave $v$). Point (a) follows.

For (b), we write 
\begin{align*}
\E[R_1(y) | A_S\cap \{\cC_x=\by\}\cap\{w=y\}]=&\E[\Pr(R(y)=1 |
\cG\lor\sigma(A_{S\cup\by}))|A_{S\cup\by}\cap\{w=y\}]\\
=&\Pr(R(y)=1 |A_{S\cup\by}\cap\{w=y\})\\
=&\Pr(R(y)=1 |A_{S\cup\by})\\
=&m(S\cup\by,y).
\end{align*}
For the second equality, we used that $\{w=y\}\in\cG\lor\sigma(A_{S\cup\by})$.
For the third equality, we used that $w$ is of course independent of $R(y)$
knowing $A_{S\cup\by}$.
\end{proof}

We next give the 

\begin{proof}[Proof of Lemma \ref{msok}]
We fix $S\in\cS_f$ such that $\Pr(A_S)>0$ and $z \in L_S$.

If first $m(S,z)=\Pr(R(z)=1|A_S)=0$, then of course
$R(z)=0$ a.s. on $A_S$, whence also $\Pr(R(z)=1|\cG \lor \sigma(A_S))=0$ a.s. on $A_S$
(recall Definition \ref{dfms}) and thus $s(S,z)=0$.

Similarly, $m(S,z)=1$ implies that $R(z)=1$ a.s. on $A_S$ and that $s(S,z)=0$.

It only remains to prove that $s(S,z)=0$ implies that $m(S,z)\in\{0,1\}$.

If $\Pr(z\in\cL |A_S)>0$, then either $\Pr(z\in\cL,R(z)=0 |A_S)>0$ or 
$\Pr(z\in\cL,R(z)=1 |A_S)>0$. If $s(S,z)=0$, we
deduce from Lemma \ref{hor} that either $[m(S,z)]^2=0$ or $[1-m(S,z)]^2=0$, whence 
$m(S,z)\in\{0,1\}$.

If $\Pr(z\in\cL |A_S)=0$, we consider a finite tree $T$ with root $z$ such 
that $\Pr(\cT_z = T|A_S)>0$.
We set $U_z=S$ and, for all $x\in T\setminus\{z\}$, $U_x=S\cup \bigcup_{y \in B_{zx}\setminus\{x\}} C^T_y\in \cS_f$.
It holds that $x \in L_{U_x}$ for all $x \in T$ and, if $x\in T\setminus L_T$, 
$U_x\cup C^T_x=U_y$ for all $y\in C^T_x$.

We now prove by backward induction that for any $x\in T$, $s(U_x,x)=0$ 
implies that $m(U_x,x)\in\{0,1\}$.
Applied to $x=z$, this will complete the proof. 

If first $x\in L_T$, then  $\Pr(x\in\cL |A_{U_x})>0$, because $A_S \cap \{\cT_z=T\} \subset 
A_{U_x} \cap \{x\in\cL\}$, because $\Pr(\cT_z =T|A_S)>0$ and because $\Pr(A_S)>0$.
We thus have already seen that  $s(U_x,x)=0$ implies that $m(U_x,x)\in\{0,1\}$.

If next $x\in T\setminus L_T$, we introduce $\by=C^T_x$ and we 
see that $\Pr(\cC_x=\by |A_{U_x})>0$, because $A_S \cap \{\cT_z=T\} \subset 
A_{U_x} \cap \{\cC_x=\by\}$, because $\Pr(\cT_z =T|A_S)>0$ and because $\Pr(A_S)>0$. We deduce from  
Lemma \ref{hor} that if $s(U_x,x)=0$, then $\Gamma(U_x,x,\by,y)=0$ for all $y\in\by$. 
If e.g. $t(x)=0$, this implies that for all $y\in\by$ (recall that $U_x \cup \by=U_y$),
$$
\Gamma(U_x,x,\by,y)=s(U_y,y)\!\!\prod_{u\in\by\setminus\{y\}}m(U_u,u)
+ \Big[\!\!\prod_{u\in\by} m(U_u,u) - m(U_x,x)\Big]^2=0.
$$
Thus we always have $m(U_x,x)=\prod_{u\in\by} m(U_u,u)$ and 
either (i) $s(U_u,u)=0$ for all $u\in \by$ or (ii) there is $u\in \by$ such that $m(U_u,u)=0$.
In case (i), we deduce from the induction assumption that $m(U_u,u)\in\{0,1\}$
for all $u\in \by$, whence $m(U_x,x)\in\{0,1\}$. In case (ii), we of course have
$m(U_x,x)=0$.
\end{proof}

We next study the information provided by some admissible algorithm. Here, Assumption \ref{as}
is not necessary. The following result is intuitively obvious, but we found no short proof.

\begin{lem}\label{cruzero}
Recall Setting \ref{n1}: an admissible algorithm provided some leaves $\bx_n=\{x_1,\dots,x_n\}$
together with the objects $\cD_{\bx_n}$ and $(R(x))_{x\in\bx_n}$.
For any (deterministic) $\by_n=\{y_1,\dots,y_n\}\subset \TT$, any $D_n \subset \DD_{\by_n}$ and 
any $(a(y))_{y\in\by_n}\subset\{0,1\}^{\by_n}$, the law of
$(\cT,(R(y))_{y\in \cL})$ knowing
$$
A_n=\{\bx_n=\by_n,\cD_{\by_n} =D_n,(R(y))_{y\in\by_n}=(a(y))_{y\in\by_n}\}
$$ 
is the same as knowing 
$$
A_n'=\{\by_n\subset \cL,\cD_{\by_n}=D_n,(R(y))_{y\in\by_n}=(a(y))_{y\in\by_n}\}
$$
as soon as $\Pr(A'_n)>0$.
\end{lem}

\begin{proof} We work by induction on $n$.

\emph{Step 1.} We fix $y_1 \in \TT$, we set $\by_1=\{y_1\}$, we consider
$D_1 \subset \DD_{\by_1}$ and $a(y_1)\in\{0,1\}$. In this step we prove that
the law of $(\cT,(R(y))_{y\in \cL})$ knowing $A_1=\{x_1=y_1,\cD_{\by_1}=D_1,R(y_1)=a(y_1)\}$ is the same as knowing
$A_1'=\{y_1\in\cL,\cD_{\by_1}=D_1,R(y_1)=a(y_1)\}$.

To this aim, we consider $T\in\cS_f$ such that $y_1\in L_T$, $D_{\by_1}^T=D_1$
and $(\alpha(y))_{y \in L_T} \in \{0,1\}^{L_T}$ such that $\alpha(y_1)=a(y_1)$. We have to check that
$$
\frac{\Pr(\cT=T,(R(y))_{y\in L_T}=(\alpha(y))_{y \in L_T},x_1=y_1)}{\Pr(\cD_{\by_1}=D_1,R(y_1)=a(y_1),x_1=y_1)}
=\frac{\Pr(\cT=T,(R(y))_{y\in L_T}=(\alpha(y))_{y \in L_T} )}{\Pr(y_1\in\cL,\cD_{\by_1}=D_1,R(y_1)=a(y_1))}
$$
or, equivalently, that $p=q$, where
\begin{align*}
p=&\Pr(x_1=y_1 |\cT=T,(R(y))_{y\in L_T}=(\alpha(y))_{y \in L_T}),\\
q=& \Pr(x_1=y_1 |y_1\in\cL,\cD_{\by_1}=D_1,R(y_1)=a(y_1)).
\end{align*}
Recalling that $x_1$ is the leave resulting from a uniformly random match 
starting from $r$, one easily
gets convinced that $p=\Pr(x_1=y_1 |\cT=T)=\prod_{z\in B_{ry_1}\setminus \{y_1\}} |C^T_z|^{-1}$.
By the same way,
$q=\Pr(x_1=y_1 |y_1\in\cL,\cD_{\by_1}=D_1)=\prod_{z\in B_{ry_1}\setminus \{y_1\}} 
|C^{B_{y_1}\cup D_1}_z|^{-1}$.
Since $D_{\by_1}^T=D_1$, we have $C^T_z=C^{B_{y_1}\cup D_1}_z$
for all $z\in B_{ry_1}\setminus \{y_1\}$
and the conclusion follows.

\emph{Step 2.} Assume that the statement holds true with some $n\geq 1$.
Consider some deterministic $\by_{n+1}=\{y_1,\dots,y_{n+1}\}\subset \TT$,
$D_{n+1} \subset \DD_{\by_{n+1}}$ and $(a(y))_{y\in\by_{n+1}}\subset\{0,1\}^{\by_{n+1}}$, as well as the 
events
\begin{align*}
A_{n+1}=&\{\bx_{n+1}=\by_{n+1},\cD_{\by_{n+1}}=D_{n+1},(R(y))_{y\in\by_{n+1}}=(a(y))_{y\in\by_{n+1}}\}\\
A'_{n+1}=&\{\by_{n+1}\subset \cL,\cD_{\by_{n+1}}=D_{n+1},(R(y))_{y\in\by_{n+1}}=(a(y))_{y\in\by_{n+1}}\}.
\end{align*}

Recall that $x_{n+1}$ is chosen as follows: for some deterministic function $F$ as in Remark \ref{tasoeur}
and some $X_n\sim\cU([0,1])$ independent of everything else, we set
$z_n=F(\bx_n,\cD_{\bx_n},(R(x))_{x\in\bx_n},X_n)$,
which belongs to $\bx_n\cup\cD_{\bx_n}$. If $z_n\in\bx_n$, we set $x_{n+1}=z_n$, else,
we handle a uniformly random match starting from $z_n$ and denote by $x_{n+1}$ the resulting leave.

If $y_{n+1}\in\by_n$, then we have
$A_{n+1}=A_n \cap \{F(\by_n,D_{n},(a(x))_{x\in\by_n},X_n)=y_{n+1}\}$
and $A_{n+1}'=A_n'$, where $D_n=D_{n+1}$, where 
$A_n=\{\bx_n=\by_n,\cD_{\by_n}=D_{n},(R(y))_{y\in\by_n}=(a(y))_{y\in\by_n}\}$
and where ${A_n'=\{\by_n\subset \cL,\cD_{\by_n}=D_{n},(R(y))_{y\in\by_n}=(a(y))_{y\in\by_n}\}}$. 
By induction assumption,
we know that the law of $(\cT,(R(y))_{y\in \cL})$ knowing $A_{n}$ is the same as knowing $A_n'$.
Since $X_{n}$ is independent of $(\cT,(R(y))_{y\in \cL}),A_n$, 
the law of $(\cT,(R(y))_{y\in \cL})$ knowing $A_{n+1}$ is the same as knowing $A_{n}$ and thus the same as knowing
$A_{n+1}'$ (which equals $A_{n}'$).

If $y_{n+1}\notin\by_n$, let $x$ be the element of $B_{ry_{n+1}}\cap B_{\by_n}$ the closest to $y_{n+1}$ 
and let $z_n$ be the child of $x$ belonging to $B_{ry_{n+1}}$.
We set $D_n=(D_{n+1}\setminus \TT_{z_n})\cup \{z_n\}$.
Then $A_{n+1}=A_n\cap B_1 \cap B_2$ and $A_{n+1}'=A_n'\cap B_2'$, where $A_n$ and $A_{n}'$ are 
as in the statement and
\begin{align*}
B_1=&\{F(\by_n,D_n,(a(x))_{x\in\by_n},X_n)=z_n\},\\
B_2=&\{x_{n+1}=y_{n+1},\cD_{\{y_{n+1}\}}\cap\TT_{z_n}=D_{n+1}\cap \TT_{z_n},R(y_{n+1})=a(y_{n+1})\},\\
B_2'=&\{y_{n+1}\in \cL, \cD_{\{y_{n+1}\}}\cap\TT_{z_n}=D_{n+1}\cap \TT_{z_n},R(y_{n+1})=a(y_{n+1})\}.
\end{align*}
First, since $X_n$ is independent of everything else, the law of $(\cT,(R(y))_{y\in \cL})$ knowing 
$A_{n+1}$ is the same as knowing $A_n\cap B_2$ (from now on, we take the convention that
in $B_2$, $x_{n+1}$ is the leave resulting from a uniformly random match starting from $z_n$).
We thus only have to prove that the law of $(\cT,(R(y))_{y\in \cL})$ knowing 
$A_{n}\cap B_2$ is the same as knowing $A_n'\cap B_2'$.
Consider $T\in\cS_f$ and $(\alpha(y))_{y\in L_T}\in \{0,1\}^{L_T}$, such that 
$\by_{n+1}\subset L_T$, $D^T_{\by_{n+1}}=D_{n+1}$ and $(\alpha(y))_{y\in \by_{n+1}}=(a(y))_{y\in \by_{n+1}}$.
We have to prove that
\begin{equation*}
\Pr(\cT=T,(R(y))_{y\in L_T}=(\alpha(y))_{y \in L_T})|A_n\cap B_2)
=\Pr(\cT=T,(R(y))_{y\in L_T}=(\alpha(y))_{y \in L_T})|A'_n\cap B'_2).
\end{equation*}
We start from 
\begin{align*}
\Pr(\cT\!=\!T,(R(y))_{y\in L_T}\!=\!(\alpha(y))_{y \in L_T})|A_n\cap B_2)=&
\frac{\Pr(\{\cT\!=\!T,(R(y))_{y\in L_T}\!=\!(\alpha(y))_{y \in L_T})\}\cap B_2|A_n)}{\Pr(B_2|A_n)}\\
=&\frac{\Pr(\{\cT\!=\!T,(R(y))_{y\in L_T}\!=\!(\alpha(y))_{y \in L_T})\}\cap B_2|A_n')}{\Pr(B_2|A_n')}
\end{align*}
thanks to our induction assumption. On the one hand, exactly as in Step 1, we have 
\begin{align*}
&\Pr(\{\cT\!=\!T,(R(y))_{y\in L_T}\!=\!(\alpha(y))_{y \in L_T}\}\cap B_2|A_n')\\
=&\Pr(\cT\!=\!T,(R(y))_{y\in L_T}\!=\!(\alpha(y))_{y \in L_T}, x_{n+1}=y_{n+1}|A_n')\\
=&\Pr(\cT\!=\!T,(R(y))_{y\in L_T}\!=\!(\alpha(y))_{y \in L_T}|A_n')\prod_{u\in B_{z_ny_{n+1}}\setminus\{y_{n+1}\}}
|C^T_u|^{-1}.
\end{align*}
On the other hand,
\begin{align*}
\Pr(B_2|A_n')=&\Pr(B_2'\cap\{x_{n+1}=y_{n+1}\}|A_n')=\Pr(B_2'|A_n')
\!\!\!\!\!\!\prod_{u\in B_{z_ny_{n+1}}\setminus\{y_{n+1}\}}\!\!\!\!\!\! \big|C^{B_{z_ny_{n+1}}\cup (D_{n+1}\cap\TT_{z_n})}_u\big|^{-1}.
\end{align*}
Since $C^T_u=C^{B_{z_ny_{n+1}}\cup (D_{n+1}\cap\TT_{z_n})}_u$ for all $u\in\setminus\{y_{n+1}\}$
(because $D_{\by_{n+1}}^T=D_{n+1}$), we conclude that
\begin{align*}
\Pr(\cT\!=\!T,(R(y))_{y\in L_T}\!=\!(\alpha(y))_{y \in L_T})|A_n\cap B_2)
=&\frac{\Pr(\cT\!=\!T,(R(y))_{y\in L_T}\!=\!(\alpha(y))_{y \in L_T}|A_n')}{\Pr(B_2'|A_n')}.
\end{align*}
Since finally $\{\cT\!=\!T,(R(y))_{y\in L_T}\!=\!(\alpha(y))_{y \in L_T}\}\subset B_2' $, we conclude that
$$
\Pr(\cT\!=\!T,(R(y))_{y\in L_T}\!=\!(\alpha(y))_{y \in L_T})|A_n\cap B_2)
=\Pr(\cT\!=\!T,(R(y))_{y\in L_T}\!=\!(\alpha(y))_{y \in L_T}|A_n'\cap B_2'),
$$
which was our goal.
\end{proof}

We deduce the following observation, that is crucial to our study.

\begin{lem}\label{cru}
Suppose Assumption \ref{as} and 
recall Setting \ref{n1}: an admissible algorithm provided some leaves $\bx_n=\{x_1,\dots,x_n\}$
together with the objects $\cD_{\bx_n}$ and $(R(x))_{x\in\bx_n}$ and we define
$\cF_n=\sigma(\bx_n,\cD_{\bx_n},(R(x))_{x\in\bx_n})$. Recall also Remark \ref{ts}.

(i) Knowing $\cF_n$, for all $x \in \cD_{\bx_n}$, the conditional law of 
$(\cT_x,(R(y))_{y \in \cL \cap \cT_x})$ is $G_{\cK_x,x}$.

(ii) Knowing $\cF_n$, for all $x\in B_{\bx_n}\setminus\bx_n$, the family
$((\cT_u,(R(y))_{y \in \cL \cap \cT_u},u \in \cC_x)$ is independent.
\end{lem}

Recall that for all $x \in \cD_{\bx_n}$, $\cK_x$ is $\cF_n$-measurable and that 
for all $x\in B_{\bx_n}\setminus\bx_n$, $\cC_x$ is $\cF_n$-measurable. Hence this statement is meaningful.

\begin{proof} 
We observe that $\cF_n$ is generated by the events of the form 
$$A_n=\{\bx_n=\by_n,\cD_{\by_n} =D_n,(R(y))_{y\in\by_n}=(a(y))_{y\in\by_n}\}$$ as in Lemma \ref{cruzero}.
Let $A_n'=\{\by_n\subset \cL,\cD_{\by_n} =D_n,(R(y))_{y\in\by_n}=(a(y))_{y\in\by_n}\}$.
We see that on $A_n'$ (which contains $A_n$), $\cK_x=K_x^T$ for all $x\in D_n$ and
$\cC_x=C_x^T$ for all $x\in B_{\bx_n}\setminus\bx_n$, where
$T=B_{\by_n}\cup D_n$.

To check (i), it thus suffices to prove that knowing $A_n$,
for all $x \in D_n$, the law of $(\cT_x,(\!R(y))_{y \in \cL \cap \cT_x}$
is $G_{K^T_x,x}$. We fix $x \in D_n$. 
By Lemma \ref{cruzero}, it suffices to verify that the law of 
$(\cT_x,(R(y))_{y \in \cL \cap \cT_x}$ 
knowing $A_n'$ is $G_{K^T_x,x}$. Recalling that 
$A_{K^T_x}=\{K^T_x\subset \cT,\cD_{K^T_x}=\emptyset\}$, we write
$A_n'=A_{K^T_x}\cap \bigcap_{u \in L_{K^T_x}\setminus\{x\}} E_u$, 
where 
$$E_u=\{ \by_n\cap \TT_u \subset \cL,\cD_{\by_n}\cap \TT_u=D_n\cap \TT_u,
(R(y))_{y\in \by_n \cap \TT_u}=(a(y))_{y\in \by_n \cap \TT_u}\}.$$
By Assumption \ref{as}, $(\cT_x,(R(y))_{y \in \cL \cap \cT_x}$ is independent 
of $\bigcap_{u \in L_{K^T_x}\setminus\{x\}}E_u$ knowing 
$A_{K^T_x}$. Thus the law of $(\cT_x,(\!R(y))_{y \in \cL \cap \cT_x}$
knowing $A_n'$ equals the law of $(\cT_x,(R(y))_{y \in \cL \cap \cT_x}$ knowing $A_{K^T_x}$,
which is $G_{K^T_x,x}$ by definition.

For (ii), we show that
for any $x\in B_{\by_n}\setminus\by_n$, the family 
$((\cT_u,(R(y))_{y \in \cL \cap \cT_u},u \in C_x^T)$ is independent
conditionally on $A_n$, or equivalently, conditionally on $A_n'$.
To this aim, we introduce $S= T \setminus \cup_{y \in C^T_x}(\TT_{y}\setminus\{y\})$ (i.e. $S$ 
is the tree $T$ from which we have removed all the subtrees strictly below the children of $x$).
We write 
$A_n'=A_S\cap \bigcap_{u \in L_S} E_u$ with $E_u$ as in the proof of (i).
We know from Assumption \ref{as} that the family $((\cT_u,(R(y))_{y \in \cL \cap \TT_u},u \in L_{S}) $
is independent conditionally on $A_S$.
Observing that $E_u \in \sigma(\cT_u,(R(y))_{y \in \cL \cap \cT_u})$ for all $u\in L_S$,
we conclude that the family $((\cT_u,(R(y))_{y \in \cL \cap \cT_u},u \in L_S)$
is independent conditionally on $A_n'$. 
(Here we used that if a family of random variables $(X_i)_{i\in I}$ is independent conditionally
on some event $A$ and if we have some events $E_i\in \sigma(X_i)$, for $i\in I$, then
the family $(X_i)_{i\in I}$ is independent conditionally on $A \cap \bigcap_{i\in I} E_i$).
Since $C^T_x\subset L_S$, 
the conclusion follows.
\end{proof}

\section{Proof of the main result}\label{pprroo}

In the whole section, we take Assumption \ref{as} for granted.
We first compute the conditional minimax values.

\begin{proof}[Proof of Proposition \ref{rn}]
We work under Setting \ref{n1}.
If first $x\in \bx_n$, then $R(x)$ is $\cF_n$-measurable, so that
$R_n(x)=\Pr(R(x)=1|\cF_n)=R(x)$.

Next, for $x \in \cD_{\bx_n}$, we know from Lemma \ref{cru} that 
the law of $(\cT_x,(R(y))_{y \in \cL \cap \cT_x}$ knowing $\cF_n$ is $G_{\cK_x,x}$.
Recalling Definition \ref{dfms} and Remark \ref{ts}, we see that $R_n(x)=\Pr(R(x)=1|\cF_n)=
m(\cK_x,x)$.

Finally, for $x \in B_{\bx_n} \setminus \bx_n$, Lemma \ref{cru}
tells us that the family $((\cT_y,(R(u))_{u \in \cL \cap \cT_y},y \in \cC_x)$ is independent
conditionally on $\cF_n$.
But for $y \in \cC_x$, $R(y)$ is of course $\sigma(\cT_y,(R(u))_{u \in \cL \cap \cT_y})$-measurable
(recall Remark \ref{mm}). Thus the family $(R(y),y \in \cC_x)$ is independent
conditionally on $\cF_n$.
If $t(x)=0$, we may write, by \eqref{rr},
\begin{align*}
R_n(x)=\Pr(R(x)=1 | \cF_n)=
\Pr(\min\{R(y):y \in \cC_x\}=1\vert \cF_n)
=\prod_{y \in \cC_x} \Pr(R(y)=1\vert \cF_n),
\end{align*}
which equals $\prod_{y \in\cC_x} R_n(y)$ as desired. If now $t(x)=1$, we find similarly
\begin{align*}
R_n(x)=\Pr(R(x)=1 | \cF_n)
=\Pr(\max\{R(y):y \in\cC_x\}=1\vert \cF_n)
=1-\prod_{y \in \cC_x} \Pr(R(y)=0\vert \cF_n),
\end{align*}
which is nothing but $1-\prod_{y \in\cC_x}(1-R_n(y))$.
\end{proof}

We now study the different possibilities
for the $(n+1)$-th step. 

\begin{prop}\label{minoumax}
Adopt Setting \ref{n1} and Notation \ref{n2}. 
For $z\in\cD_{\bx_n}$ and $x\in B_{\bx_n}\cup\cD_{\bx_n}$, we set
$$
\Delta_n^z(x)=\E[(R_{n+1}^z(x)-R_n(x))^2\vert\cF_n].
$$

(i) We have 
$\arg\!\min\{ \E[(R_{n+1}^z(r)-R(r))^2\vert \cF_n] : z\in \cD_{\bx_n}\}
=\arg\!\max\{ \Delta_n^z(r) : z\in \cD_{\bx_n}\}$.

(ii) For all $z\in\cD_{\bx_n}$, we have $\Delta_n^z(z)=s(\cK_z,z)$.

(iii) For all $z\in\cD_{\bx_n}$, all $x\in B_{rz}\setminus\{r\}$,
we have $\Delta_n^z(f(x))=(U_n(x))^2 \Delta_n^z(x)$.

(iv) For all $z\in \cD_{\bx_n}$, $\Delta_n^z(r)=s(\cK_z,z) \prod_{y\in B_{rz}\setminus\{r\}} (U_n(y))^2$.
\end{prop}

\begin{proof}
Point (i) is not difficult: for all $z \in \cD_{\bx_n}$,
\begin{align*}
\E[(R_{n}(r)-R(r))^2\vert\cF_n]=& \E[(R_{n+1}^z(r)-R(r))^2\vert\cF_n] 
+ \E[(R_{n+1}^z(r)-R_n(r))^2\vert\cF_n]\\
&+2\E[(R_{n+1}^z(r)-R(r))(R_n(r)-R_{n+1}^z(r))\vert\cF_n].
\end{align*}
Using that $\cF_n \subset \cF_{n+1}^z$ and that $R_{n+1}^z(r)=\E[R(r)|\cF_{n+1}^z]$, we conclude that
$$
\E[(R_{n}(r)-R(r))^2\vert\cF_n]= \E[(R_{n+1}^z(r)-R(r))^2\vert\cF_n] 
+ \E[(R_{n+1}^z(r)-R_n(r))^2\vert\cF_n].
$$
Since the left hand side does not depend on $z$, minimizing $\E[(R_{n+1}^z(r)-R(r))^2\vert\cF_n]$
is equivalent to maximizing $\E[(R_{n+1}^z(r)-R_n(r))^2\vert\cF_n]$.

Also, point (iv) immediately follows from points (ii) and (iii). 

To check points (ii) and (iii), we fix $z\in\cD_{x_n}$. We recall that
$\cF_{n+1}^z=\cF_n \lor \cG$, where $\cG=\sigma(y,\cK_{zy},R(y))$, where $y$ is the leave resulting
from a uniformly random match starting from $z$ and where 
$\cK_{zy}=\cK_y \cap \TT_z$. We also recall
that $R_{n+1}^z(x)=\Pr(R(x)=1|\cF_{n+1}^z)$ for all $x \in B_{\bx_n}\cup\cD_{\bx_n}$.

We know from Lemma \ref{cru} that 
the law of $(\cT_z,(R(y))_{y \in \cL \cap \cT_z})$ knowing $\cF_n$ is $G_{\cK_z,z}$.
Recalling Definition \ref{dfms} and Remark \ref{ts}, we immediately deduce that
$R_n(z)=m(\cK_z,z)$ and that 
$$
\Delta_n^z(z)=\E[(R_{n+1}^z(z)-R_n(z))^2\vert\cF_n]=s(\cK_z,z).
$$
This proves (ii). To prove (iii), we fix $x \in B_{rz}\setminus \{r\}$ and 
we set $v=f(x)$. By Lemma \ref{cru},
the family $((\cT_y,(R(u))_{u \in \cL \cap \cT_y}),y \in \cC_v)$ is independent
conditionally on $\cF_n$. Furthermore, $\cG$, which only concerns 
$(\cT_x,(R(u))_{u \in \cL \cap \cT_x})$ is independent of the family 
$((\cT_y,(R(u))_{u \in \cL \cap \cT_y}),y \in \cH_x)$. Finally,
we recall that for all $y \in \cC_v$, $R(y)$ is $\sigma(\cT_y,(R(u))_{u \in \cL \cap \cT_y})$-measurable.

If $t(v)=0$,
$$
R_{n+1}^z(v)=\Pr(R(v)=1 | \cF_n\lor \cG)
=\Pr(\min\{R(y):y \in \cC_v\}=1 \vert \cF_n \lor \cG)
$$
whence, by conditional independence, $R_{n+1}^z(v)= \prod_{y \in \cC_v} 
\Pr(R(y)=1\vert \cF_n\lor \cG)$ and thus
\begin{align*}
R_{n+1}^z(v)=
\Big(\prod_{y \in \cH_x} \Pr(R(y)=1\vert \cF_n)\Big)\Pr(R(x)=1 | \cF_n\lor \cG)
=\Big(\prod_{y \in \cH_x} R_n(y)\Big)R_{n+1}^z(x),
\end{align*}
whence $R_{n+1}^z(v)=U_n(x)R_{n+1}^z(x)$.
Since now $R_n(v)=\prod_{y \in \cC_v} R_n(y)=U_n(x)R_n(x)$ by Proposition
\ref{rn}, we conclude that $(R_{n+1}(v)-R_n(v))^2= U_n^2(x)(R_{n+1}^z(x)-R_n(x))^2$, whence,
taking expectations conditionally on $\cF_n$, $\Delta^z_{n}(v)=U_n^2(x)\Delta^z_{n}(x)$.

If next $t(v)=1$,
\begin{align*}
R_{n+1}^z(v)=&\Pr(R(v)=1 | \cF_n\lor \cG)
=\Pr(\max\{R(y):y \in \cC_v\}=1\Big\vert \cF_n \lor \cG),
\end{align*}
so that $R_{n+1}^z(v)=1-\prod_{y \in \cC_v} \Pr(R(y)=0\vert \cF_n\lor \cG)$. We conclude that
\begin{align*}
R^z_{n+1}(v)=&
1 - \Big(\prod_{y \in \cH_x} \Pr(R(y)=0\vert \cF_n)\Big)\Pr(R(x)=0 | \cF_n\lor \cG)\\
=&1- \Big(\prod_{y \in \cH_x} (1-R_n(y))\Big) (1-R_{n+1}^z(x)),
\end{align*}
so that $R^z_{n+1}(v)= 1- U_n(x)(1-R_{n+1}^z(x))$.
Since finally $R_n(v)=1-\prod_{y \in \cC_v} (1-R_n(y))=1-U_n(x)(1-R_n(x))$ by Proposition
\ref{rn}, $(R_{n+1}^z(v)-R_n(v))^2= U_n^2(x)(R_{n+1}^z(x)-R_n(x))^2$, whence,
taking expectations conditionally on $\cF_n$,
$\Delta^z_{n}(v)=U_n^2(x)\Delta^z_{n}(x)$.
\end{proof}

We now have all the weapons to give the

\begin{proof}[Proof of Theorem \ref{mr}]
Recall that we work under Setting \ref{n1} and that we adopt Notation \ref{n2}
in which $U_n$, $Z_n$ and $\cF^{n+1}_z$ are defined.
For $x\in B_{\bx_n}\cup\cD_{\bx_n}$, we set $\bar U_n(x)=\prod_{B_{rx}\setminus\{r\}} U_n(y)$,
with the convention that $\bar U_n(r)=1$.

{\it Step 1.} In view of the explicit formula for $\Delta_n^z(r)$ checked in Proposition \ref{minoumax}-(iv), the
natural way to find $z_*$ maximizing $\Delta_n^z(r)$ is to start from $r$ and to go down in 
$B_{\bx_n} \cup\cD_{\bx_n}$ following the maximum values of $(N_n(x))_{x\in B_{\bx_n} \cup\cD_{\bx_n}}$ defined as follows.
Set $N_n(x)=0$ for $x\in \bx_n$, set $N_n(x)=(\bar U_n(x))^2 s(\cK_x,x)$ 
for $x\in\cD_{\bx_n}$ and put
$N_n(x)=\max\{N_n(y):y \in \cC_x\}$ for $x \in B_{\bx_n}\setminus\bx_n$.

We claim that $N_n(x)=(\bar U_n(x))^2 Z_n(x)$ for all $x\in B_{\bx_n}\cup\cD_{\bx_n}$.

Indeed, set $\tN_n(x)=(\bar U_n(x))^2 Z_n(x)$ and recall Notation \ref{n2}-(ii).
We obviously have 
$\tN_n(x)=N_n(x)$ for $x\in \bx_n$ (because then $Z_n(x)=0$) and for $x\in\cD_{\bx_n}$
(because then $Z_n(x)=s(\cK_x,x)$). 
And for $x\in B_{\bx_n}\setminus\bx_n$,
we have $\tN_n(x)=(\bar U_n(x))^2\max \{(U_n(y))^2Z_n(y) : y\in\cC_x\}
=\max \{\tN_n(y) : y\in\cC_x\}$, because for $y \in \cC_x$, we have
$\bar U_n(x) U_n(y)=\bar U_n(y)$. The claim follows by (backward) induction.

{\it Step 2.} We define $z^* \in \cD_{\bx_n}\cup\bx_n$ as follows. Put $y_0=r$. Find
$y_1=\arg\!\max\{N_n(y):y\in\cC_{y_0}\}$. If $y_1 \in \cD_{\bx_n}\cup\bx_n$, set $z^*=y_1$.
Else, put $y_2=\arg\!\max\{N_n(y):y\in\cC_{y_1}\}$. If $y_2 \in \cD_{\bx_n}\cup\bx_n$, 
set $z^*=y_2$.
Else, put $y_3=\arg\!\max\{N_n(y):y\in\cC_{y_2}\}$, etc.

By construction, $z^*=\arg\!\max\{N_n(z): z \in \cD_{\bx_n}\cup\bx_n\}$.
Also, $N_n(x)=N_n(r)$ for all $x \in B_{rz^*}$.

{\it Step 3.} We recall that  $z_* \in \cD_{\bx_n}\cup\bx_n$ was defined, similarly, as follows: put $y_0=r$ and find
$y_1=\arg\!\max\{U_n^2(y)Z_n(y):y\in\cC_{y_0}\}$. If $y_1 \in \cD_{\bx_n}\cup\bx_n$, set 
$z_*=y_1$. Else, put $y_2=\arg\!\max\{U_n^2(y)Z_n(y):y\in\cC_{y_1}\}$. 
If $y_2 \in \cD_{\bx_n}\cup\bx_n$, set $z_*=y_2$.
Else, put $y_3=\arg\!\max\{U_n^2(y)Z_n(y):y\in\cC_{y_2}\}$, etc.

{\it Step 4.} We now prove that if $N_n(r)>0$, then $z_*=z^*$. 

First observe that 
for any $x\in B_{\bx_n}\setminus\bx_n$ such that $\bar U_n(x)>0$,
\begin{align}\label{ttaacc}
\arg\!\max\{(U_n(y))^2Z_n(y):y \in \cC_x\}
=&\arg\!\max\{ (\bar U_n(x))^2(U_n(y))^2Z_n(y):y \in \cC_x\}\\
=&\arg\!\max\{ (\bar U_n(y))^2Z_n(y):y \in \cC_x \}\notag\\
=&\arg\!\max\{ N_n(y):y \in \cC_x \}.\notag
\end{align}
Furthermore, we have $N_n(x)=N_n(r)$ for all $x\in B_{rz^*}$. Hence if 
$N_n(r)>0$, then $\bar U_n(x)>0$ for all $x\in B_{rz^*}$ (recall that $N_n(x)=(\bar U_n(x))^2 Z_n(x)$). Consequently, 
\eqref{ttaacc} holds true during the whole computation of 
$z^*$, so that $z_*=z^*$.

{\it Step 5.} 
By Steps 2 and 4, $z_*=z^*=\arg\!\max\{N_n(z): z \in \cD_{\bx_n}\cup\bx_n\}$ on $\{N_n(r)>0\}$.
And we know from Proposition \ref{toujconv} and Lemma \ref{msok} that on $\{R_n(r)\notin \{0,1\}\}$,
$z_* \in \cD_{\bx_n}$. We also know that $R_n(r)\notin \{0,1\}$ implies that
$Z_n(r)>0$ (see Step 2 of the proof of Proposition \ref{toujconv}), whence $N_n(r)=Z_n(r)>0$
(recall Step 1 and that $\bar U_n(r)=1$).
Thus on $\{R_n(r)\notin \{0,1\}\}$, we have $z_*\in\cD_{\bx_n}$ and 
$z_*=\arg\!\max\{N_n(z): z \in \cD_{\bx_n}\}=\arg\!\max\{(\bar U_n(z))^2 s(\cK_z,z): z \in \cD_{\bx_n}\}$
by definition of $N_n$, see Step 1.
By Proposition
\ref{minoumax}-(iv), we conclude that on $\{R_n(r)\notin \{0,1\}\}$
\begin{align*}
z_*=
\arg\!\max\{\E[(R_{n+1}^z(r)-R_n(r))^2\vert \cF_n] : z\in \cD_{\bx_n}\},
\end{align*}
which equals $\arg\!\min\{\E[(R_{n+1}^z(r)-R(r))^2\vert \cF_n] : z\in \cD_{\bx_n}\}$
by Proposition \ref{minoumax}-(i). We have verified that on $\{R_n(r)\notin \{0,1\}\}$,
$z_* \in \cD_{\bx_n}$ and $z_*=\arg\!\min\{\E[(R_{n+1}^z(r)-R(r))^2\vert \cF_n] : z\in \cD_{\bx_n}\}$,
which was our goal.
\end{proof}

\begin{rk}\label{update}
As seen in the proof, the natural way to find $z_*$ would be to start from $r$
and to go down in the tree following the highest values of $N_n$.
Recalling the discussion of Subsection \ref{qu},
this would lead to an algorithm with cost of order $Kdn^2$: since (generally)
$R_{n+1}(x)\neq R_n(x)$ for $x \in B_{rx_{n+1}}$, this (generally) modifies the value of 
$\bar U_n(x)$ for all $x\in B_{\bx_n}\setminus B_{rx_{n+1}}$
(actually, except for $x\in B_{\bx_n}\cap B_{rx_{n+1}}$) and thus the values of $N_n(x)$ for all
$x$ on the whole explored tree $B_{\bx_n}\cup\cD_{\bx_n}$. 
The observation \eqref{ttaacc}, which asserts that $\arg\!\max\{N_n(y) : y \in\cC_x\}=
\arg\!\max\{(U_n(y))^2Z_n(y):y \in \cC_x\}$ is thus crucial, 
as well as the fact that $U_n$ and $Z_n$ enjoy a quick update property.
\end{rk}

\section{Computation of the parameters for a few specific models}\label{comput}

Here we present a few models for the tree $\cT$ and the outcomes $(R(x))_{x\in\cL}$
where our assumptions are met and where we can compute, at least numerically, the functions
$m$ and $s$ introduced in Definition \ref{dfms}. Recall that these functions
are necessary to implement Algorithm \ref{ouralgo}.

\subsection{Inhomogeneous Galton-Watson trees}\label{stm1}

We assume that $\cT$ is the realization of an inhomogeneous Galton-Watson tree
with reproduction laws $\mu_0,\dots,\mu_K$: the number of children of the root $r$ follows the law
$\mu_0\in\cP(\nn)$, the number of children of these children are independent and $\mu_1$-distributed, etc.
We assume that
$\mu_K=\delta_0$, so that any individual of generation $K$ is a leave and thus $K$ is the maximal
depth of $\cT$. 

We also consider a family $q_0,\dots,q_K$ of numbers in $[0,1]$.
Conditionally on $\cT$, we assume that the family
$(R(x))_{x\in\cL}$ is independent and that  $R(x)\sim$ Bernoulli$(q_{|x|})$ for all $x\in\cL$.

\begin{ex}\label{tm1}
With such a model, Assumption \ref{as} is fulfilled, and for any $S\in S_f$ and $x \in L_S$
such that $\Pr(A_S)>0$, we have 
$m(S,x)=\cm(|x|)$ and $s(S,x)=\cs(|x|)$, where 

(i) $\cm$ is defined by backward induction by $\cm(K)=q_K$ and, for $k=0,\dots,K-1$,
$$
\cm(k)=\mu_k(0)q_k+ \sum_{\ell\geq 1}\mu_k(\ell) \Big(\indiq_{\{k \hbox{ \tiny is odd}\}}(\cm(k+1))^\ell 
+ \indiq_{\{k \hbox{ \tiny is even}\}}(1-[1-\cm(k+1)]^\ell)
\Big),
$$

(ii) $\cs$ is defined by backward induction by $\cs(K)=q_K(1-q_K)$ and, 
for $k=0,\dots,K-1$,
\begin{align*}
\cs(k)=&\mu_k(0)[q_k(1-\cm(k))^2+(1-q_k)(\cm(k))^2]\\
& +\indiq_{\{k \hbox{ \tiny is odd}\}} \sum_{\ell\geq 1}\mu_k(\ell) \Big[(\cm(k+1))^{2\ell-2}\cs(k+1) + 
\big((\cm(k+1))^\ell-\cm(k)\big)^2\Big] \\
& + \indiq_{\{k \hbox{ \tiny is even}\}}\sum_{\ell\geq 1}\mu_k(\ell) \Big[(1-\cm(k+1))^{2\ell-2}\cs(k+1) + 
\big((1-\cm(k+1))^\ell-(1-\cm(k))\big)^2\Big].
\end{align*}
\end{ex}

These quantities can be computed once for all if one knows the parameters
$\mu_0,\dots,\mu_K$ and $q_0,\dots,q_K$ of the model. If unknown, as is generally the case, these parameters 
can be evaluated
numerically quite precisely, handling an important number of uniformly random matches. From these
evaluations, we can derive some approximations of $\cm$ and $\cs$.
However, the main problem is of course that in general, assuming that the true game is the realization
of such a model is not very realistic.

\begin{proof}
First, Assumption \ref{as} is satisfied, thanks to the classical branching property of Galton-Watson
trees. Indeed,
consider $S\in\cS_f$ and $x\in L_S$ such that $\Pr(A_S)>0$. Conditionally on $A_S$,
we can write $\cT=S \cup \bigcup_{x\in L_S}\cT_x$ and
the family $((\cT_x,(R(y))_{y\in\cL \cap \cT_x}, x\in L_S)$ is independent by construction.
Furthermore, for any $x\in L_S$, the law $G_{S,x}$ of  $(\cT_x,(R(y))_{y\in\cL \cap \cT_x}$ knowing $A_S$
depends only on the depth $|x|$.

Consequently, there are $(\cm(k))_{k=0,\dots,K}$ and $(\cs(k))_{k=0,\dots,K}$ such that
for $S\in\cS_f$ and $x\in L_S$ with $\Pr(A_S)>0$, $m(S,x)=\cm(|x|)$ and $s(S,x)=\cs(|x|)$.

If $|x|=K$, then $x$ is necessarily a leave, so that $\Pr(x\in\cL|A_S)=1$, whence, by Lemma \ref{hor},
$\cm(K)=m(S,x)=\Pr(x\in\cL,R(x)=1|A_S)=q_K$ and
$$
\cs(K)=s(S,x)=\Pr(x\in\cL,R(x)=1|A_S)(1-m(S,x))^2+\Pr(x\in\cL,R(x)=0|A_S)(m(S,x))^2,$$
which equals $q_K(1-q_K)^2+(1-q_K)q_K^2=q_K(1-q_K)$ as desired.

Finally, the proof can be completed by using Lemma \ref{hor} and that 
if $|x|=k\in\{0,\dots,K-1\}$ and if for example $t(x)=0$ (i.e. $k$ is odd),
for any $\ell\geq1$,

$\bullet$ $\Pr(x\in\cL,R(x)=1|A_S)=\mu_k(0)q_k$ and $\Pr(x\in\cL,R(x)=0|A_S)=\mu_k(0)(1-q_k)$,

$\bullet$ $\sum_{\by\subset \CC_x,|\by|=\ell}\Pr(\cC_x=\by|A_S)=\mu_k(\ell)$,

$\bullet$ $\Theta(S,x,\by)$ and $\Gamma(S,x,\by,y)$
depend only on $k$ and on $\ell=|\by|$ and (since $t(x)=0$), 
$\Theta(S,x,\by)=[\cm(k+1)]^{\ell}$ and $\Gamma(S,x,\by,y) = \cs(k+1) [\cm(k+1)]^{2\ell-2}+ 
[(\cm(k+1))^\ell - \cm(k)]^2$.
\end{proof}

\begin{rk}\label{pearl}
In Pearl's model \cite{p}, $\cT$ is the deterministic regular tree with degree $d\geq 2$ and depth $K\geq 1$
and the family $(R(x))_{x\in\cL}$ is i.i.d. Bernoulli$(p)$-distributed.
This is a particular case of Example \ref{tm1} with 
$\mu_0=\dots=\mu_{K-1}=\delta_d$ and $\mu_K=\delta_0$ and $q_K=p$ (the values of $(q_k)_{k=0,\dots,K-1}$ being
irrelevant). One thus finds $\cm(K)=p$, $\cs(K)=p(1-p)$ and, for $k=0,\dots,K-1$,
\begin{align*}
\cm(k)=&\indiq_{\{k \hbox{ \tiny is odd}\}}(\cm(k+1))^d+\indiq_{\{k \hbox{ \tiny is even}\}}(1-[1-\cm(k+1)]^d),\\
\cs(k)=&\indiq_{\{k \hbox{ \tiny is odd}\}}(\cm(k+1))^{2d-2}\cs(k+1)
+\indiq_{\{k \hbox{ \tiny is even}\}}(1-\cm(k+1))^{2d-2}\cs(k+1).
\end{align*}
\end{rk}

\subsection{Inhomogeneous Galton-Watson trees of order two}\label{stm2}

Here we mention that we can also deal
with random trees that enjoy some independence properties without being Galton-Watson trees. 
For example, the following model \emph{of order $2$} allows one to build a broad 
class of random trees
with non-increasing degree (along each branch), which might be useful for real games. 
It is possible
to treat some models of higher order, but the functions $m$ and $s$ then become really 
tedious to compute 
theoretically and to approximate in practice.

We consider a family of probability measures on $\nn$: $\mu_0$ and $\mu_{k,d}$ for $k=1,\dots,K$ and $d\geq 1$.
We assume that $\mu_{K,d}=\delta_0$ for all $d\geq 1$
and $K$ will represent the maximum depth of the tree. 

We build the random tree $\cT$ as follows:
the root has $D_r\sim \mu_0$ children. Conditionally on $D_r$,
all the children $x$ of the root produce, independently, a number $D_x\sim \mu_{1,D_r}$ of children.
Once everything is built up to generation $k\in\{0,\dots,K-1\}$,
all the individuals $x$ with $|x|=k$ produce, independently (conditionally on what is 
already built), a number $D_x\sim \mu_{k,D_{f(x)}}$ of children. 

We also consider a family $q_0,\dots,q_K$ of numbers in $[0,1]$.
Conditionally on $\cT$, we assume that the family
$(R(x))_{x\in\cL}$ is independent and that $R(x)\sim$ Bernoulli$(q_{|x|})$ for all $x\in\cL$.

\begin{ex}\label{tm2}
With such a model, Assumption \ref{as} is fulfilled, and for any $S\in S_f$ (with $\{r\}\subsetneq S$)
such that $\Pr(A_S)>0$ and any $x\in L_S$, we have 
we have $m(S,x)=\cm(|x|,|C^S_{f(x)}|)$ and $s(S,x)=\cs(|x|,|C^S_{f(x)}|)$, where 
$\cm$ and $\cs$ can be computed by
backward induction as follows.

(i) $\cm(K,d)=q_K$ for all $d\geq 1$ and, for $k=1,\dots,K-1$ and $d\geq 1$,
$$
\cm(k,d)=\mu_{k,d}(0)q_k+ \sum_{\ell\geq 1}\mu_{k,d}(\ell) \Big[\indiq_{\{k \hbox{ \tiny is odd}\}}
(\cm(k+1,\ell))^\ell 
+ \indiq_{\{k \hbox{ \tiny is even}\}}(1-[1-\cm(k+1,\ell)]^\ell)
\Big].
$$

(ii) $\cs(K,d)=q_K(1-q_K)$ for all $d\geq 1$ and, 
for $k=1,\dots,K-1$ and $d\geq 1$,
\begin{align*}
\cs(&k,d)=\mu_{k,d}(0)[q_k(1-\cm(k,d))^2+(1-q_k)(\cm(k,d))^2]\\
& +\indiq_{\{k \hbox{ \tiny is odd}\}} \sum_{\ell\geq 1}\mu_{k,d}(\ell) 
\Big[(\cm(k+1,\ell))^{2\ell-2}\cs(k+1,\ell) + 
\big((\cm(k+1,\ell))^\ell-\cm(k,d)\big)^2\Big] \\
& + \indiq_{\{k \hbox{ \tiny is even}\}}\sum_{\ell\geq 1}\mu_{k,d}(\ell) 
\Big[(1-\cm(k+1,\ell))^{2\ell-2}\cs(k+1,\ell) + 
\big((1-\cm(k+1,\ell))^\ell-(1-\cm(k,d))\big)^2\Big].
\end{align*}
\end{ex}

We could easily express $m(\{r\},r)$ and $s(\{r\},r)$, but these values are useless as far as
Algorithm \ref{ouralgo} is concerned.

\begin{proof} First, Assumption \ref{as} is satisfied.
Indeed, consider $S\in\cS_f$ and $x\in L_S$ such that $\Pr(A_S)>0$. Conditionally on $A_S$,
we can write $\cT=S \cup \bigcup_{x\in L_S}\cT_x$ and
the family $((\cT_x,(R(y))_{y\in\cL \cap \cT_x}, x\in L_S)$ is independent by construction.
Furthermore, for any $x\in L_S$, the law $G_{S,x}$ of  $(\cT_x,(R(y))_{y\in\cL \cap \cT_x}$ knowing $A_S$
depends only on the depth $|x|$ and of $|C_{f(x)}^S|$ (except if $S=\{r\}$ and $x=r$).

Consequently, there are $(\cm(k,d))_{k=1,\dots,K,d\geq 1}$ and $(\cs(k,d))_{k=1,\dots,K,d\geq 1}$ such that
for $S\in\cS_f$ and $x\in L_S$ with $|C^S_{f(x)}|=d$,
$\Pr(A_S)>0$, $m(S,x)=\cm(|x|,d)$ and $s(S,x)=\cs(|x|,d)$.

If $|x|=K$, then $x$ is necessarily a leave, so that $\Pr(x\in\cL|A_S)=1$, whence, by Lemma \ref{hor},
$\cm(K,d)=m(S,x)=\Pr(x\in\cL,R(x)=1|A_S)=q_K$ and
$$
\cs(K,d)=s(S,x)=\Pr(x\in\cL,R(x)=1|A_S)(1-m(S,x))^2+\Pr(x\in\cL,R(x)=0|A_S)(m(S,x))^2,$$
which equals $q_K(1-q_K)^2+(1-q_K)q_K^2=q_K(1-q_K)$ as desired.

Finally, the proof can be completed by using Lemma \ref{hor} and that 
if $|x|=k\in\{0,\dots,K-1\}$ (and $|C^S_{f(x)}|=d$) and if for example $t(x)=0$ (i.e. $k$ is odd),
for any $\ell\geq 1$,

\noindent $\bullet$ $\Pr(x\in\cL,R(x)=1|A_S)=\mu_{k,d}(0)q_k$ and $\Pr(x\in\cL,R(x)=0|A_S)=\mu_{k,d}(0)(1-q_k)$,

\noindent $\bullet$ $\sum_{\by\subset \CC_x,|\by|=\ell}\Pr(\cC_x=\by|A_S)=\mu_{k,d}(\ell)$,

\noindent $\bullet$ $\Theta(S,x,\by)$ and $\Gamma(S,x,\by,y)$
depend only on $k,d$ and $\ell=|\by|$ and, if e.g. $t(x)=0$ (i.e. $k$ is odd), 
$\Theta(S,x,\by)=[\cm(k+1,\ell)]^{\ell}$ and $\Gamma(S,x,\by,y) = \cs(k+1,\ell) [\cm(k+1,\ell)]^{2\ell-2}
+ {[(\cm(k+1,\ell))^\ell - \cm(k,d)]^2}$.

The last point uses that if $\cC_x=\by$ with $|\by|=\ell$, then $|\cC_{f(y)}|=|\cC_x|=\ell$
for all $y \in\by$.
\end{proof}

\subsection{Symmetric minimax values}\label{stm3}
Here we discuss the formulas introduced in Subsection \ref{pchoice}.

We fix some value $a\in (0,1)$. For $S\in\cS_f$, we build the family 
$(m_a(S,x))_{x\in S}$ by induction, starting from the root, setting $m_a(S,r)=a$ and, 
for all $x\in S\setminus \{r\}$,
\begin{equation}\label{q}
m_a(S,x)=\indiq_{\{t(f(x))=0\}} [m_a(S,f(x))]^{1/|C^S_{f(x)}|}+\indiq_{\{t(f(x))=1\}}
(1-[1-m_a(S,f(x))]^{1/|C^S_{f(x)}|}).
\end{equation}
Observe that $m_a(S,x)$ actually depends only on $K_x^S$, i.e. $m_a(S,x)=m_a(K^S_x,x)$

\begin{ex}\label{tm3}
Consider a possibly random tree $\cT\in \cS_f$ enjoying the property that for any 
$S\in\cS_f$ with leaves $L_S$, the family
$(\cT_x)_{x\in L_S}$ is independent conditionally on $A_S=\{S\subset \cT,\cD_{L_S}=\emptyset\}$
as soon as $\Pr(A_S)>0$.
Fix $a\in (0,1)$ and assume that conditionally on $\cT$,
the family $(R(y))_{y\in \cL}$ is independent and $R(y)\sim$ Bernoulli$(m_a(\cT,y))$ 
for all $y\in\cL$.
Then Assumption \ref{as} is fulfilled, and 
for any $S\in S_f$ such that $\Pr(A_S)>0$ and any $x\in L_S$,
we have $m(S,x)=m_a(S,x)$. We are generally not able to compute $s(S,x)$.
\end{ex}

Observe that this is a qualitative \emph{symmetry} assumption, 
saying that knowing $\cT$, for any $v\in\cT\setminus \cL$,
the family of the minimax values $(R(x),x \in \cC_v)$ is i.i.d. Once this is assumed,
the only remaining parameter is the mean minimax rating
of the root (which we set to $a$).

Once the value $a=m_a(\cT,r)$ is chosen (even if not knowing $\cT$), it is easy to make the algorithm 
compute the necessary values of $m_a$, as explained in Subsection \ref{pchoice}: 
each time a new node $x$ of $\cT$ is created by the algorithm,
we can compute $m_a(\cK_x,x)$ from $m_a(\cK_{f(x)},f(x))$ and $|\cC_{f(x)}|$.

\begin{proof}
Assumption \ref{as} is satisfied because (a) the 
random tree $\cT$ is supposed to satisfy the required independence property and (b) conditionally on $\cT$,
for any $x\in\cL$, $m_a(\cT,x)$ depends only on $\cK_x$. 

It remains to verify that $m(S,x)=m_a(S,x)$
for all $S\in\cS_f$ of which $x$ is a leave.

We first show by backward induction that $\Pr(R(x)=1|\cT)=m_a(\cT,x)$ for all $x\in \cT$. 
First, this is obvious
if $x\in \cL$ by construction. Next, if this is true for all the children (in $\cT$) of
$x\in\cT\setminus\cL$ with e.g. $t(x)=1$, then we have 
$\Pr(R(x)=1|\cT)=1-\prod_{y\in\cC_x}\Pr(R(y)=0|\cT)
=1-\prod_{y\in\cC_x}(1-m_a(\cT,y))$. We first used that the family $(R(y))_{y\in\cC_x}$
is independent conditionally on $\cT$ and then the induction assumption.
Using finally \eqref{q} (recall that $t(x)=1$ and that $f(y)=x$ for all $y\in\cC_x$), we find
$\Pr(R(x)=1|\cT)=1-\prod_{y\in\cC_x}(1-[1-(1-m_a(\cT,x))^{1/|\cC_x|}])=m_a(\cT,x)$.

Fix now $S\in S_f$ such that $\Pr(A_S)>0$, where $A_S=\{S\subset\cT, \cD_{L_S}=\emptyset\}$ and $x\in L_S$.
Since $A_S\in\sigma(\cT)$, we deduce that $m(S,x)=\Pr(R(x)=1|A_{S})=\E[\Pr(R(x)=1|\cT)|A_S]
=\E[m_a(\cT,x)|A_S]$. But we know that $m_a(\cT,x)=m_a(\cK_x,x)$. Since $\cK_x=K_x^S$ on $A_S$,
we conclude that $m(S,x)=m_a(K_x^S,x)=m_a(S,x)$ as desired.
\end{proof}

Let us mention that Pearl's model, which we already interpreted as a particular case of
Example \ref{tm1}, can also be seen as a particular case of Example \ref{tm3}, where we can furthermore
compute $s$.

\begin{rk}\label{pearl2} 
Consider again Pearl's model \cite{p}: $\cT$ is the deterministic regular tree with degree $d\geq 2$ 
and depth $K\geq 1$ and the family $(R(x))_{x\in\cL}$ is i.i.d. Bernoulli$(p)$-distributed.
Then we already know that for all $S\in \cS_f$ and $x \in L_S$ such that
$\Pr(A_S)>0$, we have $m(S,x)=\cm(|x|)$ and $s(S,x)=\cs(|x|)$, with $\cm$ and $\cs$ as in 
Remark \ref{pearl}. 
One then also has, for all
$S\in \cS_f$ and $x \in L_S$ such that $\Pr(A_S)>0$, if $x\neq r$, denoting by $v=f(x)$
and $S_v=S\setminus C^S_v$,
\begin{align*}
m(S,x)=&\indiq_{\{t(v)=0\}}[m(S_v,v)]^{1/|C^S_{v}|}+\indiq_{\{t(v)=1\}}
(1-[1-m(S_v,v)]^{1/|C^S_v|}),\\
s(S,x)=&\Big(\indiq_{\{t(v)=0\}}m(S_v,v)+\indiq_{\{t(v)=1\}}[1-m(S_v,v)])\Big)
^{2(|C^S_v|-1)}\!\!s(S_v,v).
\end{align*}
Setting $a=\cm(0)$, which can be computed from $p,K,d$, we thus have $m(S,x)=m_a(S,x)$ as defined in \eqref{q},
and we can compute $s(S,x)$.
Note that it is not necessary to determine precisely $s(\{r\},r)$: we can set
$s(\{r\},r)=1$ (or any other positive constant) by Remark \ref{nvi}.
\end{rk}

Indeed, the above formulas are nothing but a complicated version of the ones in Remark \ref{pearl}, since
we have $m(S,x)=\cm(|x|)$, $s(S,x)=\cs(|x|)$, $|C^S_{v}|=d$, $m(S_v,v)=\cm(|v|)$ and  $s(S_v,v)=\cs(|v|)$.

There are other cases where we can characterize $s$, which should thus be numerically computable.

\begin{ex}\label{tm4}
Assume that $\cT$ is a homogeneous Galton-Watson tree with reproduction 
law $\mu$ such that 
$\sum_{\ell \geq 1} \ell \mu(\ell)\leq 1$ and $\mu(0)>0$, so that $\cT$ is a.s. finite.
Fix $a\in (0,1)$ and assume that conditionally on $\cT$, the family $(R(y))_{y\in \cL}$ 
is independent and that
$R(y)\sim$ Bernoulli$(m_a(\cT,y))$ for all $y\in\cL$. Then
for all $S\in S_f$ such that $\Pr(A_S)>0$ and all $x \in L_S$, 
we have $m(S,x)=m_a(S,x)$ 
and $s(S,x)=\cs(1-m_a(S,x))\indiq_{\{t(x)=0\}}+\cs(m_a(S,x))\indiq_{\{t(x)=1\}}$, 
where $\cs$ is the unique function from $[0,1]$ into $[0,1]$ such that for all $\alpha\in[0,1]$,
\begin{equation}\label{tosolve}
\cs(\alpha)=\mu(0)\alpha(1-\alpha)+ \sum_{\ell\geq 1}\mu(\ell) 
(1-\alpha)^{2(\ell-1)/\ell}\cs((1-\alpha)^{1/\ell}).
\end{equation}
\end{ex}

\begin{proof}
We already know from Example \ref{tm3} that Assumption \ref{as} is satisfied and that 
$m(S,x)=m_a(S,x)$.
Next, \eqref{tosolve} has a unique solution because the map $F:E\mapsto E$, where
$E$ is the set of all functions from $[0,1]$ into $[0,1]$, defined by
$$
F(\cs)(\alpha)=\mu(0)\alpha(1-\alpha)+ \sum_{\ell\geq 1}\mu(\ell) 
(1-\alpha)^{2(\ell-1)/\ell}\cs((1-\alpha)^{1/\ell}),
$$
is a contraction. Indeed,
$||F(\cs_1)-F(\cs_2)||_\infty \leq \kappa ||\cs_1-\cs_2||_\infty$ with
$\kappa=\sum_{\ell\geq 1}\mu(\ell)<1$.

Let us denote by $\cs(a)=s(\{r\},r)$, which clearly depends only on $a$ (and $\mu$).

For any $S \in S_f$ such that $\Pr(A_S)>0$ and $x \in L_S$,
we have $s(S,x)=\cs(m_a(S,x))$ if $t(x)=1$ and $s(S,x)=\cs(1-m_a(S,x))$ if $t(x)=0$.
Indeed, the law of $(\cT_x,(R(u))_{u\in \cL\cap\cT_x})$ conditionally on $A_S$ is the same
as that of $(\cT,(R(u))_{u\in \cL})$ (re-rooted at $x$), replacing $a$ by $m_a(S,x)$:
$\cT_x$ is a Galton-Watson tree with reproduction law $\mu$ and, knowing $A_S$ and $\cT_x$,
one easily checks that $m_a(\cT,y)=m_{m_a(S,x)}(\cT_x,y)$ for all $y\in \cL\cap \cT_x$.
Hence we have $s(S,x)=\cs(m_a(S,x))$ if $t(x)=1$. If now $t(x)=0$,
we see that $s(S,x)=\cs(1-m_a(S,x))$ by exchanging the roles of the two players.

It remains to verify that $\cs$ satisfies \eqref{tosolve}. To this end, it suffices to 
apply the formula
of Lemma \ref{hor} concerning $s$ with $S=\{r\}$ (whence $A_S=\Omega$) and $x=r$ (with $t(r)=1$)
and to observe that

\noindent $\bullet$ $s(\{r\},r)=\cs(a)$ and $m(\{r\},r)=a$,

\noindent $\bullet$ $\Pr(r\in\cL,R(r)=1)=\mu(0)a$ and $\Pr(r\in\cL,R(r)=0)=\mu(0)(1-a)$,

\noindent $\bullet$ $\sum_{\by\subset \CC_r,|\by|=\ell}\Pr(\cC_r=\by)=\mu(\ell)$,

\noindent $\bullet$ for any $\by\subset \CC_r$ with $|\by|=\ell\geq 1$, conditionally on $\cC_r=\by$,
for all $y\in\by$,
we have $m_a(\{r\}\cup \by,y)= 1-(1-a)^{1/\ell}$ and thus, since $t(y)=1$, 
$s(\{r\}\cup \by,y)=\cs((1-a)^{1/\ell})$,
whence $\Gamma(\{r\},r,\by,y)=\cs((1-a)^{1/\ell})(1-a)^{2(\ell-1)/\ell} + [(1-a)^{\ell/\ell}- (1-a)]^2
=(1-a)^{2(\ell-1)/\ell}\cs((1-a)^{1/\ell})$.
\end{proof}

It does not seem easy to solve \eqref{tosolve}. However, here is one possibility.

\begin{rk}\label{rtm3}
Assume that $\cT$ is a homogeneous Galton-Watson tree with reproduction law $\mu=(1-p)\delta_0+p\delta_d$, 
with $d\geq 2$ and $p\in (0,1/d]$. Consider the unique solution $a_0\in (0,1)$ to $a_0=(1-a_0)^{1/d}$.
Assume that conditionally on $\cT$, the family $(R(y))_{y\in \cL}$ 
is independent and that
$R(y)\sim$ Bernoulli$(m_{a_0}(\cT,y))$ for all $y\in\cL$. Then
for all $S\in S_f$ such that $\Pr(A_S)>0$ and all $x \in L_S$, 
we have $m(S,x)=m_{a_0}(S,x)$ and $s(S,x)=\cs(a_0)=(1-p)a_0^{d+1}/[1-p a_0^{2d-2}]$ is constant.
\end{rk}

Indeed, in such a case \eqref{tosolve} rewrites as 
$$
\cs(\alpha)=(1-p)\alpha(1-\alpha)+p(1-\alpha)^{2(d-1)/d}\cs((1-\alpha)^{1/d}),
$$ 
whence $\cs(a_0)=(1-p)a_0^{d+1}/[1-p a_0^{2d-2}]$.
Also,
one easily checks that for any $S\in\cS_f$ such that $\Pr(A_S)>0$ (so that $S$ is $d$-regular) 
and any $x \in L_S$,
we have $m_{a_0}(S,x)=a_0$ if $t(x)=1$ and  $m_{a_0}(S,x)=1-a_0$ if $t(x)=0$. This of course uses \eqref{q} 
and that $1-(1-a_0)^{1/d}=1-a_0$ and $(1-a_0)^{1/d}=a_0$.
Consequently, $s(S,x)$ always equals $\cs(a_0)$: $s(S,x)=\cs(m_{a_0}(S,x))=\cs(a_0)$ if $t(x)=1$ and 
$s(S,x)=\cs(1-m_{a_0}(S,x))=\cs(a_0)$ if $t(x)=0$.

\section{Global optimality fails}\label{gof}

\begin{proof}[Proof of Remark \ref{not}]
We assume here that $\cT$ is the binary tree with depth $3$.
We thus have the eight leaves $111$, $112$, $121$, $122$, $211$, $212$, $221$, $222$
(recall Subsection \ref{ssn}). We also assume that 
the family $(R(x))_{x\in \cL}$ is i.i.d. with common law Bernoulli$(1/2)$.

Observe that $\cT$ can be seen as an inhomogeneous Galton-Watson tree with reproduction laws
$\mu_0=\mu_1=\mu_2=\delta_2$ and $\mu_3=\delta_0$.
Applying Example \ref{tm1} (with $q_3=1/2$, the values of $q_0,q_1,q_2$ being irrelevant), 
we can compute the functions $m$ and $s$:
for any $S\in\cS_f$ such that $\Pr(A_S)>0$ and $x\in L_S$, we have 
$m(S,x)=\cm(|x|)$ and $s(S,x)=\cs(|x|)$, where
$$
\cm(1)=\frac{9}{16},\;\;\cm(2)=\frac{3}{4},\;\;\cm(3)=\frac 12,\quad
\cs(1)=\frac{9}{256},\;\;\cs(2)=\frac{1}{16},\;\;\cs(3)=\frac 14.
$$
The value of $\cm(0)$ and $\cs(0)$ are not useful to the algorithm. Let us however notice that
$\E[R(r)]=\cm(0)=207/256$. 

By symmetry, we can replace the uniformly random matches (starting from some $z$)
used in any admissible algorithm, see Definition \ref{proc},
by the visit of any deterministic leave (under $z$), without changing (at all) the 
performance of the algorithm.
With this slight modification,
some tedious computations show that Algorithm \ref{ouralgo},
using the above function $m$ and $s$, leads to the following strategy (and results) 
for the three first steps.

{\tt

\noindent Visit the leave $x_1=111$.

\noindent If $R(x_1)=1$ (whence $R_1(r)=228/256$), then

$\{$ visit the leave $x_2=121$.

\hskip0.35cm  If $R(x_2)=1$ (whence $R_2(r)=1$), then 

\hskip0.7cm $\{$ stop here (or visit any other leave). We have  $R_3(r)=1.$ $\}$

\hskip0.35cm  If $R(x_2)=0$ (whence $R_2(r)=200/256$), then 

\hskip0.7cm $\{$ visit the leave $x_3=122$. \!\!We have $R_3(r)=1$ if $R(x_3)=1$ 
and $R_3(r)=144/256$ 

\hskip1.1cm if $R(x_3)=0$. $\}$ $\}$

\noindent If $R(x_1)=0$ (whence $R_1(r)=186/256$), then

$\{$ visit the leave $x_2=112$.

\hskip0.35cm  If $R(x_2)=1$ (whence $R_2(r)=228/256$), then 

\hskip0.7cm $\{$ visit the leave $121$. We have  $R_3(r)=1$ if $R(x_3)=1$ and $R_3(r)=200/256$

\hskip1.1cm if $R(x_3)=0$. $\}$

\hskip0.35cm  If $R(x_2)=0$ (whence $R_2(r)=144/256$), then 

\hskip0.7cm $\{$ visit the leave $x_3=211$. We have $R_3(r)=192/256$ if $R(x_3)=1$

\hskip1.1cm and $R_3(r)=96/256$  if $R(x_3)=0$. $\}$ $\}$
}

Noting that $\E[(R_n(r)-R(r))^2]=\E[(R(r))^2]-\E[(R_n(r))^2]=207/256
-\E[(R_n(r))^2]$ because $\E[R(r)R_n(r)]=\E[(R_n(r))^2]$ (since $R_n(r)=\E[R(r)|\cF_n]$)
and since $\E[(R(r))^2]=\E[R(r)]=\cm(0)=207/256$, we conclude that 
\begin{align*}
\E[(R_1(r)-R(r))^2]=&\frac{207}{256}- \frac 12 \Big[\Big(\frac{228}{256}\Big)^2
+\Big(\frac{186}{256}\Big)^2 \Big]
=\frac{4851}{32768} ,\\
\E[(R_2(r)-R(r))^2]=&\frac{207}{256}- \frac 14 \Big[1+\Big(\frac{200}{256}\Big)^2
+\Big(\frac{228}{256}\Big)^2 
+\Big(\frac{144}{256}\Big)^2 \Big]
=\frac{2107}{16384} ,\\
\E[(R_3(r)-R(r))^2]=&\frac{207}{256}- \frac 18 \Big[1+1+1+\Big(\frac{144}{256}\Big)^2
+1+\Big(\frac{200}{256}\Big)^2+\Big(\frac{192}{256}\Big)^2+\Big(\frac{96}{256}\Big)^2    \Big]
= \frac{859}{8192}.
\end{align*}

The following strategy, that we found with the help of a computer, is less efficient in two steps 
but more efficient
in three steps. We use the notation $\tR_n(r)$ as in the statement.

{\tt
\noindent Visit the leave $\tx_1=111$.

\noindent If $R(\tx_1)=1$ (whence $\tR_1(r)=228/256$), then

$\{$ visit the leave $\tx_2=121$.

\hskip0.35cm  If $R(\tx_2)=1$ (whence $\tR_2(r)=1$), then 

\hskip0.7cm $\{$ stop here (or visit any other leave). We have  $\tR_3(r)=1.$ $\}$

\hskip0.35cm  If $R(\tx_2)=0$ (whence $\tR_2(r)=200/256$), then 

\hskip0.7cm $\{$ visit the leave $\tx_3=122$. \!\!We have $\tR_3(r)=1$ if $R(\tx_3)=1$ 
and $\tR_3(r)=144/256$ 

\hskip1.1cm if $R(\tx_3)=0$. $\}$ $\}$

\noindent If $R(\tx_1)=0$ (whence $\tR_1(r)=186/256$), then

$\{$ visit the leave $\tx_2=211$.

\hskip0.35cm  If $R(\tx_2)=1$ (whence $\tR_2(r)=216/256$), then

\hskip0.7cm $\{$ visit the leave $221$. We have  $\tR_3(r)=1$ if $R(\tx_3)=1$ and $\tR_3(r)=176/256$

\hskip1.1cm if $R(\tx_3)=0$. $\}$

\hskip0.35cm  If $\tR(\tx_2)=0$ (whence $\tR_2(r)=156/256$), then 

\hskip0.7cm $\{$ visit the leave $\tx_3=112$. We have $\tR_3(r)=216/256$ if $R(\tx_3)=1$

\hskip1.1cm and $\tR_3(r)=96/256$  if $R(\tx_3)=0$. $\}$ $\}$
}

We conclude, using the same argument as previously, that 
\begin{align*}
\E[(\tR_1(r)-R(r))^2]=&\frac{207}{256}- \frac 12 \Big[\Big(\frac{228}{256}\Big)^2
+\Big(\frac{186}{256}\Big)^2 \Big]
=\frac{4851}{32768} ,\\
\E[(\tR_2(r)-R(r))^2]=&\frac{207}{256}- \frac 14 \Big[1+\Big(\frac{200}{256}\Big)^2
+\Big(\frac{216}{256}\Big)^2 
+\Big(\frac{156}{256}\Big)^2 \Big]
=\frac{2215}{16384} ,\\
\E[(\tR_3(r)-R(r))^2]=&\frac{207}{256}- \frac 18 \Big[1+1+1+\Big(\frac{144}{256}\Big)^2
+1+\Big(\frac{176}{256}\Big)^2+\Big(\frac{216}{256}\Big)^2+\Big(\frac{96}{256}\Big)^2    \Big]
= \frac{847}{8192}.
\end{align*}
The proof is complete.
\end{proof}

\section{Numerical results}\label{numer}

\subsection{Numerical problems}
Algorithm \ref{ouralgo} is subjected to numerical problems due to the fact that it 
proceeds to a high number of multiplications of reals in $[0,1]$. 
For example, the algorithm continuously computes products of the form 
$\prod_{k=1}^d r_k$ and $1-\prod_{k=1}^d (1-r_k)$, with
$r_1,\dots,r_d \in [0,1]$. If coded naively, it immediately finds $0$ or $1$
and does not work at all. We overcame 
such problems with the change of variables $\phi(r)=\log[r/(1-r)]$.
Everywhere, we used $\phi(R)$, $\phi(U)$ and $\phi(Z)$ instead of $R$, $U$ and $Z$ 
(we mean, concerning the values $R_n(x)$, $U_n(x)$ and the $Z_n(x)$).
Actually, for large games, some numerical problems persist: at some steps,
we have numerically $(R_{n+1}(r),U_{n+1}(r),Z_{n+1}(r))=(R_n(r),U_{n}(r),Z_{n}(r))$
(even after the change of variables), which should never be the case.
However, the above trick eliminates most of them. Instead of using $\phi$, one could manipulate
simultaneously $\log r$ and $\log(1-r)$. This would be more or less equivalent, the use of $\phi$ is just
slightly more concise.

We used the following expressions. We carefully separated different cases,
because e.g. for $u$ very large (say, $u\geq 750$), the computer answers $\log(1+\exp(u))=+\infty$
but $u+\log(1+\exp(-u))=u$, these two quantities being theoretically equal.
Consider $r,s \in [0,1]$ and $u=\phi(r)$, $v=\phi(s)$. 
We e.g. assume that $0\leq s \leq r \leq 1$ (whence $-\infty\leq v\leq u\leq +\infty$) 
and we naturally allow 
$\phi(0)=-\infty$, $\phi(1)=+\infty$, $\exp(-\infty)=0$, etc.
Observe that $r+s\leq 1$ if and only if $u+v\leq 0$.

(a) $\phi(1-r)=-\phi(r)$,

(b) $\phi(rs)=\begin{cases}u+v-\log(1+e^{u}+e^{v}) & \hbox{if $u<0$,}\\
+\infty&\hbox{if $v=+\infty$,}\\
v-\log(1+e^{-u}+e^{v-u})& \hbox{if $v<+\infty$ and $u\geq 0$}.
\end{cases}$

(c) $\phi(r+s)=\begin{cases} -\infty & \hbox{if $u=-\infty$},\\
+\infty & \hbox{if $v=-\infty$ and $u=+\infty$},\\
u+\log(1+e^{v-u}+2e^v)-\log(1-e^{u+v}) &\hbox{if $u>-\infty$ and $u+v\leq 0$}. \end{cases}$

(d) $\phi(s/r)=\begin{cases} +\infty & \hbox{if $u=v$},\\
v & \hbox{if $v<u=+\infty$},\\
\log(1+e^{u})-u+v-\log(1-e^{v-u}) &\hbox{if $v<u<0$}, \\
\log(1+e^{-u})+v-\log(1-e^{v-u}) &\hbox{if $0\leq u < +\infty$ and $v<u$}.\end{cases}$

(e) $\log r = \begin{cases}
u-\log(1+e^u)&\hbox{if $u<0$,}\\
-\log(1+e^{-u})&\hbox{if $u\geq 0$.}
\end{cases}$

We can compute $\phi(r-s)=\phi(r(1-s/r))$ from $u,v$ using (a), (b) and (d).
For $r_1,\dots,r_d \in [0,1]$, we can compute
$\phi(\prod_1^d r_k)$ from the values of $u_k=\phi(r_k)$ recursively using (b), and we can get
$\phi(1-\prod_1^d (1-r_k))=-\phi(\prod_1^d (1-r_k))$ using furthermore (a). Etc.

\subsection{The algorithms}\label{thealgo}

We will make play some versions of our algorithm against the two versions of MCTS recalled
in Section \ref{smcts} on a few real games (variations of Connect Four) and on Pearl's model.

We call MCTS($a,b$) Algorithm \ref{mcts} and MCTS'($a,b$) Algorithm \ref{mctsinf} with the parameters $a,b>0$.

We call GW (resp. GW2) Algorithm \ref{ouralgo} with the functions $m$ and $s$ of Example \ref{tm1} (resp.
Example \ref{tm2}),
assuming that the game is the realization of an inhomogeneous Galton-Watson tree
(resp. inhomogeneous Galton-Watson tree of order $2$). The functions
$\cm$ and $\cs$ are computed as follows, for example in the case of GW. 
We handle a large number ($10^8$) of
uniformly random matches starting from the initial configuration of the game, this allows us
to estimate $\mu_0,\dots,\mu_K$. The values $q_0,\dots,q_K$ are trivial for 
Connect Four, since we simply have $q_k=\indiq_{\{k \hbox{ \tiny is odd}\}}$.
We then use the formulas stated in Example \ref{tm1}.
We obtain rather stable results. Of course, this is done once for all for each game.

We call Sym Algorithm \ref{ouralgo} with the function $m=m_a$ defined by \eqref{q},
with the choice $a=1/2$,
and with $s\equiv 1$. Recall that $a\in (0,1)$ is the expected minimax rating of the root.
This is the most simple and universal algorithm, although not fully theoretically justified
(see however Remark \eqref{rtm3} in Subsection \ref{stm3}).

We finally call SymP Algorithm \ref{ouralgo} with the functions $m$ and $s$ defined
in Remark \ref{pearl2}, here again with $a=1/2$. 
This is the theoretical algorithm furnished by our study in the case of Pearl's game (if $a=1/2$)
and it is precisely the same as GW in this case, see Remarks \ref{pearl} and \ref{pearl2}.

Let us now give a few precisions. 

(i) In all the experiments below, each algorithm keeps the information provided by its own simulations
handled to decide its previous moves. 
In practice, this at most doubles the quantity of information (when compared to the case
where we would delete everything at each new move), because most of the previous simulations led
to other positions.

(ii) Concerning Sym and SymP, we actually use $a=1/2$ as expected minimax value of the {\it true} root
of the game, that is the true initial position. When in another configuration $x$, we use 
$m_a(\cT,x)$ (with $a=1/2$, here $\cT$ is the tree representing the whole game) 
as expected minimax value of the {\it current} root $x$
(i.e. the current position of the game). Such a value is automatically computed when playing the game.

\subsection{The numerical experiments}

First, let us mention that we handled many trials using GW and GW2. They almost always worked
less well than Sym concerning Connect Four, so we decided
not to present those results.
Also, concerning Sym and SymP, we tried other values for the expected minimax value
$a \in (0,1)$ of the root without observing significantly better results, so we always use $a=1/2$.
Similarly, we experimented other values for $b\in \rr$ (see Subsection \ref{pchoice}, 
Sym corresponds to the case
$b=0$ and SymP to the case where $b=2$), here again without clear success.

In each subsection below (except Subsection \ref{rat}), which concerns one given game, we proceed as follows.

For each given amount of time per move, we first fit the parameters of MCTS. To this aim,
we perform a championship involving MCTS($a,b$), for all $a=k/2$, $b=\ell/2$ with $1\leq k\leq \ell \leq 10$
(we thus have $55$ players).
Each player competes $40$ times against all the other ones ($20$ times as first player, $20$ times as second one), 
and we select the player with the highest number of victories.
Observe that each player participates to $2160$ matches.
The resulting best player is rather unstable, but we believe
this is due to the fact that the competition is very tight among a few players. 
Hence even if we do not select the true best player,
we clearly select a very good one. Note that we impose $a\leq b$ because, after many trials allowing $a>b$,
the best player was always of this shape.

Of course, we do exactly the same thing to fit the parameters of MCTS'.

Then, we make our algorithm (Sym or SymP) compete against the best MCTS and the best MCTS',
$10000$ times as first player and $10000$ times as second one.

Also, we indicate the (rounded) mean number of iterations made by each algorithm {\it at the first move}.
This sometimes looks ridiculous: when e.g. playing a version of Connect Four with large degree with
$1$ millisecond per move, this mean 
number of iterations is $8$ for Sym and $11$ for MCTS. 
However, after sufficiently many moves, this mean number of iterations
becomes much higher. In other words, the algorithms more or less play at random at the beginning,
but become more and more {\it clever} as the game progresses. So in some sense,
the algorithm that wins is the one becoming clever before the other.

Finally, let us explain how to read the tables below, which are all of the same shape. 
For example, the first table, when playing a large Pearl game, is as follows.

\btp
\hline
\multicolumn{2}{|c|}
{$t=1$ ms. MCTS(3.5,5): 498, MCTS'(2,3): 442, SymP: 163}\\
\hline
SymP vs MCTS(3.5,5) & SymP vs MCTS'(2,3)\\
\hline
5882/4118(0), 4340/5660(0) & {\bf 4992/5008(0), 5221/4779(0)} \\
\hline
\etp

Each player had $1$ millisecond to decide each of its moves. The championships were won by MCTS(3.5,5)
and MCTS'(2,3). When playing its first move, MCTS(3.5,5) (resp. MCTS'(2,3), resp. SymP) 
proceeded in mean to 498 (resp. 442, resp. 163) iterations.

When playing first, SymP won 5882 times and lose 4118 times against MCTS(3.5,5), 
and there has been 0 draw.
When playing first, MCTS(3.5,5) won 4340 times and lose 5660 times against SymP,
and there has been 0 draw.

When playing first, SymP won 4992 times and lose 5008 times against MCTS'(2,3),
and there has been 0 draw. When playing first, MCTS'(2,3) won 5221 times and lose 4779 times against SymP, 
and there has been 0 draw.

Finally, the results {\bf in bold} mean that our algorithm (Sym or SymP) is beaten.

\subsection{Large Pearl game}

We first consider Pearl's game, that is the game of Example \ref{pearl} with the regular tree
with degree $d=2$ and depth $K=32$, with i.i.d. Bernoulli$(p)$ random variables on the leaves and with
$p$ such that $\E[R(r)]=1/2$ ($p$ is easily computed numerically using the formula
of Remark \ref{pearl} and that $p=\cm(K)$ and $1/2=\cm(0)$). 

Here and in the next subsection, to be as fair as possible, in each cell of the tables below,
each match is played in both senses on a given realization of the model,
so that if the two opponents were playing perfectly, one would find twice the same results
in each cell.

\btp
\hline
\multicolumn{2}{|c|}
{$t=1$ ms. MCTS(3.5,5): 498, MCTS'(2,3): 442, SymP: 163}\\
\hline
SymP vs MCTS(3.5,5) & SymP vs MCTS'(2,3)\\
\hline
5882/4118(0), 4340/5660(0) & {\bf 4992/5008(0), 5221/4779(0)} \\
\hline
\etp

\btp
\hline
\multicolumn{2}{|c|}
{$t=2$ ms. MCTS(3.5,5): 717, MCTS'(3,4.5): 579, SymP: 238}\\
\hline
SymP vs MCTS(3.5,5) & SymP vs MCTS'(3,4.5)\\
\hline
5751/4249(0), 4267/5733(0) & {\bf 4727/5273(0), 5372/4628(0)} \\
\hline
\etp

\btp
\hline
\multicolumn{2}{|c|}
{$t=4$ ms. MCTS(4,4.5): 1137, MCTS'(3.5,5): 848, SymP: 395}\\
\hline
SymP vs MCTS(4,4.5) & SymP vs MCTS'(3.5,5)\\
\hline
5921/4079(0), 4044/5956(0) & {\bf 4676/5324(0), 5306/4694(0)} \\
\hline
\etp

\btp
\hline
\multicolumn{2}{|c|}
{$t=8$ ms. MCTS(3,3.5): 2210, MCTS'(3.5,5): 1561, SymP: 693}\\
\hline
SymP vs MCTS(3,3.5) & SymP vs MCTS'(3.5,5)\\
\hline
5926/4074(0), 4080/5920(0) & {\bf 4641/6359(0), 5438/4562(0)} \\
\hline
\etp

\btp
\hline
\multicolumn{2}{|c|}
{$t=16$ ms. MCTS(4,5): 4229, MCTS'(4,4.5): 2893, SymP: 1307}\\
\hline
SymP vs MCTS(4,5) & SymP vs MCTS'(4,4.5)\\
\hline
5998/4002(0), 4012/5988(0) & {\bf 5000/5000(0), 5112/4888(0)} \\
\hline
\etp

\btp
\hline
\multicolumn{2}{|c|}
{$t=32$ ms. MCTS(4,4.5): 8277, MCTS'(4.5,5): 5922, SymP: 2473}\\
\hline
SymP vs MCTS(4,4.5) & SymP vs MCTS'(4.5,5) \\
\hline
6154/3846(0), 3781/6219(0)
& 5059/4941(0), 4923/5077(0)\\
\hline
\etp

\btp
\hline
\multicolumn{2}{|c|}
{$t=64$ ms. MCTS(4.5,4.5): 15983, MCTS'(3.5,4.5): 14399, SymP: 4766}\\
\hline
SymP vs MCTS(4.5,4.5) & SymP vs MCTS'(3.5,4.5) \\
\hline
6380/3620(0), 3628/6372(0) & 5244/4756(0), 4800/5200(0)\\
\hline
\etp

We observe that SymP is better than MCTS but beats MCTS' only when the amount of time per play is high enough.

For this game, SymP performs around 4 times fewer iterations per unit of time than MCTS.

\subsection{Small Pearl game}

We next consider a much smaller Pearl game: as previously, $d=2$ and $p$ is chosen is such that
$\E[R(r)]=1/2$, but the depth of the tree is $K=16$ (instead of $K=32$).

Here, and only here, due to the smallness of the game, we needed to modify the way we fit the parameters of 
MCTS and MCTS'. Namely, we perform a championship involving MCTS($a,b$), for all $a=k/2$, $b=\ell/2$ with 
$1\leq k\leq \ell \leq 20$ (we thus have $210$ players).
Each player competes $20$ times against all the other ones ($10$ times as first player, $10$ times as second one), 
and we select the player with the highest number of victories.
Each player participates to $4180$ matches.

Proceeding as in the previous subsection, we found the following results.

\btp
\hline
\multicolumn{2}{|c|}
{$t=1$ ms. MCTS(5,8.5): 1100, MCTS'(5,7.5): 954, SymP: 301}\\
\hline
SymP vs MCTS(5,8.5) & SymP vs MCTS'(5,7.5)\\
\hline
5242/4758(0), 3966/6034(0)
& {\bf 4334/5666(0), 4868/5132(0)} \\
\hline
\etp

\btp
\hline
\multicolumn{2}{|c|}
{$t=2$ ms. MCTS(6,10): 1637, MCTS'(5.5,9.5): 1584, SymP: 444}\\
\hline
SymP vs MCTS(6,10) & SymP vs MCTS'(5.5,9.5)\\
\hline
5104/4896(0), 4027/5973(0) & {\bf 4346/5654(0), 4899/5101(0)}\\
\hline
\etp

\btp
\hline
\multicolumn{2}{|c|}
{$t=4$ ms. MCTS(5.5,9): 2829, MCTS'(5.5,9): 2776, SymP: 725}\\
\hline
SymP vs MCTS(5.5,9) & SymP vs MCTS'(5.5,9)\\
\hline
5307/4693(0), 3953/6047(0)
& {\bf 4296/5704(0), 4797/5203(0)}\\
\hline
\etp

\btp
\hline
\multicolumn{2}{|c|}
{$t=8$ ms. MCTS(7,9.5): 4876, MCTS'(5.5,9): 5745, SymP: 1128}\\
\hline
SymP vs MCTS(7,9.5) & SymP vs MCTS'(5.5,9)\\
\hline
4968/5032(0), 3817/6183(0) & {\bf 4434/5566(0), 4826/5174(0)}\\
\hline
\etp

\btp
\hline
\multicolumn{2}{|c|}
{$t=16$ ms. MCTS(6.5,7.5): 10303, MCTS'(4.5,6): 12729, SymP: 1692}\\
\hline
SymP vs MCTS(6.5,7.5) & SymP vs MCTS'(4.5,6)\\
\hline
5132/4868(0), 4519/5481(0)
& 4950/5050(0), 4943/5057(0)\\
\hline
\etp

\btp
\hline
\multicolumn{2}{|c|}
{$t=32$ ms. MCTS(6.5,8.5): 25929, MCTS'(7.5,9): 27585, SymP: 1840}\\
\hline
SymP vs MCTS(6.8,8.5) & SymP vs MCTS'(7.5,9)\\
\hline
5034/4966(0), 5010/4990(0)
& 5083/4917(0), 5083/4917(0)\\
\hline
\etp

Here it seems that SymP almost always finds the winning strategy at the first move with $16$ ms, 
this explains why the number of iterations is so small (at $16$ and $32$ ms): 
SymP stops before $16$ ms are elapsed. We thus believe it always finds it with $32$ ms.

At $32$ ms, it seems that MCTS' also always found the winning strategy among $5083$ times it started the game
with a possible winning strategy.
MCTS missed $24$ times the winning strategy among $5034$ times it started the game
with a possible winning strategy.

Observe that SymP is better than MCTS but, here again, 
beats MCTS' only when the amount of time per play is high enough.

Also, even if the game is theoretically fair (because $\E[R(r)]=1/2$), it seems easier, for
all the algorithms, to find the winning strategy when being the second player
(this can be observed until $8$ ms). This might be explained
by the fact that the second player is the one playing the last move.

Here also, SymP performs around $4$ times fewer iterations per unit of time than MCTS. 

\subsection{Standard Connect Four}
We now play Connect Four in its usual version: we have $7$ columns, $6$ lines, 
and the goal is to connect (horizontally, vertically, or diagonally) $4$ discs.

\btp
\hline
\multicolumn{2}{|c|}
{$t=1$ ms. MCTS(1.5,1.5): 53, MCTS'(1,1): 50, Sym: 33}\\
\hline
Sym vs MCTS(1.5,1.5) & Sym vs MCTS'(1,1)\\
\hline
{\bf 4174/5809(17), 6806/3146(48)}  & {\bf 4066/5904(30), 6983/2966(51)} \\
\hline
\etp

\btp
\hline
\multicolumn{2}{|c|}
{$t=2$ ms. MCTS(1,1): 82, MCTS'(2,2.5): 76, Sym: 48}\\
\hline
Sym vs MCTS(1,1) & Sym vs MCTS'(2,2.5)\\
\hline
{\bf 3099/6873(28), 7558/2392(50)}  & {\bf 3184/6799(17), 7632/2317(51)} \\
\hline
\etp

\btp
\hline
\multicolumn{2}{|c|}
{$t=4$ ms. MCTS(2,2): 137, MCTS'(1.5,1.5): 124, Sym: 81}\\
\hline
Sym vs MCTS(2,2) & Sym vs MCTS'(1.5,1.5)\\
\hline
{\bf 2460/7527(13), 7989/1966(45)}&{\bf 2485/7500(15), 7908/2043(49)} \\
\hline
\etp

\btp
\hline
\multicolumn{2}{|c|}
{$t=8$ ms. MCTS(3,3.5): 254, MCTS'(3.5,4): 225, Sym: 146}\\
\hline
Sym vs MCTS(3,3.5) & Sym vs MCTS'(3.5,4)\\
\hline
{\bf 1923/8074(3), 8163/1791(46)}  & {\bf 2004/7991(5), 8211/1751(38)} \\
\hline
\etp

\btp
\hline
\multicolumn{2}{|c|}
{$t=16$ ms. MCTS(4,4.5): 491, MCTS'(3.5,4): 441, Sym: 280}\\
\hline
Sym vs MCTS(4,4.5) & Sym vs MCTS'(3.5,4)\\
\hline
{\bf 1353/8644(3), 8341/1599(60)}  & {\bf 1520/8477(3), 8312/1647(41)} \\
\hline
\etp

We are largely beaten, and this is worse and worse as the given amount of time increases.
Observe however the high number of draws, which indicates that even if MCTS almost always wins at the end,
the competition is tight.

Let us mention that with $1024$ ms per move, we found, for Sym vs MCTS(10,10): 
14/986(0), 925/74(1) and, for Sym vs MCTS'(10,10): 14/986(0), 924/74(2):
we are destroyed.

Sym performs here around twice fewer iterations per unit of time than MCTS.

\subsection{A version of Connect Four with small degree}

We next consider the variation of Connect Four with $4$ columns, $10$ lines, and where the goal is
to connect (horizontally, vertically, or diagonally) $3$ discs.

\btp
\hline  
\multicolumn{2}{|c|}
{$t=1$ ms. MCTS(4,5): 223, MCTS'(2,2.5): 257, Sym: 107}\\
\hline
Sym vs MCTS(4,5) & Sym vs MCTS'(2,2.5)\\
\hline
{\bf 7591/2409(0), 8685/1315(0)}&{\bf 7082/2918(0), 8944/1056(0)}\\
\hline
\etp

\btp
\hline  
\multicolumn{2}{|c|}
{$t=2$ ms. MCTS(1.5,2): 441, MCTS'(2.5,3): 432, Sym: 167}\\
\hline
Sym vs MCTS(1.5,2) & Sym vs MCTS'(2.5,3)\\
\hline
{\bf 7941/2059(0), 8805/1195(0)}
&{\bf 7585/2415(0), 9286/714(0)}\\
\hline
\etp

\btp
\hline  
\multicolumn{2}{|c|}
{$t=4$ ms. MCTS(2,2): 722, MCTS'(3.5,4): 921, Sym: 294}\\
\hline
Sym vs MCTS(2,2) & Sym vs MCTS'(3.5,4)\\
\hline
{\bf 8647/1353(0), 9181/819(0)}
&{\bf 8445/1555(0), 9553/447(0)}\\
\hline
\etp

\btp
\hline  
\multicolumn{2}{|c|}
{$t=8$ ms. MCTS(4.5,5): 2215, MCTS'(4.5,5): 2966, Sym: 530}\\
\hline
Sym vs MCTS(4.5,5) & Sym vs MCTS'(4.5,5)\\
\hline
9920/80(0), 9694/306(0)
& 9908/92(0), 9879/121(0)\\
\hline
\etp

\btp
\hline  
\multicolumn{2}{|c|}
{$t=16$ ms. MCTS(3.5,4): 7409, MCTS'(4,4): 6705, Sym: 982}\\
\hline
Sym vs MCTS(3.5,4) & Sym vs MCTS'(4,4)\\
\hline
10000/0(0), 9952/48(0) & 10000/0(0), 9983/17(0)\\
\hline
\etp

Here we observe that MCTS and MCTS' win when having a lot time (for such a small game,
$1$ millisecond is much), but when the amount of time is so high that we are close to finding
the winning strategy at the first move, Sym is better. 
This might be due to the fact that Sym automatically does some pruning.

Here Sym performs around $3$ or $4$ times fewer iterations per unit of time than MCTS.
This is not the case with $16$ ms because it finds the right move, 
and thus stops, before the 16 ms are elapsed.

\subsection{A first version of connect 4 with large degree}
Here we consider a very simple version of Connect Four, with
$15$ columns, $15$ lines, and where the goal is to connect (horizontally, vertically, or diagonally) $3$ discs.
A human player immediately finds the winning strategy when playing first.
However, for a computer, the situation is not so easy, because there are many possibilities (if not taking 
advantage of the symmetries of the game).

\btp
\hline
\multicolumn{2}{|c|}
{$t=1$ ms. MCTS(2.5,4): 21, MCTS'(1.5,2.5): 20, Sym: 16}\\
\hline
Sym vs MCTS(2.5,4) & Sym vs MCTS'(1.5,2.5)\\
\hline
8396/1604(0), 5236/4764(0) & 8453/1547(0), 5292/4708(0)\\
\hline
\etp

\btp
\hline
\multicolumn{2}{|c|}
{$t=2$ ms. MCTS(2,3.5): 32, MCTS'(2,3.5): 31, Sym: 24}\\
\hline
Sym vs MCTS(2,3.5) & Sym vs MCTS'(2,3.5)\\
\hline
8368/1632(0), 5652/4348(0) & 8355/1645(0), 5812/4188(0)   \\
\hline
\etp

\btp
\hline
\multicolumn{2}{|c|}
{$t=4$ ms. MCTS(1.5,2.5): 54, MCTS'(3,5): 52, Sym: 40}\\
\hline
Sym vs MCTS(1.5,2.5) & Sym vs MCTS'(3,5)\\
\hline
8792/1208(0), 5898/4102(0) & 8749/1251(0), 5893/4107(0) \\
\hline
\etp

\btp
\hline
\multicolumn{2}{|c|}
{$t=8$ ms. MCTS(2.5,4): 100, MCTS'(3,5): 96, Sym: 72}\\
\hline
Sym vs MCTS(2.5,4) & Sym vs MCTS'(3,5)\\
\hline
9207/793(0), 5894/4106(0)& 9219/781(0), 6032/3968(0)   \\
\hline
\etp

\btp
\hline
\multicolumn{2}{|c|}
{$t=16$ ms. MCTS(1.5,2): 194, MCTS'(3.5,5): 182, Sym: 134}\\
\hline
Sym vs MCTS(1.5,2) & Sym vs MCTS'(3.5,5)\\
\hline
9388/612(0), 7048/2952(0) & 9417/583(0), 7153/2847(0)\\
\hline
\etp

\btp
\hline  
\multicolumn{2}{|c|}
{$t=32$ ms. MCTS(3,3.5): 379, MCTS'(3,4): 368, Sym: 260}\\
\hline
Sym vs MCTS(3,3.5) & Sym vs MCTS'(3,4)\\
\hline
9656/344(0), 8430/1570(0) & 9626/374(0), 8512/1488(0)\\
\hline
\etp

\btp
\hline  
\multicolumn{2}{|c|}
{$t=64$ ms. MCTS(4,5): 781, MCTS'(2,2.5): 785, Sym: 523}\\
\hline
Sym vs MCTS(4,5) & Sym vs MCTS'(2,2.5)\\
\hline
9851/149(0), 8953/1047(0) & 9830/170(0), 8784/1216(0) \\
\hline
\etp

\btp
\hline  
\multicolumn{2}{|c|}
{$t=128$ ms. MCTS(2.5,3): 1735, MCTS'(2,2.5): 1704, Sym: 1032}\\
\hline
Sym vs MCTS(2.5,3) & Sym vs MCTS'(2,2.5)\\
\hline
9919/81(0), 9221/779(0) & 9928/72(0), 9117/883(0)\\
\hline
\etp

Here we are really better than MCTS. We believe this is due to the fact that the degree of the game is very 
large.

For this game, Sym performs around $2$ times fewer iterations per unit of time than MCTS.

\subsection{A second version of Connect Four with large degree}

We consider here the version of Connect Four with $15$ columns, $6$ lines, and where the goal is to 
connect (horizontally, vertically, or diagonally) $5$ discs. Here the situation is intractable
for a normal human player.

\btp
\hline
\multicolumn{2}{|c|}
{$t=1$ ms. MCTS(2.5,4): 11, MCTS'(3,5): 10, Sym: 8}\\
\hline
Sym vs MCTS(2.5,4) & Sym vs MCTS'(3,5) \\
\hline
6368/3631(1), 4892/5106(2) & 6521/3477(2), 4892/5104(4)\\
\hline
\etp

\btp
\hline
\multicolumn{2}{|c|}
{$t=2$ ms. MCTS(0.5,0.5): 16, MCTS'(1,1.5): 15, Sym: 12}\\
\hline
Sym vs MCTS(0.5,0.5) & Sym vs MCTS'(1,1.5) \\
\hline
6730/3270(0), 4412/5586(2) & 6793/3205(2), 4537/5458(5)\\
\hline
\etp

\btp
\hline
\multicolumn{2}{|c|}
{$t=4$ ms. MCTS(0.5,0.5): 27, MCTS'(0.5,0.5): 25, Sym: 19}\\
\hline
Sym vs MCTS(0.5,0.5) & Sym vs MCTS'(0.5,0.5) \\
\hline
6305/3692(3), 5056/4936(8) & 6479/3515(6), 5038/4959(3)\\
\hline
\etp

\btp
\hline
\multicolumn{2}{|c|}
{$t=8$ ms. MCTS(0.5,0.5): 50, MCTS'(0.5,0.5): 45, Sym: 33}\\
\hline
Sym vs MCTS(0.5,0.5) & Sym vs MCTS'(0.5,0.5) \\
\hline
6160/3821(19), 4925/5044(31) & 5966/4022(12), 5159/4808(33)\\
\hline
\etp

\btp
\hline
\multicolumn{2}{|c|}
{$t=16$ ms. MCTS(1.5,2): 95, MCTS'(1.5,2): 88, Sym: 61}\\
\hline
Sym vs MCTS(1.5,2) & Sym vs MCTS'(1.5,2) \\
\hline
{\bf 4377/5598(25), 6499/3427(74)} & {\bf 4286/5692(22), 6526/3388(86)}\\
\hline
\etp

The number of iterations seems very small here, but let us recall
that this concerns only the beginning of the game. As already mentioned, this number of iterations 
actually increases considerably when approaching the end of the game.

We observe that Sym is really better than MCTS and MCTS' when the amount of time per move is small. 
When this amount of time 
increases, we are beaten. We have
two main explanations for this, see Subsection \ref{numconcl}.

Here Sym performs around twice fewer iterations per unit of time than MCTS.

\subsection{Inverse and large Connect Four}

We next consider the inverse version of Connect Four with $15$ columns, $6$ lines, and where the first player
that connects (horizontally, vertically, or diagonally) $5$ discs looses.
Although this modification is coded very easily (only one line has to be modified so that 
{\it win} and {\it loss} are exchanged), 
this considerably modifies the game. The algorithms (MCTS, MCTS' and Sym) all take this into account 
immediately: at the beginning of a match, one usually sees the algorithms play in the middle of the board,
while they play near the extremities in the {\it inverse} case.

\btp
\hline
\multicolumn{2}{|c|}
{$t=1$ ms. MCTS(1.5,3): 10, MCTS'(2,4): 9, Sym: 8}\\
\hline
Sym vs MCTS(1.5,3) & Sym vs MCTS'(2,4) \\
\hline
5094/4889(17), 3793/6182(25) & 4954/5037(9), 3792/6185(23)\\
\hline
\etp

\btp
\hline
\multicolumn{2}{|c|}
{$t=2$ ms. MCTS(1.5,3): 15, MCTS'(1,2): 14, Sym: 12}\\
\hline
Sym vs MCTS(1.5,3) & Sym vs MCTS'(1,2) \\
\hline
4478/5507(15), 4338/5628(34) &{\bf 4305/5681(14), 4444/5526(30)}\\
\hline
\etp

\btp
\hline
\multicolumn{2}{|c|}
{$t=4$ ms. MCTS(0.5,0.5): 25, MCTS'(2,4): 23, Sym: 19}\\
\hline
Sym vs MCTS(0.5,0.5) & Sym vs MCTS'(2,4) \\
\hline
{\bf 4058/5931(11), 5084/4887(29)}
&{\bf 4236/5757(7), 4607/5356(37)}\\
\hline
\etp

Here again, we observe that Sym wins when the amount of time per move is very small.

\subsection{Back to Pearl's game: rate of convergence}\label{rat}

Here we consider Pearl's game with $d=2$, $K$ even and with the value $p=(\sqrt 5 - 1)/2\simeq 0.618$. 
With this particular value of $p$, Pearl \cite{p} showed that $\E[R(r)]=p$ and that
the expected required number of visited leaves for AlphaBeta to determine $R(r)$ equals $(2/(\sqrt 5 -1))^{K}$.
Here $r$ is the true root of the game. 
In the whole subsection, we use SymP with the correct value of $a$, i.e. $a=p$.

First, we call $\tau_K$ the number of leaves that SymP needs to determine $R(r)$. 
Denoting by $\bar \tau_K$ the average value over $10000$ trials, we found

\begin{center}\begin{tabular}{|C{2.2cm}|C{2cm}|C{2cm}|C{2cm}|C{2cm}|}
\hline
$K$ & 4 & 8 & 12 & 16 \\
\hline
$\bar \tau_K$  & 6.84 & 47.14 & 323.51 & 2207.89 \\
\hline
$(2/(\sqrt 5 -1))^{K}$ & 6.8541 & 46.9787 & 321.9969 & 2206.9995\\
\hline\end{tabular}
\end{center}

I thus seems highly plausible that our algorithm visits the leaves in the same order as AlphaBeta
(for a Pearl game), up to some
random permutation. This is rather satisfying, since Tarsi \cite{tarsi} showed that for a Pearl game,
AlphaBeta is optimal in the sense of the expected number of leaves necessary to determine $R(r)$.

Finally, we plot a Monte-Carlo approximation (with $10000$ trials)
of $\E[(R_n(r)-R(r))^2]$ as a function of the number $n$ of iterations, when $K=8$ and $K=16$,
as well as one trajectory of $n\mapsto R_n(r)$ when $K=16$.

\begin{center}
\noindent\fbox{\begin{minipage}{0.9\textwidth}
\begin{center}
\includegraphics[width=6.5cm]{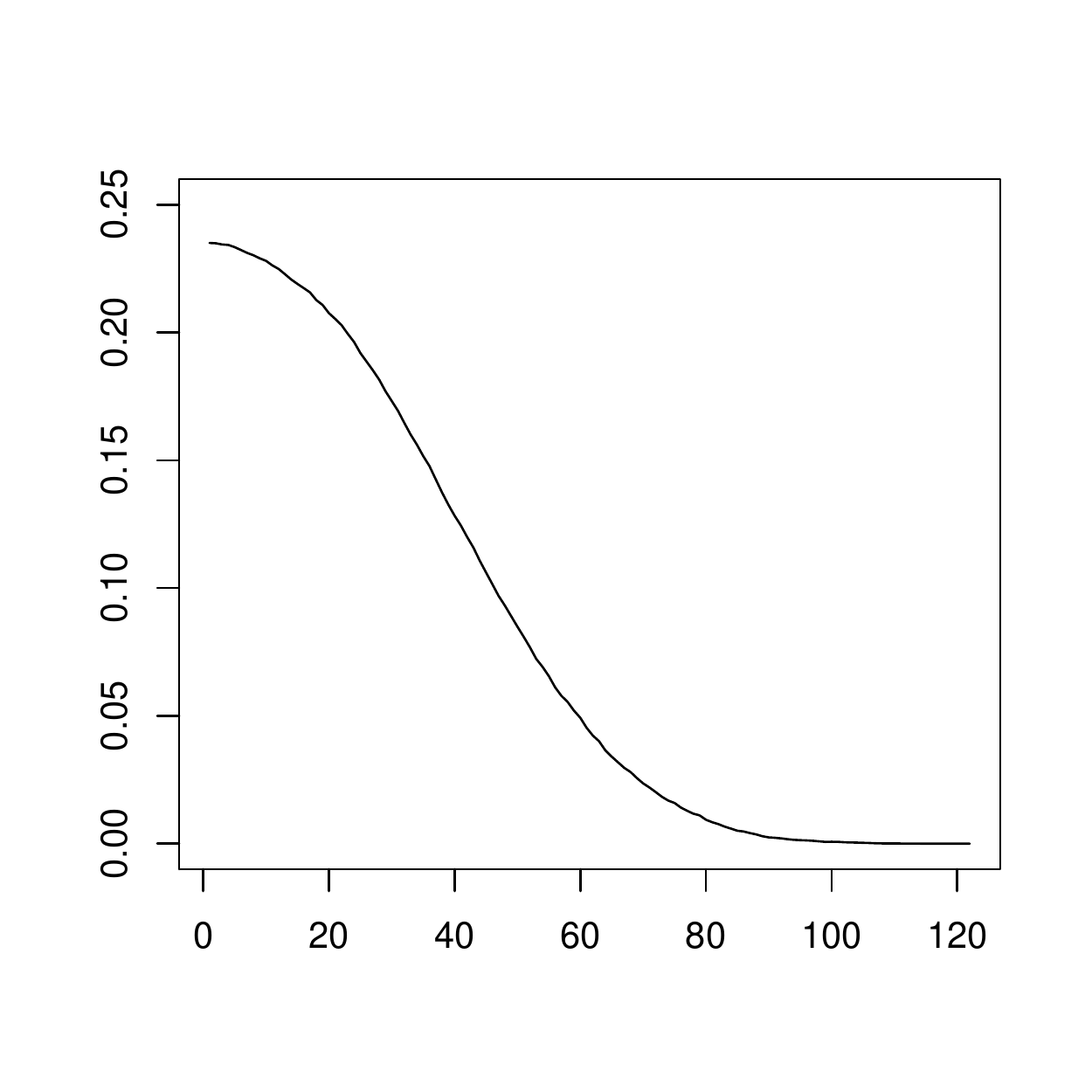}
\includegraphics[width=6.5cm]{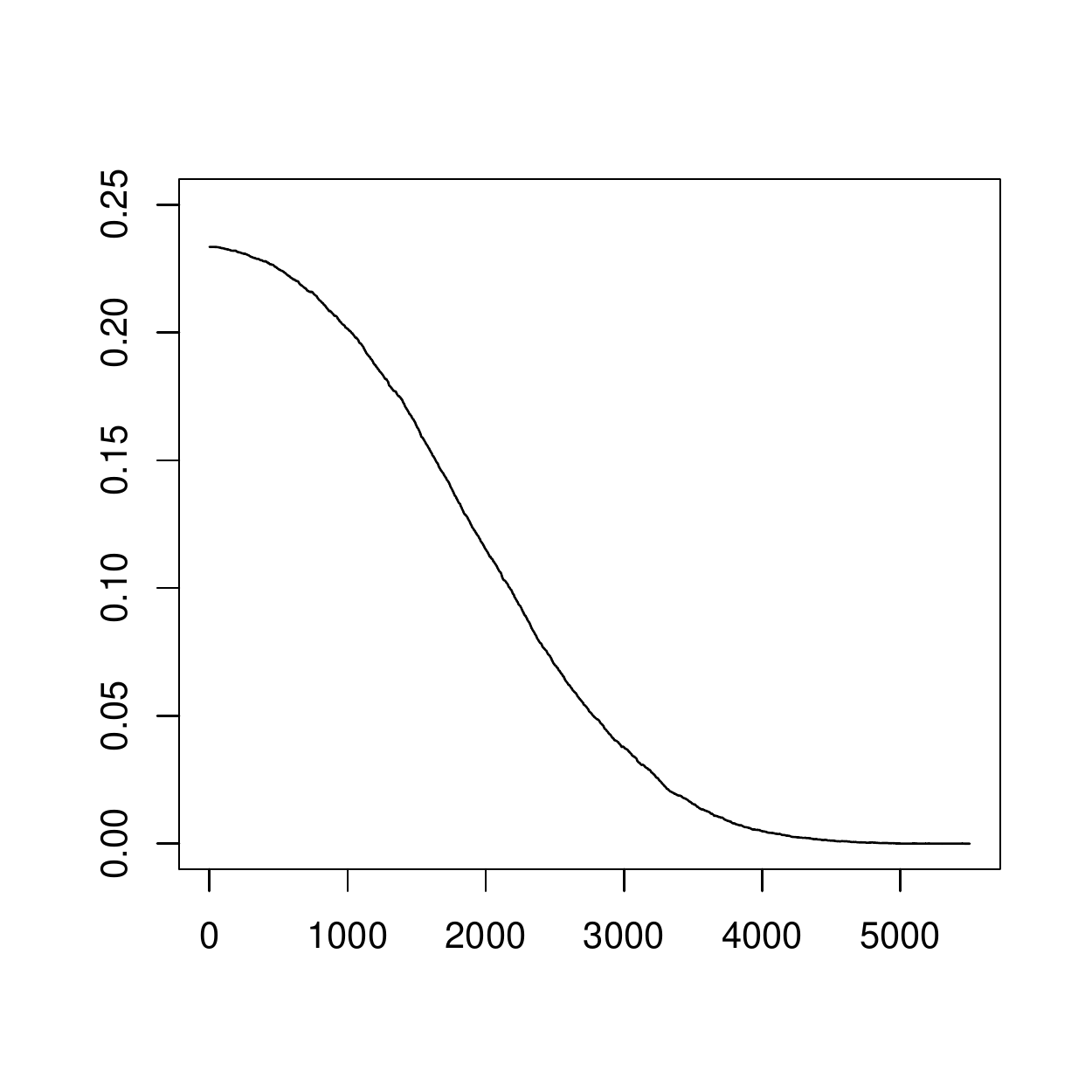}
\vskip-5.3cm $\left.\right.$ \hskip3.5cm $K=8$ \hskip5.5cm $K=16$ 
\vskip-1.5cm $\E[(R_n(r)-R(r))^2]$ as a function of the number $n$ of iterations.
\vskip4.9cm $\left.\right.$
\end{center}
\end{minipage}}
\end{center}

\begin{center}
\noindent\fbox{\begin{minipage}{0.9\textwidth}
\begin{center}
\includegraphics[width=6.5cm]{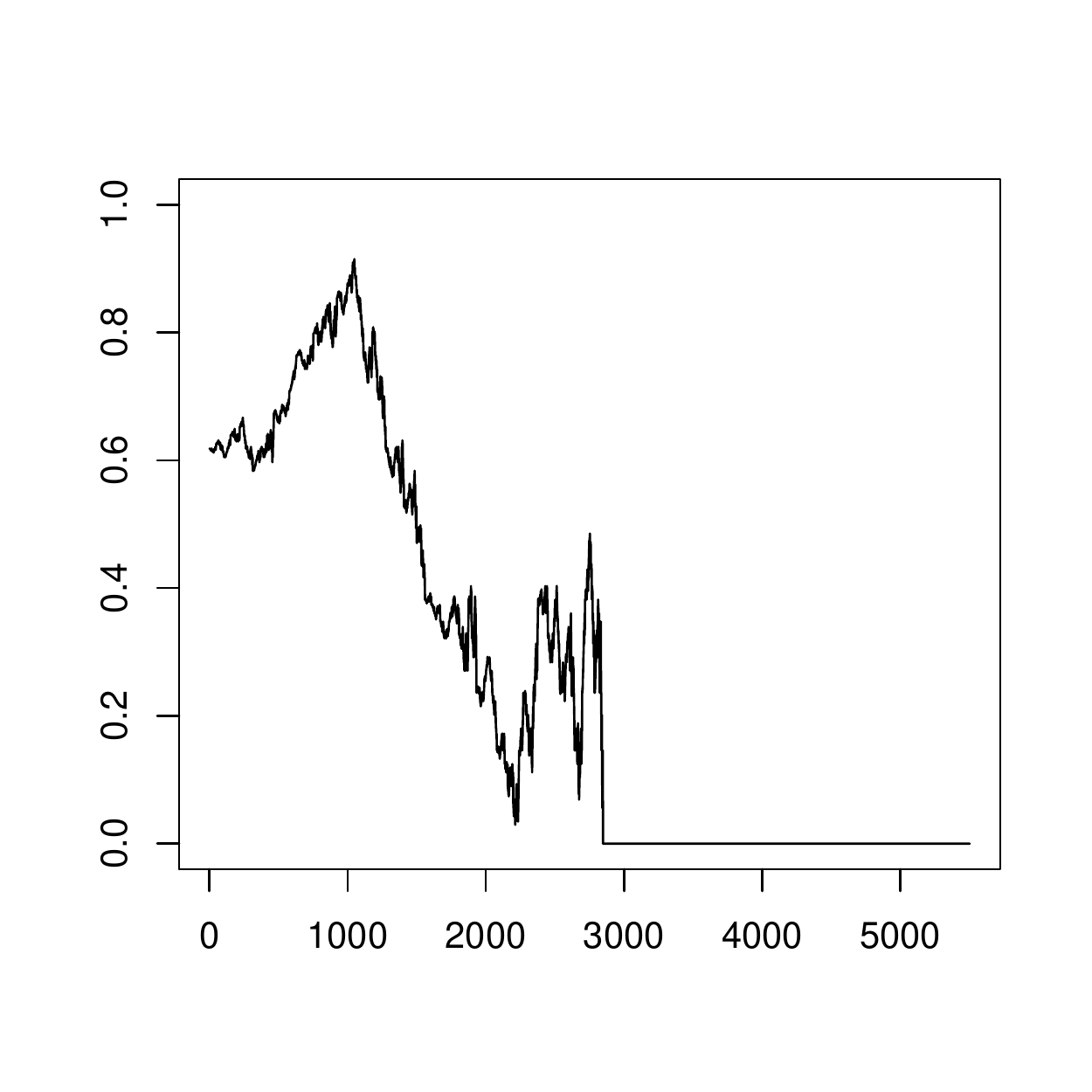}
\vskip-5.3cm $\left.\right.$ \hskip3.6cm $K=16$ 
\vskip-1.5cm One trajectory of $n \mapsto R_n(r)$.
\vskip4.9cm $\left.\right.$
\end{center}
\end{minipage}}
\end{center}

\subsection{Conclusion}\label{numconcl}

When playing Pearl's game, SymP beats MCTS and seems competitive against MCTS'.
This is reassuring, since our algorithms are typically designed for such games.

On true games, MCTS and MCTS' seem globally much better
than Sym. However, we found two situations where Sym may win.

The first and  most interesting situation is the one where 
the game is so large (or the amount of time so small) that very few iterations can be performed
by the challengers. This is quite natural, and there are two possible reasons for that.

\noindent $\bullet$ Our algorithm is only optimal {\it step by step}. So, 
it is absolutely not clear that it works well when we have enough time to handle many iterations.

\noindent $\bullet$ Assumption \ref{as} imposes some independence properties. While this is reasonable,
on any game, in some sense to be precised, 
after a small number of iterations (when playing a small number of uniformly
random matches, it is rather clear that the issues will be almost independent), this is clearly not the case
when performing a large number of well-chosen matches.

The second situation is that where the game is so small that we can hope to find the winning strategy
at the first move, and where Sym may find it before MCTS and MCTS'. As already mentioned,
we believe this is due to the fact that Sym does some pruning.

From a theoretical point of view, it would be very interesting to study more relevant models such as
the one proposed by Devroye-Kamoun \cite{dk}. Clearly, this falls completely out of our scope.
On the contrary, it does not seem completely desesperate to find empirically variants of
our algorithm that work much better in practise. For example, it may be relevant to use other choices 
of functions $m$ and $s$, to try a clever default policy, etc.

\section{Appendix: Monte Carlo Tree Search algorithms}\label{smcts}

In this subsection, we write down precisely the versions of the MCTS algorithm we used
to test our algorithm. We start with a modified version, more close to our study, where we do not throw down
any information.

\begin{algo}[Modified MCTS]\label{mctsinf}
Consider $\phi:\nn\times\nn\mapsto \rr$, e.g.
$\phi(w,c)=\frac{w+a}{c+b}$ for some $a,b>0$.

{\bf Step 1.} 
Simulate a uniformly random match from $r$, call $x_1$ the resulting leave and put $\bx_1=\{x_1\}$.

During this random match, keep track of $R(x_1)$, of $B_{\bx_1}=B_{rx_1}$ and of 
$\cD_{\bx_1}=\cup_{y\in B_{rx_1}}\cH_y$ and set $C_1(x)=W_1(x)=0$ for all $x\in\cD_{\bx_1}$.

For all  $x\in B_{x_1}$, set $C_1(x)=1$, $W_1(x)=R(x_1)$.

{\bf Step n+1.} Put $z=r$ and do
$z=\arg\!\max\{\phi(V_n(y),C_n(y)):y \in \cC_z\}$
until $z\in \bx_n\cup\cD_{\bx_n}$, where
$$ V_n(y)=\indiq_{\{t(f(y))=1\}} W_{n}(y)+\indiq_{\{t(f(y))=0\}}(C_{n}(y)-W_{n}(y)).$$

Set $z_n=z$.

(i) If $z_n\in \bx_n$ (this will almost never occur if $n$ is reasonable for a large game), set
$\bx_{n+1}=\bx_n$, $B_{\bx_{n+1}}=B_{\bx_n}$ and $\cD_{\bx_{n+1}}=\cD_{\bx_n}$

For all $x\in B_{rz_{n}}$, set $C_{n+1}(x)=C_n(x)+1$,  
$W_{n+1}(x)=W_n(x)+R({z_{n}})$.

For all $x\in (B_{\bx_{n+1}}\cup \cD_{\bx_{n+1}})\setminus B_{rz_n}$, set 
$C_{n+1}(x)=C_n(x)$, $W_{n+1}(x)=W_n(x)$.

(ii) Else (then $z_n\in\cD_{\bx_n}$), simulate a uniformly random match from $z_n$, call 
$x_{n+1}$ the resulting leave, set $\bx_{n+1}=\bx_n\cup\{x_{n+1}\}$.

During this random match, keep track of $R(x_{n+1})$, of 
$B_{\bx_{n+1}}=B_{\bx_n}\cup B_{rx_{n+1}}$ and of $\cD_{\bx_{n+1}}=(\cD_{\bx_n}\setminus \{z_n\}) \cup \bigcup_{y\in B_{z_nx_{n+1}}
\setminus \{z_n\}}\cH_y$ and set $R_{n+1}(x)=W_{n+1}(x)=0$ for all 
$x \in \bigcup_{y\in B_{z_nx_{n+1}}\setminus \{z_n\}}\cH_y$.

For all $x\in B_{rx_{n+1}}$, set $C_{n+1}(x)=C_n(x)+1$,  $W_{n+1}(y)=W_n(x)+R(x_{n+1})$.

Finally, set $C_{n+1}(x)=C_n(x)$, $W_{n+1}(x)=W_n(x)$ 
for all $x\in (B_{\bx_{n}}\cup \cD_{\bx_{n}})\setminus B_{rx_{n+1}}$.

{\bf Conclusion.} Stop after a given number of iterations $n_0$ 
(or after a given amount of time). As \emph{best child} of $r$, choose
$x_*=\arg\!\max\{\phi(V_{n_0}(x),C_{n_0}(x)) : x$ child of $r\}$.
\end{algo}

Algorithm \ref{mctsinf} updates the information on the whole visited branch at each new rollout. Here is a more 
standard version: it only creates, after each new simulation, one new node (together with its brothers)
and updates the information only on the branch from the root to this new node.
It seems clear that Algorithm \ref{mctsinf} should be better, but it may lead to 
memory problems if the game is very large.
We do not discuss such memory problems in the present paper.

\begin{algo}[MCTS]\label{mcts}
Consider  $\phi:\nn\times\nn\mapsto \rr$, e.g.
$\phi(w,c)=\frac{w+a}{c+b}$ for some $a,b>0$.

{\bf Step 1.} 
Simulate a uniformly random match from $r$, call $u$ the resulting leave.

During this random match, keep track of $R(u)$ and of $T_1=\{r\}\cup \cC_r$.

Set $C_1(x)=W_1(x)=0$ for all $x\in T_1 \setminus B_{ru}$.

Set $C_1(x)=1$ and $W_1(x)=R(u)$ for all $x\in T_1 \cap B_{ru}$.

{\bf Step n+1.} Put $z=r$ and do
$z=\arg\!\max\{\phi(V_n(y),C_n(y)):y \in \cC_z\}$
until $z\in L_{T_n}$, where
$$ V_n(y)=\indiq_{\{t(f(y))=1\}} W_{n}(y)+\indiq_{\{t(f(y))=0\}}(C_{n}(y)-W_{n}(y)).$$

Set $z_n=z$.

(i) If $z_n\in \cL$ (this will never occur if $n$ is reasonable for a large game), set
$T_{n+1}=T_n$.

For all
$x\in B_{rz_n}$, set $C_{n+1}(x)=C_n(x)+1$,  $W_{n+1}(x)=W_n(x)+R(z_n)$.

For all $x\in T_{n+1}\setminus B_{rz_n}$, set $C_{n+1}(x)=C_n(x)$, $W_{n+1}(x)=W_n(x)$.

(ii) If $z_n\notin\cL$, simulate a uniformly random match from $z_n$, call $u$ the resulting leave,
define $y$ as the child of $z_n$ belonging to $B_{z_nu}$ and set $T_{n+1}=T_n\cup\cC_{z_n}$.

For all $x\in B_{rz_{n}}$, set $C_{n+1}(x)=C_n(x)+1$,  $W_{n+1}(y)=W_n(x)+R(u)$.

Set $C_{n+1}(y)=1$,  $W_{n+1}(y)=\indiq_{\{R(u)=1\}}$ and, for all $x\in \cH_{y}$, set $C_{n+1}(x)=W_{n+1}(x)=0$.

For all $x \in T_{n}\setminus B_{rz_n}$, set $C_{n+1}(x)=C_n(x)$ and $W_{n+1}(x)=W_n(x)$.

{\bf Conclusion.} Stop after a given number of iterations $n_0$ 
(or after a given amount of time). As \emph{best child} of $r$, choose
$x_*=\arg\!\max\{\phi(V_{n_0}(x),C_{n_0}(x)) : x$ child of $r\}$.
\end{algo}

Of course, in both algorithms, the choice of the function $\phi$ is debatable.
The choice $\phi(w,c)=(w+a)/(c+b)$, with $a>0$ and $b>0$ chosen empirically,
seems to be a very good choice and was proposed by Lee et al \cite[Subsection II-B-1]{teytetal}.

Algorithm \ref{mctsinf} is not admissible in the sense of 
Definition \ref{proc} because it may take different decisions with the same information. Indeed,
it might visit twice the same leave consecutively:
this does not modify the information but changes the values of $W_n$ and $C_n$.
However, it is \emph{almost} admissible: 
it would suffice to forbid two visits at the same leave
(or alternatively to set $C_{n+1}(x)=C_n(x)$ and $W_{n+1}(x)=W_n(x)$ for all 
$x\in B_{\bx_n}\cup\cD_{\bx_n}$ in the case where $x_{n+1}\in\bx_n$) to make it admissible.
Since such a double visit almost never happens in practice, we decided not to complicate
the definition of admissible algorithms nor to modify Algorithm \ref{mctsinf}.

Algorithm \ref{mcts} is not admissible, because it has not the good structure
(it does not keep track of the whole observed information),
but we see it as an truncated version of Algorithm \ref{mctsinf},
which is itself almost admissible.

Observe that in Algorithm \ref{mctsinf}, $C_n(x)$ is the number of times (iterations)
where the node $x$ has been crossed and $W_n(x)$ is the number of times where
$x$ has been crossed and where this has led to a victory of $J_1$, 
all this after $n$ matches. The $(n+1)$-th match is as follows:
we start from the root and make $J_1$ play the most promising move
(for itself, i.e. the child with the highest $\phi(W_n,C_n)$) and $J_0$ play
the most promising move (for itself, i.e. the child with the highest $\phi(C_n-W_n,C_n)$)
until we reach an uncrossed position $z_n \in \cD_{\bx_n}$. From there,
we end the match at uniform random, until we arrive at some leave $u$.
We finally update the explored tree as well as its boundary and the values
of the numbers of crosses and of victories of each node of the branch from $r$ to $u$.


\begin{thebibliography}{99}

\bibitem{a}{{\sc B. Abramson}, {\it Expected-Outcome: A General Model of Static
Evaluation}, IEEE Trans. Pattern Anal. Mach. Intell. 12 (1990), 182--193.}

\bibitem{akdn}{{\sc T. Ali Khan, L. Devroye, R. Neininger}, {\it A limit law for the root value of minimax trees,}
Electron. Comm. Probab. 10 (2005), 273--281.}

\bibitem{acbf}{{\sc P. Auer, N. Cesa-Bianchi, P. Fischer}, {\it Finite-time Analysis
of the Multiarmed Bandit Problem,} Mach. Learn. 47 (2002), 235--256.}

\bibitem{bs}{{\sc E.B. Baum, W.D. Smith}, {\it A Bayesian approach to relevance in game playing,}
Artificial Intelligence 97 (1997), 195--242.}

\bibitem{bcb}{{\sc S. Bubeck, N. Cesa-Bianchi}, {\it Regret analysis of stochastic and nonstochastic 
multi-armed bandit problems}, Foundations and Trends in Machine Learning 5 (2012), 1--122.}

\bibitem{bmp}{{\sc L. Bu\c soniu, R. Munos, E. P\'all}, {\it An analysis of optimistic, 
best-first search for minimax sequential decision making}, IEEE International Symposium on Approximate 
Dynamic Programming and Reinforcement Learning, 2014.}

\bibitem{mctsurvey}{{\sc C. Browne, E. Powley, D.  Whitehouse, S. Lucas, P.I. Cowling, P. Rohlfshagen,
S. Tavener, D. Perez, S. Samothrakis, S. Colton,} {\it A Survey of Monte Carlo Tree Search Methods},
IEEE transactions on computantional intelligence and AI in games 4 (2012), 1--43.}

\bibitem{cbsetc}{{\sc G.M.J.B. Chaslot, S. Bakkes, I. Szita, P. Spronck},
{\it Monte-Carlo Tree Search: A New Framework for Game AI}, 
in Proc. Artif. Intell. Interact. Digital Entert. Conf., Stanford Univ.,
California, 2008, 216--217.}

\bibitem{cwhetc}{{\sc G.M.J.B. Chaslot, M.H.M. Winands, H.J. van den Herik,
J.W.H.M. Uiterwijk, B. Bouzy}, {\it Progressive Strategies for Monte-Carlo Tree Search,} 
New Math. Nat. Comput. 4 (2008), pp. 343--357.}

\bibitem{cm}{{\sc P.A. Coquelin, R. Munos}, Bandit algorithms for tree search. Technical report, 
INRIA RR-6141, 2007.}

\bibitem{c}{{\sc R. Coulom}, {\it Efficient Selectivity and Backup Operators in
Monte-Carlo Tree Search}, in Proc. 5th Int. Conf. Comput. and
Games, Turin, Italy, 2006, pp. 72--83.}

\bibitem{dk}{{\sc L. Devroye, O. Kamoun}, {\it Random minimax game trees}, 
Random discrete structures (Minneapolis, MN, 1993), 55--80, IMA Vol. Math. Appl., 76, Springer, New York, 1996.}

\bibitem{gkk}{{\sc A. Garivier, E. Kaufmann, W.M. Koolen}, {\it Maximin Action Identification: 
A New Bandit Framework for Games}, JMLR: Workshop and Conference Proceedings vol 49, 1--23, 2016.}

\bibitem{mogo}{{\sc S. Gelly, Y. Wang, R. Munos, O. Teytaud}, {\it Modification of
UCT with Patterns in Monte-Carlo Go,} Inst. Nat. Rech. Inform. Auto. (INRIA), Paris, Tech. Rep., 2006.}

\bibitem{g}{{\sc M.L. Ginsberg,} {\it GIB: Imperfect Information in a Computationally Challenging Game,} 
J. Artif. Intell. Res. 14 (2001), 303--358.}

\bibitem{gk}{{\sc D. Golovin, A. Kraus}, {\it Adaptive submodularity: theory and applications in active 
learning and stochastic optimization}, J. Artificial Intelligence Res. 42 (2011), 427--486.}

\bibitem{ks}{{\sc L. Kocsis, C. Szepesv\'ari}, {\it Bandit based Monte-Carlo planning}, 
Machine learning: ECML 2006, 282--293,
Lecture Notes in Comput. Sci., 4212, Springer, Berlin, 2006.}

\bibitem{teytetal}{{\sc C.S. Lee, M.H. Wang, G.M.J.B. Chaslot, J.B. Hoock, A. Rimmel, O. Teytaud, S.R. Tsai, 
S.C. Hsu, T.P. Hong}, {\it The Computational Intelligence of MoGo Revealed in Taiwan’s
Computer Go Tournaments,} IEEE Trans. Comp. Intell. AI Games 1 (2009), 73--89.}

\bibitem{mlivre}{{\sc R. Munos}, {\it From bandits to Monte-Carlo Tree Search: The optimistic principle 
applied to optimization and planning}, Foundations and Trends in Machine Learning (Book 21),
Now Publishers Inc, 146 pp, 2014.}

\bibitem{p}{{\sc J. Pearl}, {\it Asymptotic properties of minimax trees and game-searching procedures},
Artificial Intelligence 14 (1980), 113--138.}

\bibitem{s}{{\sc B. Sheppard,} {\it World-championship-caliber Scrabble}, Artif.
Intell. 134 (2002), 241–275.}


\bibitem{alphago}{{\sc D. Silver, A. Huang, C.J. Maddison, A. Guez, L. Sifre, G. van den Driessche,
J. Schrittwieser, I. Antonoglou, V. Panneershelvam, M. Lanctot, S. Dieleman, D. Grewe,
J. Nham, N. Kalchbrenner, I. Sutskever, T. Lillicrap, M. Leach, K. Kavukcuoglu,
T. Graepel, D. Hassabis}, {\it Mastering the game of Go with deep
neural networks and tree search}, Nature 529 (2016), 484--489.}

\bibitem{tarsi}{{\sc M. Tarsi}, {\it Optimal search on some game trees,} 
J. Assoc. Comput. Mach. 30 (1983), no. 3, 389--396.}

\bibitem{trs}{{\sc G. Tesauro, V.T. Rajan, R. Segal}, {\it Bayesian Inference in Monte-Carlo Tree Search},
UAI'10 Proceedings of the Twenty-Sixth Conference on Uncertainty in Artificial Intelligence, 2010, 580--588.}


\end{thebibliography}
\end{document}